\documentclass{article}
\usepackage[utf8]{inputenc}
\usepackage{amsmath}
\usepackage{amssymb, amsfonts}
\usepackage{amsthm,mathtools}
\usepackage{tikz}
\usetikzlibrary{decorations.markings}
\usetikzlibrary{shapes.geometric, arrows}
\usepackage{pgfplots}
\usepackage{subcaption}
\usepackage{pythonhighlight}

\pgfdeclarelayer{edgelayer}
\pgfdeclarelayer{nodelayer}
\pgfsetlayers{edgelayer,nodelayer,main}

\tikzstyle{none}=[inner sep=0pt]
\tikzstyle{red}=[fill=red, draw=red, shape=circle,minimum size=1mm]
\tikzstyle{blue}=[fill=cyan, draw=cyan, shape=rectangle, minimum size=2mm]
\tikzstyle{sred}=[fill=red!20, draw=red, shape=circle, minimum size=1pt]
\tikzstyle{sblue}=[fill=cyan, draw=cyan, regular polygon, regular polygon sides=3, rotate=180, minimum size=1pt]
\definecolor{mygreen}{RGB}{4,200 ,7}
\tikzstyle{green}=[-, draw=mygreen]
\tikzstyle{black}=[-, fill={rgb,255: red,251; green,255; blue,5}, fill opacity = 0.4]
\pgfplotsset{compat=1.15}
\usepackage{mathrsfs}
\usepackage{hyperref}
\usepackage{cleveref}
\usepackage{todonotes,mathscinet}

\theoremstyle{plain}
\newtheorem{theorem}{Theorem}[section]
\newtheorem{lemma}[theorem]{Lemma}
\newtheorem{proposition}[theorem]{Proposition}
\newtheorem{corollary}[theorem]{Corollary}

\newtheorem{conjecture}[theorem]{Conjecture}
\newtheorem*{theorem*}{Theorem}

\theoremstyle{definition}
\newtheorem{definition}[theorem]{Definition}
\newtheorem{example}[theorem]{Example}
\newtheorem{remark}[theorem]{Remark}

\theoremstyle{remark}
\newtheorem{claim}[theorem]{Claim}

\renewcommand{\bar}{\overline}
\renewcommand{\tilde}{\widetilde}

\newcount\cols
{\catcode`,=\active\catcode`|=\active
 \gdef\Young#1{\hbox{$\vcenter
 {\mathcode`,="8000\mathcode`|="8000
  \def,{\global\advance\cols by 1 &}%
  \def|{\cr
        \multispan{\the\cols}\hrulefill\cr
        &\global\cols=2 }%
  \offinterlineskip\everycr{}\tabskip=0pt
  \dimen0=\ht\strutbox \advance\dimen0 by \dp\strutbox
  \halign
   {\vrule height \ht\strutbox depth \dp\strutbox##
    &&\hbox to \dimen0{\hss$##$\hss}\vrule\cr
    \noalign{\hrule}&\global\cols=2 #1\crcr
    \multispan{\the\cols}\hrulefill\cr%
   }
 }$}}
\gdef\Skew(#1:#2){\hbox{$\vcenter
{\mathcode`,="8000\mathcode`|="8000
  \dimen0=\ht\strutbox \advance\dimen0 by \dp\strutbox
  \def\boxbeg{\vbox
    \bgroup\hrule\kern-0.4pt\hbox to\dimen0\bgroup\strut\vrule\hss$}%
  \def\boxend{$\hss\egroup\hrule\egroup}%
  \def,{\boxend\boxbeg}%
  \def|##1:{\boxend\vrule\egroup\nointerlineskip\kern-0.4pt
    \moveright##1\dimen0\hbox\bgroup\boxbeg}%
  \def\\##1\\##2:{\boxend\vrule\egroup\nointerlineskip\kern-0.4pt
    \kern ##1\dimen0\moveright##2\dimen0\hbox\bgroup\boxbeg}%
  \moveright#1\dimen0\hbox\bgroup\boxbeg#2\boxend\vrule\egroup
 }$}}
}
\newcommand{\halfcirc}[2]{
    \draw [fill=cyan] (#1,#2) circle (2.5pt); 
    \fill[red] (#1,#2) --++ (180:2.5pt) arc (180:0:2.5pt) -- cycle; 
}

\newcommand{\bbQ}{\mathbb{Q}}
\newcommand{\bbZ}{\mathbb{Z}}
\newcommand{\bbN}{\mathbb{N}}
\newcommand{\bbC}{\mathbb{C}}
\newcommand{\bbR}{\mathbb{R}}
\newcommand{\bbT}{\mathbb{T}}

\newcommand{\raw}{\rightarrow}
\newcommand{\ra}{\rightarrow}

\newcommand{\affphi}{\widehat{\phi}}

\newcommand{\affell}{\widehat{\ell}}

\newcommand{\calA}{\mathcal{A}}
\newcommand{\calB}{\mathcal{B}}
\newcommand{\calP}{\mathcal{P}}

\newcommand{\calN}{\mathcal{N}}
\newcommand{\calD}{\mathcal{D}}
\newcommand{\affD}{\widehat{\calD}}

\newcommand{\Gr}{\mathcal{G}r}
\newcommand{\sch}[1]{\bar{\Gr_{#1}}}

\newcommand{\affW}{\widehat{W}}
\newcommand{\affX}{\widehat{X}}
\newcommand{\affPhi}{\widehat{\Phi}}
\newcommand{\affDelta}{\widehat{\Delta}}
\newcommand{\affeps}{\widehat{\epsilon}}
\newcommand{\affe}{\widehat{e}}
\newcommand{\afff}{\widehat{f}}

\newcommand{\eps}{\epsilon}
\newcommand{\op}{\operatorname}

\newcommand{\undc}{\underline{c}}
\newcommand{\extW}{\widehat{W}_{ext}}
\newcommand{\lam}{\lambda}
\newcommand{\lamk}{\lambda-k\varpi_2}
\newcommand{\lama}{\lambda-a\varpi_2}
\newcommand{\lamA}{\lambda-A\varpi_2}
\newcommand{\bPsi}{\bar{\Psi}}

\newcommand{\tightoverset}[2]{%
  \mathop{#2}\limits^{\vbox to -.2ex{\kern-0.75ex\hbox{$#1$}\vss}}}

\newcommand{\lce}{\left\lceil}
\newcommand{\lfl}{\left\lfloor}
\newcommand{\frf}[1]{\left\lfloor\frac{#1}{2}\right\rfloor}
\newcommand{\frc}[1]{\left\lceil\frac{#1}{2}\right\rceil}

\newcommand{\rce}{\right\rceil}
\newcommand{\rfl}{\right\rfloor}

\newcommand{\calH}{\mathcal{H}}
\newcommand{\bfH}{\mathbf{H}}
\newcommand{\undH}{\underline{\bfH}}
\newcommand{\bfN}{\mathbf{N}}
\newcommand{\tilN}{\tilde{\bfN}}
\newcommand{\htil}{\tilde{h}}
\newcommand{\Phio}{\Phi^{-\bar 1}}

\DeclareMathOperator{\Ima}{Im}
\DeclareMathOperator{\Arr}{Arr}
\DeclareMathOperator{\sgn}{sgn}
\DeclareMathOperator{\wt}{wt}
\DeclareMathOperator{\Conv}{Conv}

\DeclareMathOperator{\pat}{pat}
\DeclareMathOperator{\at}{at}
\DeclareMathOperator{\str}{str}
\DeclareMathOperator{\HL}{HL}
\DeclareMathOperator{\word}{word}
\DeclareMathOperator{\grdim}{grdim}

\makeatletter
\newcommand{\mylabel}[2]{#2\def\@currentlabel{#2}\label{#1}}
\makeatother

\newcommand{\stn}{\str_2}
\title{Atoms and charge in type $C_2$}
\author{Leonardo Patimo, Jacinta Torres}

\begin{document}

\maketitle
\begin{abstract}
     We construct atomic decompositions for crystals of type $C_{2}$ and define a charge statistic on them, thus providing positive combinatorial formulas for Kostka--Foulkes polynomials associated to them together with a natural geometric interpretation.
\end{abstract}
\tableofcontents

\section*{Introduction}
Let $\mathfrak{g}$ be the symplectic Lie algebra $\mathfrak{sp}_4(\bbC)$, i.e. the simple Lie algebra of type $C_2$. 
The irreducible $\mathfrak{g}$-modules are the highest weight modules $V(\lam)$, with $\lambda$ a  dominant weight. Given an arbitrary weight $\mu$, we denote by $d_{\lam,\mu}$ the \emph{weight multiplicity}, i.e. the dimension of the weight space $ V(\lam)_\mu$.  
The weight multiplicity $d_{\lam,\mu}$ admits a $q$-analogue, known as the Kostka--Foulkes polynomial $K_{\lam,\mu}(q)$, so that $K_{\lam,\mu}(1)=d_{\lam,\mu}$. The Kostka--Foulkes polynomials have a natural representation-theoretic interpretation since  their coefficients record the dimension of the graded pieces of the Brilinski--Konstant filtration on weight spaces. Additionally, these polynomials are also (up to renormalization) special cases of affine Kazhdan--Lusztig polynomials and have positive coefficients.

The goal of this paper is to give a combinatorial interpretation for the Kostka--Foulkes polynomials $K_{\lam,\mu}(q)$. 
Finding  such a combinatorial formula amounts to finding:
\begin{enumerate}
\item a set $\calB(\lambda)_\mu$ of cardinality $d_{\lam,\mu}$ parametrizing a basis of the $\mu$-weight space $V(\lambda)_{\mu}$.
\item a combinatorial statistic $c: \calB(\lambda)_\mu \rightarrow \mathbb{Z}_{>0}$, called the \emph{charge}, such that the Kostka--Foulkes polynomial $K_{\lambda,\mu}$ is a generating function of $\op{ch}$ on $\calB(\lambda)_\mu$,
\[K_{\lambda,\mu}(q) = \underset{T \in \calB(\lambda)_\mu}{\sum} q^{c(T)}. \]
\end{enumerate}
\noindent The set $\calB(\lam)_\mu$ has many known realizations, some of which are geometric, such as Littelmann paths, others algebro-geometric, such as Mirković--Vilonen cycles, and some purely combinatorial, such as semi-standard Young tableaux in type A or Kashiwara--Nakashima tableaux for classical types. An important feature that all of these models have in common is that they are endowed with a \textit{crystal structure}, that is, for each of these models the set $\calB(\lambda)=\bigcup \calB(\lam)_\mu$ has cardinality $\op{dim}(V(\lambda))$ and can be endowed with  the structure of a normal crystal (cf. \cite{BG01,BSch17} ).

In type $A$,  the charge statistic was first described by Lascoux and \break Schützenberger in 1978  using a combinatorial procedure on tableaux called cyclage   \cite{lscharge}. In 1995, Lascoux, Leclerc and Thibon \cite{llt95} provided another formulation of the charge statistic in terms of the crystal structure on tableaux.

In a recent work by the first named author \cite{Pat}, an alternative description of the charge statistic was obtained through a more geometric approach, which involves translating the problem of finding the charge onto the affine Grassmannian, where it becomes a variation problem for the hyperbolic localization functor.
 This geometric approach makes the problem of finding a charge statistic more accessible even beyond type $A$ (to this day this remains a mostly  open problem, except that in some special situations, such as row tableaux in type $C$ \cite{DGT} or in weight $0$ \cite{LL20}).  In fact, in the present paper we develop a similar strategy to construct a charge statistic in type $C_2$. We believe that this strategy can be further extended to cover groups of higher ranks. 
 
\subsection*{Charges via the affine Grassmannian}

 We now briefly recall the results in \cite{Pat}, at the heart of which lies the geometric Satake correspondence. Recall that the affine Grassmannian associated to $G^{\vee}$ is endowed with an action of the extended torus $ \hat{T} = T^{\vee} \times \mathbb{C}^{*}$. For $\lam\in X_+$ let $\bar{\Gr^\lam}$ denote the corresponding Schubert variety in the affine Grassmannian of $G^\vee$ (cf. \cite[\S 2.1.2.]{Pat}). For any regular $\eta \in \affX$ and any $\mu \leq \lam$ the hyperbolic localization induces a functor 
\[ \HL^\eta_\mu: \calD^b_{T^\vee \times \bbC^*}(\bar{\Gr^\lam})\raw \calD^b(pt)\cong \mathrm{Vect}^\bbZ,\]
where $\calD^b_{T^\vee \times \bbC^*}(\bar{\mathcal{G}r}^\lam)$ is the derived category of $T^\vee \times \bbC^*$-equivariant constructible sheaves on the Schubert variety $\bar{\Gr^\lam}$ with $\bbQ$-coefficients, and $\calD^b(pt)$ is the derived category of sheaves on a point, which is equivalent to the category of graded $\bbQ$-vector spaces (see \cite[\S 2.4]{Pat}).
In general, for any regular $\eta\in \affX_\bbQ$ we can define $\HL^\eta_\mu$ as $\HL^{N\eta}_\mu$, where $N$ is any positive integer such that $N\eta\in \affX$. By abuse of terminology, we are then allowed to refer to all the elements in $\affX_\bbQ$ as cocharacters.

 If $\eta$ is a dominant $T^{\vee}$ cocharacter, then the hyperbolic localization functors are the \textit{weight functors}, which send the intersection cohomology sheaf $IC_{\lambda}$ to the weight space $V(\lambda)_{\mu}$ of the simple module $V(\lambda)$. In this case, as in \cite{Pat},  we say that $\eta$ is in the \textit{MV region}, where $MV$ is short for Mirkovi\'c--Vilonen. 
 If $\eta$ is $\hat{T}$-dominant, that is, dominant for the affine root system, then the hyperbolic localization functors return graded vector spaces whose graded dimensions are renormalized Kostka--Foulkes polynomials. In this case, we say that $\eta$ is in the \textit{KL region}, where $KL$ is short for Kazhdan--Lusztig. 

Let $\htil^\eta_{\mu,\lam}(v) := \grdim(\HL^\eta_\mu(IC_\lam))$. The polynomials $\htil^\eta_{\mu,\lam}(v)$ are called \emph{renormalized $\eta$-Kazhdan--Lusztig polynomials}. We say that a function $r_{\eta}:\calB(\lam)\raw \bbZ$ is a $\eta$-\emph{recharge} for $\eta$ if we have
\[\htil^\eta_{\mu,\lam}(q^{\frac12})=\sum_{T\in \calB(\lam)_\mu} q^{r_{\eta}(T)}\in \bbZ[q^{\frac12},q^{-\frac12}].\]
If $\eta_{KL}$ is in the KL chamber and $\mu \in X_+$, then
\[K_{\mu,\lam}(q)=\htil^{\eta_{KL}}_{\mu,\lam}(q^{\frac12})q^{\frac12 \ell(\mu)}\] is a Koskta--Foulkes polynomial by \cite[Proposition 2.14]{Pat}. So if $r_{KL}$ is a recharge for $\eta_{KL}$  in the KL region, we obtain a charge statistic $c:\calB(\lam)\raw \bbZ$ by setting $c(T):=r_{KL}(T)+\frac 12 \ell(\wt(T))$. 
Notice that if $\wt(T)\in X_+$ this is equal to $c(T)=r_{KL}(T)+\langle \wt(T),\rho^\vee\rangle$.\\

 It turns out that the only situation where the hyperbolic localization functors change their value is when they ``cross'' a hyperplane of the form 
 
 \begin{align*}
 H_{\alpha^{\vee}} = \left\{\eta \in X_{*}( \hat{T}) | \langle \eta,\alpha^{\vee}\rangle = 0 \right\}
 \end{align*}
 
 \noindent
 where $\alpha^{\vee}$ is a positive real root in the root system corresponding to the Langlands dual group $G^{\vee}$. In \cite{Pat}, the first named author has observed a simple rule to compute the hyperbolic localization functor after crossing such a wall. Let $\eta_{1}$ and $\eta_{2}$ be two cocharacters on opposite sides of such a wall $H_{\alpha^{\vee}}$. Then by \cite[Proposition 2.33]{Pat} we have, for $\nu = s_{\alpha^{\vee}}(\mu)$ such that $\mu < \nu \leq \lambda$:

 \begin{align*}
\htil^{\eta_2}_{\nu,\lam}(v) &= v^{-2}\htil^{\eta_1}_{\nu,\lam}(v) \hbox{ and }\\
\htil^{\eta_2}_{\mu,\lam}(v) &= \htil^{\eta_1}_{\mu,\lam}(v) + (1 - v^{-2})\htil^{\eta_1}_{\nu,\lam}(v).
 \end{align*}

 In order to efficiently track these changes, \textit{swapping functions} $\phi : \calB(\lambda)_{\mu} \rightarrow \calB(\lambda)_{s_{\alpha^{\vee}}(\mu)}$ are constructed with the property that $r_{\eta_1}(T) - 1 = r_{\eta_1}(\phi (T))$. This allows a definition of a recharge statistic $r_{\eta_2}$ for $\eta_2$ given a recharge statistic $r_{\eta_1}$ for $\eta_1$. To do this, \textit{modified root operators} $e_{\alpha}, f_{\alpha}$ are constructed for any positive root $\alpha \in \Phi$. Combining this with the \textit{atomic decomposition} of the crystals $B(\lambda)$ in type $A_{n-1}$ given by Lecouvey--Lenart \cite{LL21}, which are obtained independently in \cite{Pat}, it is shown that the charge statistic giving the Kostka--Foulkes polynomials in type $A_{n-1}$ is given by the sum $\sum_{\alpha \in \Phi^{+}} \epsilon_{\alpha}(b)$.


 \subsection*{Results}

 Our main results consist of the atomic decomposition of the type $C_2$ crystals $\mathcal{B}(\lambda)$, as well as the construction of swapping functions. As a result we obtain the following formula for the  charge statistic in type $C_2$
 \begin{align*}
     c: \calB(\lam)_+ &\raw \bbN
     T\mapsto \eps_1(T)+\eps_2(T)+\eps_{12}(T)+\affeps_{21}(T)
 \end{align*}
 where $\affeps_{21}$ is not attached to a modified crystal operator, but rather depends on the atom in which $T$ sits.
This yields a positive combinatorial formula for the Kostka-- Foulkes polynomials. We outline our methodology below. 

 \subsubsection*{Atomic decompositions and charge statistics}
In \cite{Pat}, the first named author has shown that the LL atoms \cite{LL21} coincide with the connected components of the graph with same vertices as $\calB(\lambda)$, given by the $W$-closure of the $f_n$-orbits. This is one of the first constraints which appears when considering type $C_2$ crystals: the $W, f_2$ connected components are not atoms (cf. \Cref{def:atom}). This calls for an alternative approach. As in \cite{Pat}, the language of adapted strings will be an important tool for us. We first define an embedding of crystals (cf. \Cref{emb})

\[\Phi: \calB(\lambda) \rightarrow \calB(\lambda + 2\varpi_1).\]

We call the complement of $\Phi$ in $\calB(\lambda + 2\varpi_1)$ the \textit{principal preatom} $\calP(\lambda+2\varpi_1)$. If $\lambda = \lambda_1 \varpi_1 + \lambda_2 \varpi_2$ is such that $\lambda_1 \leq 1$, we define $\calP(\lambda): = \calB(\lambda)$.
The map $\Phi$ has an easy definition using the combinatorics of Kashiwara--Nakashima tableaux which allows to prove its properties directly, however, its reformulation in terms of adapted strings allows us to give equations describing the principal preatoms $\calP(\lambda)$, which we use throughout this work. A \textit{preatomic decomposition} of our crystal can be defined recursively. We show that the preatoms are stable under the $W$ and $f_2$ action, hence naturally generalize the LL atoms. Once the preatomic decomposition of our crystal has been defined, we are ready to define its atomic decomposition. In \Cref{atomemb} we show that there exists a weight-preserving injection

\[ \bPsi: \calP(\lambda) \rightarrow \calP(\lambda + \varpi_2) \]

\noindent
such that the sets $\calA(\lambda)$ given by the complement of $\bPsi$ for $\lambda_1 \neq 0$, respectively by the complement of $\bPsi^2$ for $\lambda_1 = 0$, are atoms. The map $\bPsi$ is defined explicitly on the string parameters for the reduced expression of the longest Weyl group element given by $s_2s_1s_2s_1$. An explicit description in terms of Kashiwara--Nakashima tableaux is provided in the appendix, although we do not need tableaux combinatorics in this paper. To show that the sets $\calA(\lambda)$ are  atoms, we resort to algebraic computations directly in the Hecke algebra. In particular, we make use of pre-canonical bases, introduced by Libedinsky--Patimo--Plaza in \cite{LPP}. In analogy to the Satake isomorphism, which in particular identifies the ungraded character of the character of $\calB(\lambda)$ with the specialization at $v =1$ of the corresponding element of the Kazhdan--Lusztig basis of the spherical Hecke algebra, in \Cref{PreatomPrecan} it is shown that the ungraded character corresponds to the specialization at $v =1$ of a modification $\tilN^{3}$ of the precanonical basis $\bfN^3$ introduced in (\ref{modifiedprecan3}).\\

In fact, the atomic and preatomic decompositions alone are already enough to define our charge statistic in type $C_2$. Let $T \in \calB(\lambda)$. We define in \Cref{preatomicnumber,defatomicnumber} the \emph{atomic number} $\at(T)$ and the \emph{preatomic number} $\pat(T)$ to be the positive integers such that 

\[T\in \calA(\lam-\at(T)\varpi_2-2\pat(T) \varpi_1)\subset \calP(\lam-2\varpi_1(T))\subset \calB(\lam).\]

A consequence of our main result reads as follows (cf. \Cref{maincharge}).

\begin{theorem*}
The function
$c:\calB(\lam)_+\raw \bbZ$ defined as
\[c(T)=\langle \lambda -\wt(T),\rho^\vee\rangle -\at(T)-\pat(T)\]
is a charge statistic.
\end{theorem*}

Indeed, our main result \Cref{maintheorem} consists in the construction of a recharge statistic $r_{\eta_i}$ for each $\eta_i$ in a family of cocharacters defined in \Cref{family} which goes between the KL and MV regions. In order to construct such recharge statistics, we need first to carefully study the geometry of atoms in type $C_2$.

\subsubsection*{Twisted Bruhat graphs and non-swappable staircases}
We consider \textit{twisted Bruhat graphs} associated to a fixed infinite reduced decomposition $y_{\infty}$ in the affine Weyl group, as in \cite{Pat}. For any $m \in \mathbb{Z}_{>0}$, let $y_m$ be the product of the first $m$ elements of $y_{\infty}$ and let $N(y_m)$ be its set of inversions. The idea is to start off by considering the Bruhat graph $\Gamma_\lambda$ of a given dominant integral weight $\lambda$, that is, the moment graph of the Schubert variety $\sch{\lambda}$. The vertices of the graph $\Gamma_\lambda$ are all the weights lesser than or equal to  $\lam$ in the dominance order. We have an edge $\mu_1\raw \mu_2$ in $\Gamma_\lambda$ if and only if $\mu_2-\mu_1$ is a multiple of a root and $\mu_1\leq \mu_2$. From $\Gamma_{\lambda}$ we obtain our twisted Bruhat graph $\Gamma^{m}_{\lambda}$ by inverting the orientation of all the arrows in $\Gamma_\lambda$ with label in $N(y_m)$. For $\mu \leq \lambda$, let $\Arr_m(\mu,\lam)$ be the set of arrows pointing to $\mu$ in $\Gamma_m^\lambda$ and by $\ell_m(\mu,\lambda):=|\Arr_m(\mu,\lam)|$ the number of those arrows (cf. \Cref{twistedarrowslambdamu}). Let $t_{m+1}:=y_{m+1}y_m^{-1}$. If $\mu < t_{m+1}\mu$ then surprisingly, for  the twisted Bruhat graphs in type $A$ (\cite[Prop. 4.14]{Pat}) the following holds: $\ell_m(\mu,\lambda)=\ell_m(t_{m+1}\mu,\lambda)-1$ if $\mu <t_{m+1}\mu \leq \lambda$. This implies that
$\ell_{m+1}(\mu,\lambda)=\ell_m(t_{m+1}\mu,\lambda)$. However, as we show in \Cref{exampleswap}, this property does not hold in type $C_2$. In \Cref{def:swappableedge} we define an edge $\mu \raw t_{m+1}\mu$ in $\Gamma_{\lambda}$ to be \emph{swappable} if and only if 
\[
 \ell_m(\mu,\lambda)= \ell_m(t_{m+1}\mu,\lambda)-1.
\]

The whole of Section 4 is dedicated to their classification. We  pay particular attention to non-swappable edges. In \Cref{nonswappablenumber} we define the number of non-swappable edges in the following sense:
\[\calN_m(\mu,\lambda):=|\{k \leq m \mid  \mu<t_k\mu \leq \lambda \text{ and }\mu\raw t_k\mu\text{ is not swappable}\}|.\]

An important property of non-swappable edges is that they will always ``be swappable'' in an atom isomorphic to $\mathcal{A}(\lambda - k\varpi_2)$ for large enough $k$. This leads to the notion of non-swappable staircases (cf. \Cref{def:nonswappable}). Essentially, a non-swappable staircase  over $(\mu,\lambda)$ consists of a sequence of edges of the form $e_i:=(\mu \raw \mu - (n+i)\alpha)$ such that $e_i$ is non-swappable in $\calA(\lambda + i \varpi_2)$. We define  $\affD_m(\mu,\lam)$ to be the length of the longest NS-staircase over $(\mu,\lam)$ where the label of every edge in $e_i$ is a root in $N(y_m)$. Moreover, in \Cref{truncatedns} we define the following statistic, which considers only NS-staircases lying in a single preatom:
\[ \calD_m(\mu,\lam,k):=\min(k,\affD_m(\mu,\lam-k\varpi_2)).\]

We are now ready to define the recharge statistics $r_{\eta_m}$, which we define in \Cref{N=0}. For $T\in  \calP(\lam)\subset \calB(\lam')$ with $\mu:=\wt(T)$.  We define
\[ r_m(T):= -\ell_m(\mu,\lama)+\calN_m(\mu,\lama)-\calD_m(\mu,\lambda,a)-\at(T)-2\pat(T)+\langle \lam',\rho^\vee\rangle.\]
Our main result, from which descends our explicit formula for the charge statistic in type $C_2$, is the following (cf. \Cref{maintheorem}).

\begin{theorem*}
The function $r_m:\calB(\lambda)\raw \bbZ$ is a recharge statistic for $\eta_m$ for any $m\in \bbN\cup \{\infty\}$.
\end{theorem*}

To prove our main theorem, we need to construct swapping functions. 

\subsubsection*{Swapping functions}

The existence of non-swappable edges in type $C_2$ means that we cannot define swapping functions within a single atom as in type $A_n$. In 
 \Cref{sec:swapping} the swapping functions we construct involve two elements from two different atoms within the same preatom. In order to determine which are the two atoms involved we need to introduce a new quantity, which we call the elevation $\Omega(e)$ of an edge $e$ that measures the height of the maximal staircases of non-swappable edges lying underneath it. For any $\mu\in X$ such that $\mu<t\mu\leq \lambda$ we define the swapping functions 
\[ \psi_{t\mu}:\calB(\lambda)_{t\mu}\ra 
\calB(\lambda)_{\mu}\]
as follows. Let $T\in \calB(\lam)_{t\mu}$ and assume that $T\in \calA(\lam-a\varpi_2)\subset \calP(\lam)$. Let $e:=(\mu \raw t\mu)\in E(\lambda-a\varpi_2)$. Then $\psi_{t\mu}(T)=T'$, where $T'$ is the only element of weight $\mu$ in $\calA(\lam-(a+\Omega(e))\varpi_2)\subset \calP(\lam)$. To prove \Cref{maintheorem} we show in \Cref{swapcheck} that 
\[r_{m+1}(T)=r_{m+1}(\psi_{t\mu}(T))+1.\]

In the proof we use many results on non-swappable staircases and non-swappable edges obtained in Section 4. 

\subsubsection*{Alternative formula}
In Section 6, we obtain an alternative formula for the charge statistic by focusing on a single element and counting how many times its recharge gets changed by a swapping function. The formula we obtain is in terms of the modified crystal operators, which we define in \Cref{def:modified}. 

Let $T\in \calA(\zeta)$ be such that  $\wt(T)=\mu$.
Let $\affeps_{21}(T)$ be the maximum integer such that 
$\mu+k\alpha_i\leq \zeta$. In Section 6 we show that 

\[ c(T)=\eps_1(T)+\eps_2(T)+\eps_{12}(T)+\affeps_{21}(T)\]
is a charge statistic on $\calB_+(\lam)$. Finally, we conjecture a formula for a charge statistic in type $C_3$, which is a natural generalization of our formula.  

\section*{Acknowledgements}
J.T. was supported by the grant UMO-2021/43/D/ST1/02290
 and partially supported by the grant UMO-2019/34/A/ST1/00263. 

\section{The root system and Hecke algebra of type \texorpdfstring{$C_2$}{C2}}
\subsection{The root system and the affine Weyl group}

Let $(X,\Phi,X^\vee,\Phi^\vee)$ be the root datum of the reductive group $Sp(4,\bbC)$. The lattices $X$ and $X^\vee$ are isomorphic to $\bbZ^2$, with bases $\{\varpi_1,\varpi_2\}$ and $\{\varpi_1^\vee,\varpi_2^\vee\}$. Let $X_+$ and $X_+^\vee$ be the subsets of dominant weights and dominant coweights. 
Sometimes we use the notation $(\lambda_1,\lambda_2)$ to the denote the weight $\lambda=\lambda_1\varpi_1 +\lambda_2\varpi_2$.

The root system $\Phi\subset X$ is a root system of type $C_2$, with positive roots \[\{\alpha_1,\alpha_2,\alpha_{12}:=2\alpha_1+\alpha_2,\alpha_{21}:=\alpha_1+\alpha_2\}\] with $\alpha_2$ and $\alpha_{12}$ being the long roots. We have $\alpha_1=2\varpi_1-\varpi_2$ and $\alpha_2=-2\varpi_1+2\varpi_2$.

The coroot system $\Phi^\vee\subset X^\vee$ has positive coroots \[\{\alpha_1^\vee, \alpha_2^\vee,\alpha_{12}^\vee:=\alpha^\vee_1+\alpha^\vee_2,\alpha_{21}^\vee:=\alpha^\vee_1+2\alpha^\vee_2\}.\]
For any $i\in \{1,2,12,21\}$, $\alpha_i^\vee$ is the coroot corresponding to $\alpha_i$.

\begin{figure}[hbt!]
\begin{center}
\begin{tikzpicture}
\draw[->] (0,0) to (1,1) node[above] {$\alpha_{21}$};
\draw[->] (0,0) to (2,0) node[above] {$\alpha_{2}$};
\draw[->] (0,0) to (0,2) node[above] {$\alpha_{12}$};
\draw[->] (0,0) to (0,-2);
\draw[->] (0,0) to (-2,0);
\draw[->] (0,0) to (-1,-1);
\draw[->] (0,0) to (1,-1);
\draw[->] (0,0) to (-1,1) node[above] {$\alpha_{1}$};
\begin{scope}[scale=1.3, xshift=4cm]
\draw[->] (0,0) to (1,1) node[above] {$\alpha^\vee_{21}$};
\draw[->] (0,0) to (1,0) node[above] {$\alpha^\vee_{2}$};
\draw[->] (0,0) to (0,1) node[above] {$\alpha^\vee_{12}$};
\draw[->] (0,0) to (0,-1);
\draw[->] (0,0) to (-1,0);
\draw[->] (0,0) to (-1,-1);
\draw[->] (0,0) to (1,-1);
\draw[->] (0,0) to (-1,1) node[above] {$\alpha^\vee_{1}$};
\end{scope}
\end{tikzpicture}
\end{center}
\caption{The root system $\Phi$ and the coroot system $\Phi^\vee$.}\label{figrootsystems}
\end{figure}
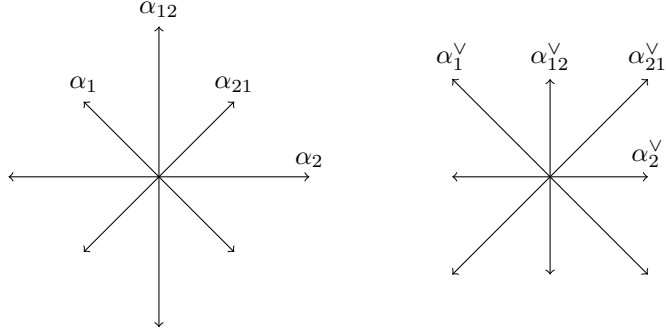

Let $\rho$ be the half-sum of the positive roots and $\rho^\vee$ be the half-sum of negative roots. We have $\rho=2\alpha_1 +\frac32 \alpha_2$ and $\rho^\vee= \frac32 \alpha_1^\vee + 2\alpha_2^\vee$.

We have $X/\bbZ \Phi\cong \bbZ/2\bbZ$ and the two classes are generated by $0$ and $\varpi_1$.

Let $\affW$ be the affine Weyl group of type $\tilde{C}_2$. The group $\affW$ has three simple reflections $s_0, s_1, s_2$ and has the following description as a Coxeter group:
\[ \affW \cong \langle s_0,s_1,s_2\mid s_0^2=s_1^2=s_2^2=(s_0s_2)^4=(s_1s_2)^4=(s_0s_1)^2=e\rangle .\]

Let $ \affX^\vee:=X^\vee\oplus \bbZ$ and let $\affPhi^\vee= \{ \alpha^\vee + m\delta \mid \alpha^\vee \in \Phi^\vee, m \in \bbZ\}$ be the corresponding affine root system. The positive roots in $\affPhi^\vee$ are \[\affPhi^\vee_{+} =\{ \alpha^\vee + m\delta \mid \alpha^\vee \in \Phi^\vee, m >0\} \cup \Phi^\vee_+\] and the simple roots are
 \[\affDelta^\vee = \{\alpha_1^\vee,\alpha_2^\vee,\alpha_0^\vee:=\delta -\alpha_{21}^\vee\}\]
There is a bijection between reflections in $\affW$ and positive roots $\affPhi^\vee_+$, with simple reflections corresponding to simple roots. For a reflection $t\in \affW$ we denote by $\alpha_t^\vee$ the corresponding root.

\subsection{The Hecke algebra and its pre-canonical bases}\label{sec:precan}

Recall from \cite{Knop, Lus83} the definition of the spherical Hecke algebra (see also \cite[\S 2.2]{LPP}).
We denote by $\calH$ the spherical Hecke algebra associated to the root system $\Phi$.
The Hecke algebra is the free module over $\bbZ[v,v^{-1}]$ with standard basis $\{\bfH_\lam\}_{\lam\in X_+}$ and a canonical basis, the Kazhdan-Lusztig basis, which we denote by $\{\undH_\lambda\}_{\lambda\in X_+}$.

The spherical Hecke algebra can be thought of as a deformation of the monoid algebra $\bbZ[X_+]$, which as an abelian group is free with basis $\{e^\lam\}_{\lam\in X_+}$. In fact, specializing at $v=1$, we obtain a ring homomorphism 
\begin{align*}(-)_{v=1}:\calH &\raw \bbZ[X_+]\\
\bfH_\lam& \mapsto e^\lam.
\end{align*}
If $\lam=\lam_1\varpi_1+\lam_2\varpi_2$ we write $\undH_{(\lam_1,\lam_2)}$ for $\undH_{\lam}$ and similarly for $\bfH$.

For $w\in W$ and $\lam\in X$ we denote by $w\cdot \lam = w(\lam+\rho)-\rho$ the dot action of $w$ on $\lam$.
We say that a weight $\lambda$ is singular if there exists $w\in W$ with $w(\lambda)=\lambda$. Clearly, a weight $\lambda$ is singular if and only if $\lambda+\rho$ is singular with respect to the dot action.

We extend to definition of $\undH_\lam$ to the whole $X$ by setting $\undH_\lam=0$ if $(\lam+\rho)$ is singular and $\undH_\lam=(-1)^{\ell(w)}\undH_{w\cdot \lam)}$ if $w\in W$ is such that $w\cdot \lam\in X_+$. Notice that in our setting $\lam=(\lam_1,\lam_2)$ is singular if and only if $\lam_1=-1$, $\lam_2=-1$, $\lam_1+\lam_2=-2$ or $\lam_1+2\lam_2=-3$.

Recall the definition of the pre-canonical bases. We have
\[ \bfN_\lam^i=\sum_{I \subset \Phi^{\geq i}} (-v^2)^{|I|}\undH_{\lam-\sum_{\alpha\in I}\alpha}\]
where $\Phi^{\geq i}$ is the subset of roots of height at least $i$. Notice that we have $\Phi^{\geq 3}=\alpha_{12}=2\varpi_1$ and $\Phi^{\geq 2}=\{\alpha_{12},\alpha_{21}\}=\{2\varpi_1,\varpi_2\}$. 
Recall by \cite[Theorem 1.2]{LPP} that $\bfN^1$ is the standard basis, while $\bfN^2$ is the atomic basis $\bfN$, that is we have
\[ \bfN^2_\lam=\bfN_\lam=\sum_{\mu \leq \lam} v^{2\langle\rho^\vee,\lam-\mu\rangle}\bfH_{\mu}.\]

It follows immediately from the definition that $\undH_\lam=\bfN^4_\lam$. 

\begin{example}
 Unfortunately, and contrary to the type $A$ situation, the coefficients of the $\undH$-basis in the $\bfN^3$-basis are in general not postive. For example, we have $\bfN^3_{(0,\lam_2)}=\undH_{(0,\lam_2)}+v^2\undH_{(0,\lam_2-1)}$. In particular, we get $\undH_{(0,1)} =\bfN^3_{(0,1)}-v^2\bfN^3_{(0,0)}$.
\end{example}

To recover positivity, we need to introduce a modification of the $\bfN^3$ basis. We define 
\begin{align}
\label{modifiedprecan3}
 \tilN^{3}_\lam=\begin{cases}\bfN^3_\lam & \text{if }\lam_1\neq 0\\
\undH_\lam & \text{if }\lam_1 =0
\end{cases}
\end{align}

\begin{lemma}
We have \[\undH_{(\lam_1,\lam_2)}=\sum_{i \leq \frf{\lam_1}} v^{2i}\tilN^3_{(\lam_1-2i,\lam_2)}\]
\end{lemma}
\begin{proof}
We prove it by induction on $\lam_1$.
The claim is clear if $\lam_1=0$.

If $\lam_1>0$, we have $\tilN^3_\lam=\undH_{\lam}-v^2\undH_{\lam-2\varpi_1}$.
If $\lam_1=1$ we have $\tilN^3_\lam=\undH_\lam$ since $\lam-2\varpi_1+\rho$ is singular. If $\lam_1\geq 2$, we get $\undH_{\lam}=\bfN^3_{\lam}+v^2\undH_{\lam-2\varpi_1}$ and the claim easily follows by induction.
\end{proof}

\begin{lemma}\label{precanN3}
We have \[\tilN^3_{(\lam_1,\lam_2)}=\begin{cases} \sum_{i \leq \lam_2} v^{2i}\bfN^2_{(\lam_1,\lam_2-i)} & \text{if }\lam_1>0\\
\sum_{i \leq \frf{\lam_2}} v^{4i}\bfN^2_{(\lam_1,\lam_2-2i)}&\text{if }\lam_1=0.\end{cases}\]
\end{lemma}
\begin{proof}
We have $\bfN^2_\lam=\bfN^3_\lam-v^2\bfN^3_{\lam-\varpi_2}$. If $\lam_1>0$ we get $\tilN^3_\lam=\bfN^3_\lam=\bfN^2_\lam+v^2\tilN^3_{\lam-\varpi_2}$ and the claim easily follows by induction on $\lam_2$.

If $\lam_1=0$ we have $\tilN^3_\lam=\tilN^3_{(0,\lam_2)}=\undH^3_{(0,\lam_2)}$ and 
\begin{align*}
 \bfN^2_{(0,\lam_2)}&=\undH_{(0,\lam_2)}-v^2\undH_{(-2,\lam_2)}-v^2\undH_{(0,\lam_2-1)}+v^4\undH_{(-2,\lam_2-1)} \\
 &=\undH_{(0,\lam_2)}+v^2\undH_{(0,\lam_2-1)}-v^2\undH_{(0,\lam_2-1)}-v^4\undH_{(0,\lam_2-2)}\\
 &=\undH_{(0,\lam_2)}-v^4\undH_{(0,\lam_2-2)}=\tilN_{(0,\lam_2)}-v^4\tilN_{(0,\lam_2-2)}
\end{align*}

If $\lam_2 \leq 1$ we get $\bfN^2_{(0,\lam_2)}=\tilN_{(0,\lam_2)}$. For $\lam_2\geq 2$ we have $\tilN^3_{(0,\lam_2)}=\bfN^2_{(0,\lam_2)}+v^3\tilN^3_{(0,\lam_2-2)}$ and the claim follows by induction.
\end{proof}

\begin{remark}
The decomposition of the $\undH$-basis in terms of the $\bfN$ basis has been already computed in \cite[Theorem 1.1]{BBP}  with different methods. 
Here we prefer to reprove it using the precanonical bases since, as it turns out, also the $\tilN^3$ basis has a natural combinatorial interpretation in terms of the crystal (cf. \Cref{PreatomPrecan}).
\end{remark}

\section{Crystals and Weyl group actions}

A (seminormal) crystal for a complex finite dimensional Lie algebra $\mathfrak{g}$ consists of a non-empty set $B$ together with maps 
\begin{align*}
\op{wt}:& B \longrightarrow X \\
e_{i},f_{i}:& B \longrightarrow B \sqcup \left\{ 0\right\}, i \in [1, \operatorname{rank}(\mathfrak{g})]
\end{align*}

\noindent such that for all $b,b' \in B$:

\begin{itemize}	
\item $b' = e_i(b)$ if and only if $b = f_i(b')$,
\item if $f_i(b) \neq 0 $ then $\textsf{wt}(f_i(b)) = \textsf{wt}(b)-\alpha_i$;
\item
if $e_i(b) \neq 0$, then
$\textsf{wt}(e_i(b)) = \textsf{wt}(b)+\alpha_i$, and
\item  $\phi_i(b)-\eps_i(b)=  \langle \textsf{wt}(b),\alpha_i^\vee  \rangle$,
\end{itemize}

\noindent where 
\
\begin{align*}
\eps_i(b)&=\max\{a \in \mathbb{Z}_{\geq 0} :e_i^a(b)\neq 0\} \hbox{ and } \\
       \phi_i(b)&=\max\{a \in \mathbb{Z}_{\geq 0 }:f_i^a(b)\neq 0\}.
       \end{align*}

\noindent 
To each such crystal $B$ is associated a \textit{crystal graph}, a coloured directed graph with vertex set $B$ and edges coloured by elements $i \in [1,\operatorname{rank}(\mathfrak{g})]$, where if $f_{i}(b) = b'$ there is an arrow $b \overset{i}{\rightarrow} b'$. A crystal is irreducible if its corresponding crystal graph is connected and finite. A seminormal crystal is called normal if it is isomorphic to the crystal of a representation of $\mathfrak{g}$. 
Irreducible normal crystals are thus indexed by dominant integral weights of $\mathfrak{g}$. 
We refer the reader to \cite{BSch17} for more background on crystals.

For  a dominant weight  $\lambda$ we denote by $\mathcal{B}(\lambda)$ the corresponding normal crystal associated to the irreducible representation of $\mathfrak{g}$ 
of highest weight $\lambda$. 

\subsection{Crystals of Kashiwara--Nakashima tableaux}

In type $C$ we can realize crystals using Kashiwara--Nakashima tableaux.

\begin{definition}
Let $n$ be a positive integer. A Kashiwara--Nakashima tableau (KN tableau for short) is a semi-standard Young tableau of shape a partition of at most $n$ parts, in the alphabet 
\[\mathcal{P}_{n} := \left\{1 < \cdots < n < \bar n < \cdots < \bar 1 \right\}\]

\noindent which satisfy the following conditions:
\begin{itemize}
 \item Each one of their columns is \textbf{admissible} (cf. \Cref{defAdm}).
 \item Their \textbf{splitting} is a semi-standard Young tableau (cf. \Cref{defSpl}). 
\end{itemize}
\end{definition}

\begin{definition}\label{defAdm}
Let $C$ be a semi-standard column in the alphabet $\mathcal{P}_{n}$ of length at most $n$. Let $Z = \left\{z_{1} > ... > z_{m} \right\}$ be the set of non-barred letters $z$ in $\mathcal{P}_{n}$ such that both $z$ and $\bar z$ both appear in $C$. We say that the column $C$ is \emph{admissible} if there exists a set $T = \left\{t_{1} > ... > t_{m} \right\}$ of non-barred letters that satisfy:
\begin{itemize}
 \item $t_{1} < z_{1}$ and is maximal with the property $t_1, \bar t_1 \notin C$; 
 \item $t_{i} < \min(t_{i-1}, z_{i})$, $t_i, \bar t_i \notin C$ and is maximal with these properties.
\end{itemize}
\end{definition}

\begin{definition}
\label{defSpl}
The split of a column is the two-column tableau $lC rC$ where $lC$ is the column obtained from $C$ by replacing $ z_{i}$ by $ t_{i}$ and possibly re-ordering, and $rC$ is obtained from $C$ by replacing $ \bar z_{i}$ by $\bar t_{i}$ and possibly re-ordering.  

The \textit{splitting} of a semi-standard Young tableau consisting of admissible columns is the concatenation of the splits of its columns.
\end{definition}

\begin{example}
Let $n = 2$. The column $\Skew(0:\hbox{\tiny{$2$}}|0: \hbox{\tiny{$\bar 2$}})$ is admissible (we have $Z=\{2\}$ and $T=\{1\}$), however, $\Skew(0:\hbox{\tiny{$1$}}|0: \hbox{\tiny{$\bar 1$}})$ is not. Notice that although each one of its columns is admissible, the tableau $\Skew(0:\hbox{\tiny{$2$}},\hbox{\tiny{$2$}}|0: \hbox{\tiny{$\bar 2$}}, \hbox{\tiny{$\bar 2$}} )$ is not KN, because its split, $\Skew(0:\hbox{\tiny{$1$}},\hbox{\tiny{$2$}}, \hbox{\tiny{$1$}},\hbox{\tiny{$2$}} |0: \hbox{\tiny{$\bar 2$}}, \hbox{\tiny{$\bar 1$}}, \hbox{\tiny{$\bar 2$}}, \hbox{\tiny{$\bar 1$}})$ is not semi-standard. 
\end{example}

\begin{definition}
\label{def:weightofatableau}
Let $T$ be a KN tableau. For $i \in \left\{1,2\right\}$ let $n_i(T)$ denote the number of $i$'s appearing in $T$ and let $n_{\bar i}(T)$ denote the number of $\bar i$'s. Let $t_{i}(T) = n_{i}(T) - n_{\bar i}(T)$. Let $\lambda_{1}(T) = t_{1}(T) - t_{2}(T)$ and $\lambda_{2}(T) = t_{2}(T)$. The weight of $T$ is defined to be $\wt(T) = (\lambda_{1}(T),\lambda_2(T)) = \lambda_1 (T)\varpi_1 + \lambda_2(T) \varpi_2$.
\end{definition}

\subsection{Words and signatures. Crystal operators and Weyl group action.}
The \textit{word} of a KN tableau $T$ is the reading of its entries, column by column, starting from the right most column and reading each column from top to bottom. We will denote the word of $T$ by $word(T)$. For example, if 

\begin{align}
\label{T}
T = \Skew(0:\hbox{\tiny{$1$}},\hbox{\tiny{$2$}}|0: \hbox{\tiny{$\bar 2$}}, \hbox{\tiny{$\bar 1$}} )
\end{align}

\noindent we have $word(T) = 2 \bar 1 1 \bar 2$. For each $1\leq i \leq n$, to a word $w\in \mathcal{P}_{n}$ we assign a labelling of the letters of $w$ by $+,-$ or no label. For $i \leq n-1$, label the letters $i, \overline{i+1}$ by $+$ and the letters $i+1, \overline{i}$ by $-$. If $i = n$, label $n$ by $+$ and $\overline{n}$ by $-$. The remaining letters remain without label. Finally, cancel out pairs of labels of the form $+ -$, that is, cancel out every label $+$ with the first $-$ to its right, starting from the left-most one. For example, if the sequence of labels is 
$-+\hbox{ }--\hbox{ }++ \hbox{ }$ (blank spaces mean no label), after the cancelling out process we obtain $-\hbox{ }\hbox{ }\hbox{ }-\hbox{ }++ \hbox{ }$. Like this, we obtain a sequence of labels which looks like this (after ignoring blank spaces):
\[(-)^r(+)^{s}\]

\noindent for some $r,s \in \mathbb{Z}_{\geq 0}$. This is the \textbf{i-signature} of $w$ (but we also keep a record of the position of the remaining labels). We will denote it by $\sigma_{i}(w)$. For example, the 1-signature of $word(T)$ as in (\ref{T}) is $- - + +$. Its 2-signature is empty. 
To apply the root operator $f_{i}$ to $T$, we replace in $T$ the letter $a$ which is tagged by the left-most $+$ in the i-signature of $word(T)$, by the letter $\bar a$, where $\overline{\bar a } = a$. If $s = 0$, then $f_{i}$ is not defined. To apply $e_{i}$, we replace in $T$ the letter $a$ which is tagged by the right-most $-$ in the i-signature of $word(T)$, by the letter $\bar a$, where $\overline{\bar a } = a$. If $r = 0$, then $e_{i}$ is not defined. 

\subsection{Plactic relations for words. }\label{sec:lecrelations}

Note that the definition of the crystal operators and therefore of the simple reflections make sense on arbitrary words in the alphabet $\mathcal{P}_{n}$. In \cite{lec02} the following plactic relations (R1-3) on words are introduced.

\begin{enumerate}
\item[\mylabel{R1}{$R1$}]
$yzx \sim yxz \hbox{ for } x \leq y < z\hbox{ with } z \neq \bar x \hbox{ and }
xzy \sim zxy \hbox{ for } x < y \leq z \hbox{ with } z \neq \bar x
$;
\medskip
\item[\mylabel{R2}{$R2$}]
$
y \overline{x-1}(x-1) \sim y x \overline{x}  \hbox{ and }
x \overline{x} y \cong \overline{x-1} (x-1)y \hbox{ for }
1 < x \leq n \hbox{ and } x \leq y \leq \bar x
$;
\medskip
\item[\mylabel{R3}{$R3$}] 
$w\sim w\setminus\{z,\overline{z}\}$, where $w\in \mathcal{P}^{*}_{n}$ and $z\in [n]$ are such that $w$ is a non-admissible column, $z$ is the
lowest non-barred letter in $w$ such that $N(z) = z+1$ and any proper factor of $w$ is an admissible column.
\end{enumerate}

\noindent
 These relations  define an equivalence relation $\cong$ on the word monoid $\mathcal{P}^{*}_{n}$. 
Each word $w \in \mathcal{P}^{*}_{n}$ is equivalent via plactic relations to the word of a unique KN tableau $P(w)$. Moreover, there is the following characterization. Let $u,v \in \mathcal{P}^{*}_{n}$ and let $U, V$ the connected components (both normal $U_{q}(\mathfrak{sp}(2n,\mathbb{C}))$-crystals) in which they lie. Then $u \cong v$ if and only if there exists a crystal isomorphism $\eta: U \rightarrow V$ such that $\eta (u) = v$.

\subsection{Weyl group actions and modified crystal operators}

Let $\sigma_{i}(word(T))=(-)^r(+)^{s}$ be the $i$-signature of $word(T)$ as defined in the previous paragraph.
To apply the simple reflection $s_{i}$ to $T$ do the following: 

\begin{itemize}
 \item If $r = s$, then $s_{i}(T) = T$.
 \item If $r>s, s_{i}(T) = e^{r-s}_{i}(T)$.
 \item If $s>r, s_{i}(T) = f^{s-r}_{i}(T)$.
\end{itemize}

Let $x = s_{i_1}\cdots s_{i_r} 
 \in W$. The action of $x$ on a KN tableau $T$ is defined by 
 
 \[ s_{i_1}(\cdots (s_{i_r}(T))).\]

 More generally, given a crystal $B$ there is an action of the Weyl group $W$ on $B$ where $s_i$ acts by reversing the $f_i$-string, i.e. for $T\in B$ with $r=\eps_i(T)$ and $s=\phi_i(T)$, we define $s_i(T)$ as $e_i^{r-s}(T)$ if $r\geq s$ and $f_i^{s-r}(T)$ if $s\geq r$.

For a proof that this defines an action of $W$ also on the weights, see \cite[Proposition 2.36]{BSch17}.  For any $x \in W$ we have $x(\wt(T))=\wt(x(T))$.

\begin{definition}
\label{def:modified}
    In analogy with \cite{Pat}, we introduce the modified crystal operator $e_{12}:=s_1 e_2 s_1$ and $f_{12}:=s_1 f_2 s_1$. 
\end{definition}

\begin{remark}
Unfortunately, we cannot just define $e_{21}$ as $s_2 e_1 s_2$ to be the modified crystal operator attached to the root $\alpha_{21}$. In fact, in our inductive procedure we need the crystal operator  to be constructed by conjugating the root of higher index, but it is not possible here since $\alpha_{21}$ and $\alpha_2$ lie in different orbits under the Weyl group ($\alpha_2$ is long while $\alpha_{21}$ is short).  One of the main hurdles of generalizing the charge statistic from type $A$ to type $C$ is in fact to find an appropriate replacement for this crystal operator in the charge formula.
\end{remark}
\subsection{Adapted strings}

There are two reduced expressions for the longest element $w_0$ of type $C_2$: $s_1s_2s_1s_2$ and $s_2s_1s_2s_1$. 
After fixing a reduced expression $\sigma=s_{i_1}s_{i_2}s_{i_3} s_{i_{4}}$ of $w_0$, an element $T\in\calB(\lam)$ is uniquely determined by a quadruple of non-negative integers $\str_\sigma(T)=(a,b,c,d)$, called the adapted string, such that $T = f_{i_1}^{a}f^{b}_{i_2}f_{i_3}^{c}f_{i_4}^{d}(v_{\lambda})$, where $v_{\lambda} \in \calB(\lambda)$ is the highest weight vertex. 
We abbreviate $\str_{s_1s_2s_1s_2}$ as $\str_1$ and $\str_{s_2s_1s_2s_1}$ as $\str_2$.
The adapted strings for each of the different reduced expressions form a cone, denoted by $C_1$ and $C_2$. The precise relation between these two cones has been given by Littelmann.

\begin{theorem}[{\cite[Prop. 2.4]{conescrystalspatterns}}]
\label{adaptedstring}
There exists piecewise linear mutually inverse bijections 
  $\theta_{12}: C_{1} \rightarrow C_{2}$ and $\theta_{21}: C_{2} \rightarrow C_{1}$, such that $\theta_{12}\circ \str_1=\str_2$ and $\theta_{21}\circ \str_2=\str_1$, given by $\theta_{12}(a,b,c,d) = (a',b',c',d')$, where
\begin{align*}
    a' &= \operatorname{max}\left\{d, c-b,b-a \right\} \\
    b' &= \operatorname{max}\left\{c,a-2b+2c,a+2d \right\} \\
    c' &= \operatorname{min}\left\{b, 2b-c+d, a+d\right\} \\
    d' &= \operatorname{min}\left\{a, 2b-c, c-2d \right\},
\end{align*}
\noindent 
and $\theta_{21}(a,b,c,d) = (a',b',c',d')$, where
\begin{align*}
    a' &= \operatorname{max}\left\{d, 2c-b, b-2a\right\} \\
    b' &= \operatorname{max}\left\{c, a+d, a+2c-b \right\} \\
    c' &= \operatorname{min}\left\{ b, 2b-2c+d,d+2a \right\} \\
    d' &= \operatorname{min}\left\{a, c-d, b-c \right\}.
\end{align*}
\end{theorem}

Moreover, Littelmann precisely characterizes the adapted strings which occur in a given crystal $\calB(\lam)$.

\begin{theorem}[{\cite[Corollary 2, Prop. 1.5]{conescrystalspatterns}}]
\label{Littelmannineq}
Let $\lambda=\lambda_1\varpi_1+\lambda_2\varpi_2$. Given $(a,b,c,d)\in \bbZ_{\geq 0}^4$, there exists $x\in \calB(\lambda)$ with $\stn(x)=(a,b,c,d)$ if and only if the following inequalities hold:
\begin{itemize}
    \item $b\geq c \geq d$
    \item $d\leq \lambda_1$
    \item $c\leq \lam_2+d$
    \item $b\leq \lam_1-2d+2c$
    \item $a\leq \lam_2+d-2c+b$
\end{itemize}
\end{theorem}
\section{The atomic and preatomic decompositions}

In this section we introduce some important decompositions of the crystal $\calB(\lam)$.

\subsection{Preatoms}

We start by defining the preatomic decomposition. As we note in \Cref{remarkPreatomic}, the preatoms turn out to be a direct generalization of the LL atoms in type $A$, although they can contain several elements with the same weight.





\label{preatoms}




\begin{proposition}
\label{emb}
There is an embedding of crystals $\Phi: \calB(\lambda) \rightarrow \calB(\lambda + 2 \varpi_{1})$.  
\end{proposition}

\begin{proof} 
We define the map $\Phi$ on Kashiwara-Nakashima tableaux as follows. Note that  since $n = 2$, all tableaux will have at most two rows. Let $T$ be a Kashiwara-Nakashima tableaux of shape a partition $[a,b]$. Then we replace the first row of $T$, say $r^{1} = \Skew(0:\hbox{\tiny{$r^{1}_1$}} , ... , \hbox{\tiny{$r^{1}_k$}})$, by $\Skew(0:\hbox{\tiny{$1$}},\hbox{\tiny{$r^1_1$}} , ... , \hbox{\tiny{$r^{1}_k$}}, \hbox{\tiny{$\bar 1$}})$. The resulting tableau will be denoted by $\Phi'(T)$. 

If $\Phi'(T)$ contains the column $\Skew(0:\hbox{\tiny{$ 1$}}|0: \hbox{\tiny{$\bar 1$}})$, we replace it with the column $\Skew(0:\hbox{\tiny{$2$}}|0: \hbox{\tiny{$\bar 2$}})$. The new tableau will be denoted by $\Phi(T)$. Note that by semi-standardness, since $T$ does not contain a column $\Skew(0:\hbox{\tiny{$ 1$}}|0: \hbox{\tiny{$\bar 1$}})$, $\Phi'(T)$ can contain at most one such column.

The map $\Phi$ is well defined:
the tableau $\Phi'(T)$ is clearly semi-standard; this implies that, in case $\Phi'(T) \neq \Phi(T)$, then the latter must be semi-standard as well. Assume then that $\Phi'(T) \neq \Phi(T)$. The $1$ in the column $\Skew(0:\hbox{\tiny{$ 1$}}|0: \hbox{\tiny{$\bar 1$}})$ of $\Phi'(T)$ is necessarily the right-most one, so all entries to its right in $\Phi'(T)$ must be strictly larger than $1$. In the second row of $\Phi'(T)$, the $\bar 1$ which  is replaced by $\bar 2$ to obtain $\Phi(T)$ has to be the left-most one, since otherwise $\Phi'(T)$ would contain the column $\Skew(0:\hbox{\tiny{$ 1$}}|0: \hbox{\tiny{$\bar 1$}})$, which is impossible, since it is not an admissible column. The last thing missing to check in order to establish that $\Phi(T)$ is indeed a KN tableau is that it does not contain as a sub tableau $\Skew(0:\hbox{\tiny{$ 2$}},\hbox{\tiny{$2$}}|0:\hbox{\tiny{$\bar 2$}}, \hbox{\tiny{$\bar 2$}})$. But this is impossible, because then $\Phi'(T)$ would necessarily have to contain $\Skew(0:\hbox{\tiny{$ 1$}},\hbox{\tiny{$2$}}|0:\hbox{\tiny{$\bar 1$}}, \hbox{\tiny{$\bar 2$}})$ as a sub-tableau, which is not semi-standard. Note  that, by construction, $\Phi$ is weight-preserving. The case $\Phi'(T) = \Phi(T)$ is left to the reader, as the arguments are very similar to the ones above. \\

It remains to show that $\Phi$ is injective and that it commutes with the crystal operators. We start with a lemma.


\def \Phio {\Phi^{-\bar 1}}
\begin{lemma}
\label{importantembtool}
Let $T$ be a Kashiwara-Nakashima tableau, and let $w = \word(T)$ be its word. Then the word $\bar 1 w 1$ is plactic equivalent to $\word(\Phi(T))$. Moreover, the word $\bar 1 w 1$ is plactic equivalent to $\bar 1 u$, where $u$ is a word plactic equivalent to $w1$.
\end{lemma}

\begin{proof}

Let $r,s$ be positive integers such that the second row of $T$ has length $s$ and the first row, $s+r$. Let $a_{1} \leq \cdots \leq a_{r+s}$ be the entries in the first row, respectively $b_{1} \leq \cdots \leq b_{s}$ the entries in the second row of $T$ (if any).

Adding a $\bar 1$ at the end of the first row of a tableau just adds a $\bar 1$ at the beginning of its word. Let $\Phio(T)$ be the tableau obtained by removing the rightmost $\bar 1$ from $\Phi(T)$. It is then enough to show that $\word(\Phio(T))\cong w 1$. 

We actually prove a slightly stronger statement by induction on $s$: we have $\word(\Phio(T))\cong w 1$ and $\word(\Phio(T))=a_{r+s}v$, for some $v\in \calP_2^*$.


  The claim is clear if $s=0$. For $s>0$, consider  the tableau $U$ consisting of the first $s-1$ columns of $T$ and let $u=\word(U)$. We have $w=a_{r+s}\ldots a_{s+1}a_sb_s u$ and by induction we have
\[w1\cong a_{r+s}\ldots a_{s+1}a_sb_s\word(\Phio(U))=a_{r+s}\ldots a_{s+1}a_sb_s a_{s-1}v.\]

Since $a_{s-1}\leq a_s<b_s$ by Relation \ref{R1} in \S \ref{sec:lecrelations} we have 
\[a_{s}b_{s}a_{s-1} \cong a_{s} a_{s-1} b_{s}\]
\noindent except for when $b_{s} = \overline{a_{s-1}}$. Assume then we are in this case. Note that $b_{s} = \bar 2$ is impossible since semi-standardness alone then implies that $a_{s-1} = a_{s} = 2$ and $b_{s-1} = b_{s} = \bar 2$ but the tableau $\Skew(0:\hbox{\tiny{$2$}}, \hbox{\tiny{$2$}} |0: \hbox{\tiny{$\bar 2$}},\hbox{\tiny{$\bar 2$}})$ is not KN. Therefore the only option is $b_{s} = \bar 1$ and $a_{s-1} = 1$. In this case we have $a_s\in \{2,\bar 2\}$ and Relation \ref{R2} tells us that 
\[a_{s}\bar 1 1 \cong a_{s} 2 \bar 2.\]
Notice that the case $b_s=\bar 1$ precisely occurs when the $s$-th column of $\Phi'(T)$ is $\Skew(0:\hbox{\tiny{1}}|0:\hbox{\tiny{$\bar 1$}})$ and is replaced by $\Skew(0:\hbox{\tiny{2}}|0:\hbox{\tiny{$\bar 2$}})$ in $\Phi(T)$. From this we observe that in both cases we have $w1\cong \word(\Phio(T))$.
\end{proof}






We now go back to the proof of \Cref{emb}. From \Cref{importantembtool} we see immediately that $\Phi$ is injective.
Let $T$ be a KN tableau and $w = \word(T)$. We have $\sigma_{1}(\bar 1 w 1) = - \sigma_{1}(w) +$. This implies that, if $f_{1}$ is defined on $ w$ then it is also defined on $\bar 1 w 1$ and 
\begin{align}
\label{emb1}
f_{1}(\bar 1 w 1) = \bar 1 f_{1}(w) 1
\end{align}
Similarly, if $e_1(w)$ is defined, then $e_1(\bar 1 w 1) = \bar 1 e_1(w) 1$. 
We know by Lemma \ref{importantembtool}
that $\bar 1 w 1 \cong \word(\Phi(T))$ therefore
\begin{align}
\label{emb2}
     f_{1}(\word(\Phi(T))) \cong f_{1}(\bar 1 w 1) = \bar 1 f_{1}(w) 1\cong \word(\Phi(f_{1}(T))).
\end{align}

This implies that, since $ f_{1}(\Phi(T)), e_{1}(\Phi(T)), \Phi(e_{1}(T))$ and $\Phi(f_{1}(T))$ are $KN$ tableaux, we have
\begin{align}
\label{emb3}
f_{1}(\Phi(T)) = \Phi(f_{1}(T))\qquad
e_{1}(\Phi(T)) = \Phi(e_{1}(T))
\end{align}

\noindent as desired. Now, $\sigma_{2}(\bar 1w 1) = \sigma_{2}(w)$ by definition, so $e_2$ and $f_2$ are defined on $\Phi(T)$ if and only if are defined on $T$. Hence $f_2(\word(\Phi(T))=\word(\Phi(f_2(T)))$ and \eqref{emb3} hold after replacing $f_{1}$ by $f_{2}$ and $e_{1}$ by $e_{2}$. 
\end{proof}

\begin{corollary}
\label{corplactic}
Given a KN tableau $T$, the new tableau $\Phi(T)$ is defined by first column inserting the letter $1$ into $T$ using symplectic insertion and subsequently adding a $\bar 1$ at the end of the first row. 
\end{corollary}

\begin{proof}
The proof follows immediately from Lemma \ref{importantembtool}.
\end{proof}

\begin{corollary}\label{preatomsclosed}
The complement of $\Ima(\Phi)$ is closed under the action of $W$, under $e_2$ and  under outwards  $e_1$, i.e.
if $T\not \in \Ima(\Phi)$ and $\langle wt(T),\alpha_1^\vee\rangle\geq 0$ and $e_1(T)\neq 0$, then $e_1(T)\not \in \Ima(\Phi)$.
\end{corollary}
\begin{proof}
Since $\Phi$ commutes with $W$, its complement is union of $W$-orbits.
Let $T\not \in \Ima(\Phi)$.
We know that $\Phi(e_i(T))=e_i(\Phi(T))$ if $e_i(T)\neq 0$. 

 Assume $e_2(T)\neq 0$. If $e_2(T)=\Phi(T')$, then it follows from $\sigma_2(\Phi(w))=\sigma_2(w)$, that $f_{2}(T') \neq 0$ and therefore $T = f_{2}(\Phi(T')) = \Phi(f_{2}(T'))$, which is impossible.

Assume $e_1(T)\neq 0$ and $\langle \wt(T),\alpha_1^\vee\rangle\geq 0$. Assume $e_{1}(T) = \Phi(T')$. Since $\langle \wt(T'),\alpha_1^\vee\rangle =\langle \wt(T),\alpha_1^\vee\rangle +2 >0$, we have $f_1(T')\neq 0$, hence $T=f_1(\Phi(T'))=\Phi(f_1(T'))$, which is impossible.
\end{proof}


\begin{remark}
\label{remarkPreatomic}
    In analogy with \cite[Definition 3.18]{Pat} we can consider the connected components obtained as  $f_2$-closure of the $W$-orbits in the crystal graph.
    From \Cref{preatomsclosed} we see that preatoms are unions of the $f_2$-closure, and moreover, it turns out that for most $\lam$ (i.e. for $\lam_1>0$) each preatom consists of exactly one or two connected components, depending on the parity of $\lam_1$. In this sense, we can think of the preatoms in type $C_2$ as a direct generalization of the LL atoms in type $A$.
\end{remark}



\begin{definition}
For $\lambda$ such that $\lambda_{1} \geq 2$, we define the \textit{principal preatom} $\calP(\lambda)$ to be the complement of $\operatorname{Im}(\Phi)$ in $\calB(\lambda)$. If $\lambda_{1} \leq 1$, we define $\calP(\lambda):= \mathcal{B}(\lambda)$.

We define the \emph{preatomic decomposition} by induction on $\lam_1$. If $\lam_1\geq 2$, let
$\calB(\lam-2\varpi_1)=\bigsqcup \calP(\mu_i)$ be the preatomic decomposition. Then, the preatomic decomposition of $\calB(\lambda)$ is
\[ \calB(\lambda) = \calP(\lambda) \sqcup \bigsqcup \Phi(\calP(\mu_i)).\]
\end{definition}

Notice that all the preatoms in $\calB(\lam)$ are image of a principal preatom $\calP(\lam-2k\varpi_1)$ under the map $\Phi^k$. By \Cref{emb1str}
 we see that each preatom has a unique element of maximal weight. In particular, for any $\lam\in X$ every preatom of highest weight $\lam$ is isomorphic via some power of $\Phi$ to the principal preatom $\calP(\lam)\subset \calB(\lam)$. 





\begin{proposition}\label{emb1str}
Let $T \in \calB(\lambda)$ and consider $\Phi : \calB(\lambda) \rightarrow \calB(\lambda + 2\varpi_{1})$. Then 
\begin{enumerate}
    \item If $str_1(T) = (a,b,c,d)$, we have $str_1(\Phi(T)) = (a+1,b+1,c+1,d)$.
    \item If $\str_2(T) = (a,b,c,d)$ we have $\str_2(\Phi(T))=(a,b+1,c+1,d+1)$.
\end{enumerate}
\end{proposition}

\begin{proof}
If $T = v_{\lambda}$ is the highest weight vector, then it follows from Lemma \ref{importantembtool} that 
\begin{align}
\label{embhwstring2}
\Phi(T) = f_{1}f_{2}f_{1}(v_{\lambda + 2\varpi_{1}}).
\end{align}
\noindent
In this case $str_1(T) = str_1(v_{\lambda}) = (0,0,0,0)$ so the claim follows since $(1,1,1,0)$ is an adapted string for $\Phi(T)$. For arbitrary $T \in \calB(\lambda)$ it follows from \Cref{emb} that 
\begin{align}
    \label{embstringbeta}
\Phi(T) = f_{1}^{a}f^{b}_{2}f_{1}^{c}f_{2}^{d}f_{1}f_{2}f_{1}(v_{\lambda + 2\varpi_{1}}). 
\end{align}
 We introduce the following notation: 
  \begin{align*}
 (a',b',c',d') &:=\str_1(f^{d}_{2}f_{1}f_{2}f_{1}(v_{\lambda + 2\varpi_{1}}))=\theta_{21}(d,1,1,1)\\
(a'',b'',c'',d'') &:=\str_1(f^{a'+c}_{1}f^{b'}_{2}f^{c'}_{1}f^{d'}_{2}(v_{\lambda + 2\varpi_{1}}))\\
(a''',b''',c''',d''') &:=\str_2(f^{a''+b}_{2}f^{b''}_{1}f^{c''}_{2}f^{d''}_{1}(v_{\lambda + 2\varpi_{1}})).
\end{align*}
 
By \Cref{adaptedstring}, we have $(a',b',c',d')=\theta_{12}(d,1,1,1)=(1,d+1,1,0)$. Moreover, it follows from \cite[Cor. 2, ii.]{conescrystalspatterns} that 
$(a'',b'',c'',d'')=(0,c+1,d+1,1)$ and $(a''',b''',c''',d''')=(1,b+1,c+1,d)$.
Putting all of this together we get that 
\begin{align*}
\Phi(T) &= f_{1}^{a}f^{b}_{2}f_{1}^{c}f_{2}^{d}f_{1}f_{2}f_{1}(v_{\lambda + 2\varpi_{1}}) \\
& = f_{1}^{a}f^{b}_{2}f_{1}^{a'+c}f_{2}^{b'}f_{1}^{c'}f_{2}^{d'}(v_{\lambda + 2\varpi_{1}}) \\
& = f_{1}^{a}f^{b+a''}_{2}f_{1}^{b''}f_{2}^{c''}f_{1}^{d''}(v_{\lambda + 2\varpi_{1}}) \\
& = f_{1}^{a+a'''}f^{b'''}_{2}f_{1}^{c'''}f_{2}^{d'''}(v_{\lambda + 2\varpi_{1}}). 
\end{align*}
Therefore
\[(a+a''',b'',c''',d''') = (a+1,b+1,c+1, d) = \str_1(\Phi(T)),\]
showing the first statement. The proof of the second statement is similar. It follows from Lemma \ref{importantembtool} that 
\begin{equation}
 \label{embhwstring}
\Phi(T) = f_{1}f_{2}f_{1}(v_{\lambda + 2\varpi_{1}}),    
\end{equation}
so that $\str_2(\Phi(v_{\lambda + 2\varpi_{1}})) = (0,1,1,1)$. Using \cite[Prop. 2.4]{conescrystalspatterns}
we get that 
\begin{align*}
    \Phi(T) &= f^{a}_{2}f^{b}_{1}f^{c}_{2}f^{d}_{1}(f_{1}f_{2}f_{1}(v_{\lambda + 2\varpi_{1}})) \\
    & = f^{a}_{2}f^{b}_{1}f^{c}_{2}f^{d+1}_{1}f_{2}f_{1}(v_{\lambda + 2\varpi_{1}})\\
    & = f^{a}_{2}f^{b}_{1}f_{1}f^{c+1}_{2}f^{d+1}_{1}(v_{\lambda + 2\varpi_{1}}) \\
    & = f^{a}_{2}f^{b+1}_{1}f^{c+1}_{2}f^{d+1}_{1} (v_{\lambda + 2\varpi_{1}}).
\end{align*}
This concludes the proof. 
\end{proof}

\begin{remark}
    Notice that one can avoid the recourse to tableaux combinatorics and use the equation in \Cref{emb1str} as the definition of $\Phi$. Then one can use the explicit description of the adapted strings in \Cref{adaptedstring} to the check that $\Phi$ is well defined and that has the desired properties.
\end{remark}

The description of the embedding $\Phi$ in terms of adapted strings allows us to give a convenient description of the elements in principal preatom $\calP(\lambda)\subset \calB(\lam)$.

\begin{corollary}
\label{preatominequalities}
There exists $T\in \calP(\lambda)$ with $\stn(T)=(a,b,c,d)$ if \textbf{and only if} all the inequalities in \Cref{Littelmannineq} hold and at least one of the following equations hold.
\begin{itemize}
    \item $d=0$
    \item $d=\lam_1$
    \item $b= \lam_1-2d +2c$
\end{itemize}
\end{corollary}
\begin{proof}
Let $T\in\calB(\lam)$ with $\str_2(a,b,c,d)$, so all the inequalities in \Cref{Littelmannineq} hold.
There exists $U\in \calB(\lambda-2\varpi_1)$ with $\str_2(U)=(a,b-1,c-1,d-1)$ so that $\Phi(U)=T$ if and only if all the inequalities in \Cref{Littelmannineq} hold for $(a,b-1,c-1,d-1)$ and $\lambda-2\varpi_1$, which written explicitly means that $d\geq 1$, $d\leq \lambda_1-1$ and $b\leq \lam_1-2d+2c-1$ (the others remain unchanged). The claim now easily follows for $\lambda_{1} \geq 2$. 
\end{proof}

\begin{definition}\label{preatomicnumber}
Let $T\in \calB(\lambda)$.
Let $\pat(T)\in \bbZ_{\geq 0}$ be such that $T\in \calP(\lambda-2\pat(T)\varpi_1)\subset \calB(\lambda)$. We call $\pat(T)$ the \emph{preatomic number} of $T$.

In other words, $\pat(T)$ is the maximum integer with $T\in \Ima(\Phi^{\pat(T)})$.
\end{definition}

We now compute the size of the preatoms using the precanonical bases from \Cref{sec:precan}.

\begin{definition}
Let $\calB^+(\lam)$ be the subset of $\calB(\lam)$ consisting of elements whose wight is dominant. For a subset of $C\subset \calB^+(\lam)$ we define the \emph{ungraded character of $C$} as 
 \[ [C]_{v=1} := \sum_{c\in C} e^{\wt(c)} \in \bbZ[X_+]\]
More generally, for a subset $C\subset \calB(\lam)$ stable under the $W$-action we define 
  \[ [C]_{v=1} := [C\cap \calB^+(\lam)]_{v=1}\]

\end{definition}

\begin{proposition}\label{PreatomPrecan}
We have $[\calB(\lam)]_{v=1}=(\undH_{\lam})_{v=1}$ and $[\calP(\lam)]_{v=1}=(\tilN^3_\lam)_{v=1}$.
\end{proposition}
\begin{proof}
The statement about $\calB(\lam)$ follows by the Satake isomorphism (see for example \cite{Knop}). The second statement follows  easily from the definition of $\tilN^3_\lam$. In fact, if $\lam_1\leq 1$ we have $\calB(\lam)=\calP(\lam)$. If $\lam_1\geq 2$ we have $\calP(\lam)=\calB(\lam)\setminus \Phi(\calB(\lam-2\varpi_1))$. Since $\Phi$ is weight preserving and injective, we have 
\[[\calP(\lam)]_{v=1}=[\calB(\lam)]_{v=1}-[\calB(\lam-2\varpi_1)]_{v=1}=(\undH_{\lam}-\undH_{\lam-2\varpi_1})_{v=1}=(\tilN^3_{\lam})_{v=1}.\qedhere\]
\end{proof}

\subsubsection{The preatomic \texorpdfstring{$Z$}{Z} function}

In analogy with \cite[Definition 3.23]{Pat} we define a $Z$ function in type $C$. 
\begin{definition}
For $T\in \calB(\lam)$, let $Z(T):=\phi_1(T)+\phi_2(T)+\phi_{21}(T)$. 
\end{definition}

The $Z$-function is not constant along preatoms but nevertheless can be used to give an explicit formula for the preatomic number.
\begin{proposition}\label{atomicnumber}


Assume $T\in \calB(\lam)$ and let $\mu:=\wt(T)$. Then we have 
\begin{equation}\label{Zform}Z(T)= \lambda_1+\lambda_2+\mu_1+\mu_2+\max\left(0,\frac{|\mu_1|-\lambda_1}{2}\right)+\pat(T).\qedhere\end{equation}
\end{proposition}

\begin{proof}
We show the claim by induction on $\pat(T)$.
    We first assume $\pat(T)=0$, or equivalently that $T\in \calP(\lam)\subset \calB(\lam)$. Let $(a,b,c,d)=\stn(T)$.  

Let $\bbT=(\bbQ\cup \{+\infty\}, \oplus, \odot)$ be the tropical semiring (see \cite{TropicalBook}), where 
 $x\oplus y=\min(x,y)$ denotes the tropical addition and $x\odot y= x+y$ is the tropical multiplication. We also write fractions in $\bbT$ for the tropical division, i.e. $\frac{x}{y}=x-y$. A tropical polynomial is the function expressing the minimum of several linear functions. A tropical rational function is the difference of two tropical polynomials.

    Our first goal is to reinterpret both sides of \eqref{Zform} as tropical rational functions in $a,b,c,d,\lam_1$ and $\lam_2$.
    For example, $\mu_1$ can be expressed as a tropical rational function: since we have $\mu_1=\lam_1+2a+2c-2b-2d$, we can write $\mu_1=\frac{\lam_1\odot a^{\odot 2} \odot c^{\odot 2}}{b^{\odot 2}\odot d^{\odot 2}}$.
    In the rest of this proof we make the notation lighter by simply writing $xy$ for $x\odot y$ and  $x^n$ for $x^{\odot n}$.
    Since $\pat(T)=0$ we can rewrite the RHS in \eqref{Zform} as 
    \[RHS(T):=\frac{\lam_1^2 \lam_2^2}{b d (1 \oplus \frac{\lam_1 ac}{bd}\oplus \frac{bd}{ac})}=\frac{ac\lam_1^2\lam_2^2}{a^2c^2\lam_1\oplus abcd\oplus b^2d^2}.\]
    Expressing the LHS of \eqref{Zform} is unfortunately a much longer computation. 
    We have $Z(T)=\phi_2(T)\odot \phi_1(T) \odot \phi_{12}(T)$ and
    \begin{itemize}
        \item $\phi_2(T)=\frac{bd\lam_2}{ac^2}$
        \item$\phi_1(T)=\phi_1'\circ \theta_{21}(a,b,c,d)$, where $\phi_1'(a,b,c,d)=\frac{b^2d^2\lam_1}{ac^2}$ and $\theta_{21}$ is as in \Cref{adaptedstring}. 
        \item $\phi_{12}(T)=\phi_2\circ \theta_{12}\circ \sigma_1\circ \theta_{21}(a,b,c,d)$ where $\sigma_1(a,b,c,d)=(\frac{\lam_1b^2d^2}{ac^2},b,c,d)$ is the transformation expressing the   action of the simple reflection $s_1$ on $\str_1$.
    \end{itemize} 
    From this, we can obtain an explicit expression of $Z(T)$ as a tropical rational function. However, this is a rather unfeasible task to do by hand, so we resort to the help of the computer algebra software  \cite{SageMath}. In Sage we can simply compute $Z(t)$ by formally treating its three factors as ordinary rational functions in $\bbQ(a,b,c,d,\lam_1,\lam_2)$.

    Then, to check the claim, we need to show that $Z(T)=RHS(T)$ when $d=0$, $d=\lam_1$ or $b=\lam_1+2d-2c$. In other words, we need to show that, as tropical rational functions on the set of elements of the crystal, we get $Z(T)/RHS(T)=1$ if we specialize $d=1$,\footnote{Recall that $0\in \bbQ$ is the multiplicative unity in $\bbT$} $d=\lam_1$ or $b=\lam_1d^2/c^2$. Again, this can be  checked with the help of SageMath. In \Cref{appendix} we attach the code that proves our claim.

    Assume now $\pat(T)>0$, so $T=\Phi(T')$ for some $T'\in \calB(\lam-2\varpi_1)$. Since $\pat(T)=\pat(T')+1$, by induction it suffices to show that $Z(T)=Z(T')+1$.  From \Cref{emb1str} it follows that $\phi_1(T)=\phi_1(T')+1$ and $\phi_2(T)=\phi_2(T')$. Moreover, we have \[\phi_{12}(T)=\phi_2(s_1(T))=\phi_2(s_1(\Phi(T')))=\phi_2(\Phi(s_1(T')))=\phi_2(s_1(T'))=\phi_{12}(T')\]
    since $\Phi$ commutes with $s_1$, and the claim follows.
\end{proof}

\subsection{Atoms}

The goal of this section is to describe a finer decomposition of $\calB(\lam)$ into atoms.
\begin{definition}
\label{def:atom}
    We call a subset $A\subset \calB(\lam)$ an \emph{atom} if $[A]_{v=1}=(\bfN_\mu)_{v=1}$ for some $\mu \in X_+$. This means that there exists $\mu\in X_+$ such that every weight smaller or equal than $\mu$ in $X$ occurs exactly once as the weight of an element in $A$.

    An \emph{atomic decomposition} is a decomposition of $\calB(\lam)$ into atoms.
\end{definition}

\begin{proposition}
\label{atomemb}
 There is an injective weight-preserving map $\bPsi:\mathcal{P}(\lambda) \hookrightarrow \mathcal{P}(\lambda + \varpi_{2})$. If $\lam_1\neq 0$ then the set $\calA(\lambda+\varpi_2):=\calP(\lambda+\varpi_2) \setminus \bar{\Psi}(\calP(\lambda))$ is an atom. If $\lam_1=0$ then the set $\calA(\lambda+2\varpi_2):=\calP(\lambda+2\varpi_2) \setminus \bar{\Psi}^2(\calP(\lambda))$ is an atom.
\end{proposition}

We divide the proof into several steps. 
We begin by defining a map $\Psi$ directly in terms of the adapted strings (we give in \Cref{sec:psitableaux} an alternative construction in terms of KN tableaux.) The map $\bar{\Psi}$ is then obtaining by making $\Psi$ symmetric  along $s_1$.
Then we prove injectivity in \Cref{omega2} and that the complement is an atom in \Cref{atoms}.

\begin{lemma}
\label{welldefomega2}
Let $T\in \calP(\lambda)$ with $\stn(T)=(a,b,c,d)$. Then we have the following:
\begin{enumerate}
    \item If $d\in \{0,\lambda_1\}$, there exists $U\in \calP(\lambda+\varpi_2)$ with $\stn(U)=(a,b+1,c+1,d)$;
    \item If $d\not\in \{0,\lambda_1\}$, there exists $U\in \calP(\lambda+\varpi_2)$ with $\stn(U)=(a,b,c+1,d+1)$.
\end{enumerate}
\end{lemma}
\begin{proof}
Assume first $d=0$ and $d=\lam_1$. The Littelmann inequalities for $(a,b+1,c+1,d)$ and $\lambda+\varpi_2$ are implied by the original ones for $(a,b,c,d)$ and $\lambda$, so there exists such $U\in \calB(\lambda+\varpi_2)$. Since $d=0$ or $d=\lam_1$ we also see that $U\in \calP(\lambda+\varpi_2)$.

Assume now $d\neq 0$ and $d\neq \lambda_1$. Since $T\in \calP(\lambda)$ we have $b=\lam_1-2d+2c$. The Littelmann inequalities for $(a,b,c+1,d+1)$ and $\lambda+\varpi_2$ are:
\begin{itemize}
    \item $b\geq c+1 \geq d+1$,
    \item $d+1\leq \lambda_1$,
    \item $c+1\leq \lam_2+1+d+1$,
    \item $b\leq \lam_1-2d+2c$, and
    \item $a\leq \lam_2+d-2c+b$.
\end{itemize}
All these inequalites are implied by the original ones (and by $d\neq \lambda_1$) except $b\geq c+1$. However, if $b<c+1$ then $b=c$ and $c=\lam_1-2d+2c$ or, equivalently, $d=\frac12 (c+\lam_1)$. Since $d\leq c$ and $d< \lam_1$ this is impossible. It follows that there exists $U\in \calB(\lambda+\varpi_2)$ with $\stn(U)=(a,b,c+1,d+1)$. Moreover, $b=\lam_1-2(d+1)+2(c+1)$, so $U\in \calP(\lambda+\varpi_2)$
\end{proof}

\Cref{welldefomega2} ensures that the following function is well defined.
\begin{definition}\label{defPsi}
    We define $\Psi:\calP(\lambda)\raw \calP(\lambda+\varpi_2)$ as follows. Let $T\in \calP(\lam)$ with $\stn(T)=(a,b,c,d)$. Then $\Psi(T)=U$ with
    \[ \stn(U)=\begin{cases}
    (a,b+1,c+1,d) & \text{if }d=0\text{ or }d=\lam_1\\
    (a,b,c+1,d+1) &\text{otherwise}.\end{cases}\]

    We also define $\bPsi:\calP(\lambda)\raw \calP(\lambda+\varpi_2)$ as follows. 
    \[\bPsi(T)=\begin{cases}\Psi(T)& \text{if }\wt(T)_1\leq 0\\
    s_1(\Psi(s_1(T)))&\text{if }\wt(T)_1\geq 0\end{cases}\]
\end{definition}

\begin{lemma}\label{lemmaonPsi}
For $T\in \calP(\lam)$ we have:
\begin{enumerate}
\item $\wt(\Psi(T))=\wt(\bPsi(T))=\wt(T)$
\item $\phi_2(\Psi(T))=\phi_2(T)$.
\item If $f_2(T)\neq 0$ also $f_2(\Psi(T))=\Psi(f_2(T))$.
\item If $e_2(T)\neq 0$ also $e_2(\Psi(T))=\Psi(e_2(T))$.
\item $s_1(\bPsi(T))=\bPsi(s_1(T))$.
\end{enumerate}
\end{lemma}
\begin{proof}
This is clear by the definition of $\stn$.
\end{proof}

\begin{lemma}
\label{omega2}
The maps $\Psi,\bPsi:\calP(\lambda)\raw \calP(\lambda+\varpi_2)$ are injective.
\end{lemma}
\begin{proof}
It is enough to prove the statement for $\Psi$.
Assume $\Psi(T)=\Psi(U)$ with $T\neq U$. Let $\stn(T)=(a,b,c,d)$ and $\stn(T)=(a',b',c',d')$. We can assume that 
$d\not\in\{ 0,\lambda_1\}$, $d'\in \{0,\lam_1\}$ and that \[\stn(\Psi(T))=(a,b,c+1,d+1)=(a',b'+1,c'+1,d').\] 

It follows that $d'=d+1$, $c'=c$ 
and $b'=b-1$. 
Since \[b-1=b'\leq \lam_1-2d'+2c'=\lam_1-2(d+1)+2c\] it follows that $b\leq \lam_1-2d+2c-1$. But this contradicts the fact that $b=\lam_1-2d+2c$.
\end{proof}

Recall the atomic basis $\bfN=\bfN^2$ of the spherical Hecke algebra from \Cref{sec:precan}.

\begin{proposition}
\label{atoms}
We have $[\calA(\lam)]_{v=1}=(\bfN_{\lam})_{v=1}$. In particular the set $\calA(\lam)$ is an atom.  
\end{proposition}
\begin{proof}
If $\lam_2=0$ we have $\calA(\lam)=\calP(\lam)$, so $[\calA(\lam)]_{v=1}=(\tilN_{\lam})_{v=1}=(\bfN_{\lam})_{v=1}$.

If $\lam_2=1$ and $\lam_1=0$ then we can easily check that $\calB(\lam)$ consists of a single atom. If $\lam_2>1$ and $\lam_1=0$ then we have
$\calA(\lam)=\calP(\lam)\setminus \bPsi^2(\calP(\lam-\varpi_2))$.
Since $\bPsi$ is injective and weight-preserving, we have by \Cref{precanN3} that
\[ [\calA(\lam)]_{v=1}=[\calP(\lam)]_{v=1}-[\calP(\lam-2\varpi_2)]_{v=1}=(\tilN^3_{\lam}-\tilN^3_{\lam-2\varpi_2})_{v=1}=(\bfN_\lam)_{v=1}.\]
Finally, assume $\lam_2 >0$ and $\lam_1>0$. Then, we have $\calA(\lam)=\calP(\lam)\setminus \bPsi(\calP(\lam-\varpi_2))$. Since $\bPsi$ is injective and weight-preserving, we have
\[ [\calA(\lam)]_{v=1}=[\calP(\lam)]_{v=1}-[\calP(\lam-\varpi_2)]_{v=1}=(\tilN^3_{\lam}-\tilN^3_{\lam-\varpi_2})_{v=1}=(\bfN_\lam)_{v=1}.\qedhere\]
\end{proof}

From this we can  obtain an atomic decomposition of $\calB(\lam)$. Because we already know how to decompose $\calB(\lam)$ into preatoms, it is enough to decompose each preatom $\calP(\lam)$ into atoms. If $\lam_2=0$ or if $\lam=(0,1)$ we have $\calP(\lam)=\calA(\lam)$. If $\lam_2>0$ and $\lam\neq (0,1)$ then we have 
\[ \calP(\lam) = \begin{cases}\calA(\lam) \sqcup \bPsi(\calP(\lam -\varpi_2))&\text{if }\lam_1>0\\
\calA(\lam) \sqcup \bPsi^2(\calP(\lam -2\varpi_2))&\text{if }\lam_1=0\\
\end{cases}\]
so, applying $\bPsi$, we obtain an atomic decomposition by induction.

\begin{remark}
It is worth noting that an atomic decomposition can also be obtained by taking the complement of $\Psi$ rather than $\bPsi$. The advantage of using $\bPsi$ is to ensure that atoms are stable under $s_1$. This stability is crucial, as our approach inherently relies on $s_1$-symmetry, as discussed for example in \Cref{21swappable}. It is therefore 
 essential to ensure that the structures we define are compatible with this symmetry.
\end{remark}

\begin{lemma}\label{phi1Psi}
Let $T\in \calP(\lam)$ with $\stn(T)=(a,b,c,d)$. Then 
\[ \phi_1(\Psi(T))=\begin{cases}
\phi_1(T)& \text{if }d=0\text{ and }2a>b>2c\text{ or }d\neq 0,\lam_1\text{ and }b> 2a+d\\
\phi_1(T)+1 \hspace{-2pt}&\text{otherwise}.
\end{cases}\]

Moreover, if $\phi_1(\Psi(T))=\phi_1(T)$ and $\mu_1\leq 0$, then $\phi_1(T)=0$
\end{lemma}

\begin{proof}
Let $\pi_1:\bbZ^4\raw \bbZ$ be the projection onto the first component. Then, we have 
\begin{equation}\label{phi1stn}\phi_1(T)=\pi_1(\theta_{21}(\stn(T)))+(\wt(T))_1=\lam_1+2a-2b+2c-2d+\max(d,2c-b,b-2a).\end{equation}
From here we see that, if $d=0$ or $d=\lam_1$, we have
\[\phi_1(\Psi(T))-\phi_1(T)=\max(d,2c-b+1,b-2a+1)-\max(d,2c-b,b-2a).\]
If $d=0$, then $\phi_1(\Psi(T))=\phi_1(T))$ if and only if $2a>b>2c$. If $d=\lam_1$, we have $2c-b\geq 2d-\lam_1=\lam_1$, so $\phi_1(\Psi(T))-\phi_1(T)=1$.

If $0<d<\lam_1$ and $b=\lam_1-2d+2c$, then \[\phi_1(\Psi(T))-\phi_1(T)=\max(d+1,2c-b+2,b-2a)-\max(d,2c-b,b-2a),\]
but $2c-b=2d-\lam_1<d$, so $\phi_1(\Psi(T))-\phi_1(T)=\max(d+1,b-2a)-\max(d,b-2a)$ and the claim easily follows.
\end{proof}

\begin{corollary}\label{cor:phi12psi}
Let $T\in \calP(\lam)$ with $\stn(T)=(a,b,c,d)$. Then 
\[ \phi_{12}(\Psi(T))=\begin{cases}
\phi_{12}(T)+1& \text{if }d=0\text{ and }2a>b>2c\text{ or }d\neq 0,\lam_1\text{ and }b> 2a+d\\
\phi_{12}(T)&\text{otherwise}.
\end{cases}\]
\end{corollary}
\begin{proof}
It follows from \Cref{atomicnumber} that
\[ \phi_{12}(\Psi(T))-\phi_{12}(T)=1-(\phi_1(\Psi(T))-\phi_1(T)),\]
so we conclude by \Cref{phi1Psi}.
\end{proof}

\begin{definition}\label{defatomicnumber}
Let $T\in \calB(\lambda)$.
Let $\at(T)\in \bbZ_{\geq 0}$ be the maximum integer such that $T$ is in the image of $\bPsi^{\at(T)}:\calP(\lam-\at(T)\varpi_2)\raw \calP(\lam)$. We call $\at(T)$ the \emph{atomic number} of $T$.
\end{definition}

\begin{proposition}\label{atomformula}
Let $T\in \calP(\lambda)\subset \calB(\lam)$ with $\stn(T)=(a,b,c,d)$ and $\wt(T)_1\leq 0$.  We have
\[ \at(T)=\begin{cases} \min(c,\lam_1+2c-b)& \text{if }d=0\\
\lam_1+2c-2d-b+\min(\lam_2+d-c,d-1) &\text{if }d>0.\end{cases}\]
\end{proposition}

\begin{proof}
Notice that since $\wt(T)_1\leq 0$ we have $\Psi=\bPsi$.

First recall that by \Cref{Littelmannineq}, we have $0 \leq d \leq \lambda_{1}$. If $d = 0$, $\at(T)$ is the maximal amount we can substract simultaneously from $b$ and $c$, decreasing at the same time the value of $\lambda_{2}$ by the same amount, so that the inequalities and equalities mentioned in Corollary \ref{preatominequalities} still hold. Since $b \geq c$, we can focus only on $c$ and the inequality $\lambda_{1}+2c-b \geq 0$, which is the only other inequality describing $\mathcal{P}(\lambda)$ which is affected after reducing $b,c$ and $\lambda_{2}$ in equal amounts. Now if we decrease $b$ and $c$ simultaneously by the same amount, the quantity $\lambda_{1}+2c-b$ decreases by the same amount. Therefore, in this case $\at(T)= \min\left(c,\lambda_{1}+2c-b \right)$ as desired. 

Assume now  $d =\lam_1$. Recall that we need to find the maximal  $\at(T)$ such that the map $\Psi^{\at(T)}(U) = T$ for an element $U \in \mathcal{P}(\lambda-\at(T)\varpi_2)$. Recall that the definition of the map $\Psi$ depends on the value of $d$. Let $\psi_1$ and  $\psi_2$ be the two possible actions on adapted strings defined by $\Psi$, corresponding to the cases $0< d< \lambda_1$ and $d = 0,\lambda_1$ respectively (i.e. we have $\psi_1,\psi_2:\bbZ^4\raw \bbZ^4$ with $\psi_1(a,b,c,d)=(a,b,c+1,d+1)$ and $\psi_2(a,b,c,d)=(a,b+1,c+1,d)$). The definitions imply that we must have 
\[ \str_2(T) = \psi_{2}^{\at_2(T)}\psi_{1}^{\at_1(T)}(\str_2(U))\]
for some $\at_1(T),\at_2(T)\in \bbN$ with $\at_1(T)+\at_2(T)=\at(T)$.

Now, to calculate $\at_1(T)$ we first need to subtract the largest possible amount from $b$, $c$ and $\lambda_{2}$ such that our inequalities and equalities stated in Corollary \ref{preatominequalities} will still hold. Analogously to the case  $d=0$ we can conclude that this number is $\at_{1}(T) = \min(c,\lambda_{1}+2c-b -2d )$. 
In this case the inequality $0 \leq \lambda_{1} + 2c-2d-b$ becomes $0 \leq 2c -\lambda_{1} -b \leq c $ since $c \leq b$. Therefore $\at_{1}(T) = \lambda_{1}+2c-b -2d$. To compute $\at_2(T)$ in this case, after already reducing $b,c$ and $\lambda_{2}$ by $\at_{1}(T)$ we need to further reduce $c' = c - \at_{1}(T)$ as well as $d$ and $\lambda_{2}' = \lambda_{2} - \at_{1}(T)$ by the maximal possible amount strictly smaller than $d$ such that the preatom inequalities/equalities will still hold. This amount is 
\[\at_{2}(T) =  \min \left(\lambda_{2}' +d - c', d-1 \right) = \min \left(\lambda_{2}+d-c, d-1 \right)\]

\noindent
since the inequality $\lambda_{2}+d - c'\geq 0$ is the only preatom inequality affected by decreasing $c,d$ and $\lambda_{2}$ simultaneously by the same amount. Moreover, it decreases precisely by this amount. 

Finally,  assume $0<d < \lambda_1$. As in the discussion above we have 
\[\str_2(T) = \psi_{2}^{\at_2(T)}(\str_2(U)),\]
and thus $\at(T)=\at_2(T)$. Moreover, if $0<d<\lam_1$ we have $b=\lam_1-2c+2d$ so we can also write $\at(T)=\lam_1+2c-2d-b+\at_2(T)=\lam_1+2c-2d-b+\at_2(T)+\min(\lam_2+d-c,d-1)$.
\end{proof}

\begin{corollary}\label{atom0}
Let $U\in \calP(\lambda)\subset \calB(\lam)$ with $\stn(U)=(a,b,c,d)$. Then $U \not \in \Psi(\calP(\lambda-\varpi_2))$ if and only if one of the following two conditions holds:
\begin{itemize}
    \item $b=\lam_1-2d+2c$ and ($d\leq 1$ or $c=\lam_2+d$)
    ;
\item  $b<\lam_1-2d+2c$ and $c=d=0$.
\end{itemize}

\end{corollary}

\begin{proof}
We know that $U \notin \Psi(\mathcal{P}(\lambda-\varpi_2)) \iff \operatorname{at}(U) = 0$. First assume $\at(U) = 0$. If $b = \lambda_1 - 2d +2c$ then from \Cref{atomformula} we see that either $d \leq 1$ or if $d>1$, we must have $\min(\lambda_2 +d-c, d-1) = 0$. Since $d >1$ this implies that $\lambda_2 +d - c = 0$. If $b < \lambda_1 - 2d +2c$ then by \Cref{atomformula} $d>0$ is impossible, so $d = 0$ necessarily. Moreover, since $\at(U) = 0$ we must have $\operatorname{min}(c, \lambda_1 + 2c-b)$, but since the second term is strictly larger than zero by assumption, we conclude $c = 0$. Conversely, if $b = \lambda_1 - 2d +2c$ and $d \leq 1$, it follows directly from \Cref{atomformula} that $\at(U) = 0$. If $c = \lambda_2 + d$ and $d>1$ then $\at(U) = 0$ also by \Cref{atomformula}. Now, if $b < \lambda_1 - 2d +2c$ and $c = d = 0$ then $\at(U) = 0$ applying the first formula in \Cref{atomformula}. 
\end{proof}

\subsection{Example: The atomic decomposition \texorpdfstring{of $\calB(k\varpi_2)$}{of B(kw2)}}
\label{sec:lam1=0}
 Let $B_{k} := \calB(k \varpi_{2})$. By definition $B_k$ consists of a single preatom. We describe now the atomic decomposition of $B_k$.
 Since $\lam_1=0$ we have $\str_2(T)=(a,b,c,0)$ for any $T\in B_k$. By \Cref{phi1Psi}, we see that $\phi_1(\Psi(T))=\phi_1(T)$ for any $T\in B_k$, hence $\Psi$ commutes with $s_1$ and we have $\Psi=\bPsi$. Then by \Cref{lemmaonPsi}.3, we see that $\Psi$ also commutes with $f_2$.
 
 Here we refer to the connected components under $W, f_2$ simply as \textit{connected components} (cf. \Cref{remarkPreatomic}). Notice that $\Psi$ preserves these connected components.
 We claim that the crystal $B_{k}$ has precisely $k+1$ connected components
\
\[
    B_{k} = \bigsqcup_{i = 0}^{k} B_{k}[i].
\]
and that $\Psi(B_{k-1}[i])=B_k[i]$.
In particular, it follows that $\calA(k\varpi_2)=B_k\setminus \Psi^2(B_{k-1})=B_k[k]\sqcup B_k[k-1]$.
\noindent

 The crystal $B_0$ consists of a single element, the empty tableau, so the claim is trivial. In $B_1$ there are two connected components. In fact, it is easy to see that 
\[B_{1}[0] = \left\{\Skew(0:\hbox{\tiny{$2$}}|0:\hbox{\tiny{$\bar 2$}}) \right\}\]

\noindent is fixed under the action of $f_{2}$ and $s_{1}$, and that its complement in $B_{1}$ is a connected component of cardinality $4$.

The weights of the elements in $\calA(k\varpi_2)$ form two square grid of side $k$ and $k+1$ as shown in \Cref{figurewts}, so $|\calA(k\varpi_2)|=(k+1)^2+k^2$. From this, it follows that $|B_k|-|B_{k-1}|=(k+1)^2$.

By induction, to show our claim it is enough to show that the complement of $\Psi(B_{k-1})$ in $B_{k}$ is a single connected component of cardinality $(k+1)^2$.


The complement of $\Psi$ always contains the highest weight vector $T_k\in B_k$. Then, for $0\leq r\leq k$, the tableaux 
\[ f^{r}_{2}(T_{k})=
\Skew(0:\hbox{\tiny{1}},\hbox{\tiny{$\cdots $}}, \hbox{\tiny{1}}, \hbox{\tiny{1}}, \hbox{\tiny{$\cdots$}}, \hbox{\tiny{1}}|0:\hbox{\tiny{2}},\hbox{\tiny{$\cdots$}},\hbox{\tiny{2}},\hbox{\tiny{$\bar 2$}},\hbox{\tiny{$\cdots$}},\hbox{\tiny{$\bar 2$}} ), \]
in which there are $r$ column of the  form $\Skew(0:\hbox{\tiny{1}}|0:\hbox{\tiny{$\bar2$}})$,  are also in the same connected component as $T_k$. 
 We obtain $s_1(f_2^r(T))$ from $f_2^r(T)$ by replacing the columns of the form $\Skew(0:\hbox{\tiny{1}}|0:\hbox{\tiny{$\bar2$}})$ by columns of the form $\Skew(0:\hbox{\tiny{2}}|0:\hbox{\tiny{$\bar1$}})$.
The  tableaux $s_1(f_2^r(T))$ are the highest element in their $f_2$-string, and there are $k+1$ elements in their $f_2$-orbit,
 given by barring some of the $2$'s. So we have seen that
 there are at least $(k+1)^2$ elements in the connected components of $T_k$. Since $\Psi$ is an embedding and $|B_k|-|B_{k-1}|=(k+1)^2$, these are precisely all the elements in the complement of $\Psi$.



 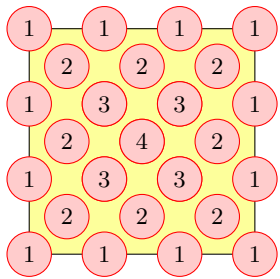
\begin{figure}
\begin{center}
 \begin{tikzpicture}[scale=0.5]
\draw[style=black] (3,3) to (-3,3) to (-3,-3) to (3,-3) to cycle;
\foreach \x in {0,1,2,3}{
    \foreach \y in {0,1,2,3}{
        \node[style=sred] at (3-2*\x,3-2*\y) {\small 1}; 
}}
\foreach \x in {0,1,2}{
    \foreach \y in {0,1,2}{
        \node[style=sred] at (2-2*\x,2-2*\y) {\small 2}; 
}}	
\foreach \x in {0,1}{
    \foreach \y in {0,1}{
        \node[style=sred] at (1-2*\x,1-2*\y) {\small 3}; 
}}	
\foreach \x in {0}{
    \foreach \y in {0}{
        \node[style=sred] at (0-2*\x,0-2*\y) {\small 4}; 
}}
 \end{tikzpicture}
 \end{center}
  \caption{The weight multiplicities of the crystal $B_3$.}
  \label{figurewts}
\end{figure}
\section{Swappable edges and their classification}

\subsection{Twisted Bruhat graphs}

The Bruhat order on the weight lattice $X$ is the order generated by the following relations

\begin{equation}\label{bruhatorder}s_\alpha^\vee(\lambda)< \lambda \iff \begin{cases}\langle \lambda,\beta^\vee\rangle > M &\text{if }M\geq 0,\\
\langle \lambda,\beta^\vee\rangle <  M&\text{if }M<0.\end{cases}\end{equation}
where $\alpha^\vee=M\delta+\beta^\vee$, with $\beta^\vee\in \Phi_+^\vee$ and $\lambda\in X$. The set of elements smaller that $\lambda$ in the Bruhat order, which we denote by $\{\leq \lambda\}$, can be characterized as
\begin{equation}\label{smallerthanlam}
    \{\leq \lambda\} =\Conv(W\cdot \lambda)\cap (\lambda+\bbZ\Phi)
\end{equation}
(see for example \cite[Chap. VIII, §7, exerc. 1]{Bou78}).

Let $\lambda\in X_+$. Let $\Gamma_\lambda$ denote the moment graph of the Schubert variety $\sch{\lambda}$. This is a directed labeled graph, also called the \emph{Bruhat graph} of $\lambda$. We recall  from \cite[\S 2.3]{Pat} the explicit description of $\Gamma_\lambda$.
The vertices of the graph $\Gamma_\lambda$ are all the weights in $\{\leq \lam\}$. We have an edge $\mu_1\raw \mu_2$ in $\Gamma_\lambda$ if and only if $\mu_2-\mu_1$ is a multiple of a root $\beta\in \Phi$ and $\mu_1\leq \mu_2$. In this case, the label of the edge $\mu_1\raw \mu_2$ is $m\delta-\beta^\vee$, where
\[m =-\frac{\langle \beta^\vee,\mu_1+\mu_2\rangle}{2}\] (cf. \cite[Lemma 2.6]{Pat}).  Notice that $s_{m\delta-\beta^\vee}(\mu_1)=\mu_2$. We denote by $E(\lambda)$ the set of edges in $\Gamma_\lam$.

Let $\Gamma_X$ denote the union of all the graphs $\Gamma_{\lambda}$, for $\lambda\in X_+$ (where $\Gamma_{\lambda}$ is regarded as a subgraph of $\Gamma_{\lambda'}$ if $\lambda\leq \lambda'$) and call it the Bruhat graph of $X$.

For $w\in \affW$ we denote by \[N(w):=\{\alpha \in \affPhi^\vee_+\mid w^{-1}(\alpha)\in \affPhi^\vee_-\}\] the set of inversions.
If $w=s_{i_1}\ldots s_{i_k}$ is a reduced expression for $w$ then \[N(w)=\{\alpha^\vee_{i_1},s_{i_1}(\alpha^\vee_{i_2}),\ldots,s_{i_1}s_{i_2}\ldots s_{i_{k-1}}(\alpha^\vee_{i_k})\}.\] 

We say that $w=s_1s_2\ldots s_k\ldots$ is a reduced infinite expression if for any $j$  the starting expression $w_j:=s_1s_2\ldots s_j$ is reduced.
If $w$ is a reduced infinite expression, let $N(w)=\bigcup_{j=1}^\infty N(w_j)$.

Consider $\undc = s_0s_2s_1s_2$. Then $y_\infty:=\undc \undc \undc\ldots$ is an infinite reduced expression.
Let $y_m$ be the element given by the first $m$ simple reflections in $y_{\infty}$.
We order the roots in $N(y_\infty)$ as follows:
\begin{multline}\label{reflectionorder}
\delta-\alpha_{21}^\vee<\delta-\alpha_{12}^\vee<2\delta-\alpha_{21}^\vee<  \delta -\alpha_{2}^\vee< 3\delta-\alpha_{21}^\vee<2\delta-\alpha_{12}^\vee< \\
\ldots<M\delta-\alpha_{12}^\vee<2M\delta-\alpha_{21}^\vee<M\delta-\alpha_{2}^\vee<(2M+1)\delta-\alpha_{21}^\vee<\ldots    
\end{multline}
so that the first $m$ roots in \eqref{reflectionorder} are precisely the elements of $N(y_m)$.


	We define the \emph{$m$-twisted Bruhat order} $\leq_m$ of $\extW$ by setting 
	\[v\leq_m w\text{ if and only if }y_m^{-1}v\leq y_m^{-1}w,\]
	and the $m$-twisted length by $\ell_m(v):=\ell(y_m^{-1}v)$. 
Recall that $X\cong \extW/W$. Hence, the twisted Bruhat order on $\extW$ also induces a twisted Bruhat order on $X$. Concretely, this means that we regard $\lambda\in X$ as a right coset in $\extW$ and denote by $\lambda_m\in \extW$ the element of minimal $y_m$-twisted length in the coset $\lambda$. Then we set $\ell_m(\lambda):=\ell_m(\lambda_m)$ and $\mu\leq_m \lambda$ if $\lambda_m\leq_m \mu_m$.

	For every $m\in \bbZ_{\geq 0}$ we define $\Gamma_\lambda^m$, the $y_m$\emph{-twisted Bruhat graph of} $\lambda$, to be the directed labeled graph with the same vertices of $\Gamma_{\lambda}$ and where there is an edge $\mu\ra \lambda$ if there exists $\alpha^\vee\in \affPhi^\vee$ such that $s_{\alpha^\vee}(\mu)=\lambda$ and $\mu<_m \lambda$. Concretely, we can obtain $\Gamma_\lambda^m$ from $\Gamma_{\lambda}$ by inverting the orientation of all the arrows in $\Gamma_\lambda$ with label in $N(y_m)$.
	
	Since each graph $\Gamma_{\lambda}$ has only a finite number of edges, the twisted graphs $\Gamma_{\lambda}^m$ stabilize for $m$ big enough, so we can define $\Gamma_{\lambda}^\infty:=\Gamma_{\lambda}^m$ for $m\gg 0$. 

For $m\in \bbZ_{\geq 0} \cup \{\infty\}$, we define $\Gamma^m_X$ as the union of all the graphs $\Gamma^m_\lambda$, for $\lambda\in X_+$. The graph $\Gamma^m_X$ can be obtained from $\Gamma_X$ by inverting the orientation of all the arrows with label in $N(y_m)$.

\begin{definition} 
\label{twistedarrowslambdamu}
	For $\mu\leq \lambda$, we denote by $\Arr_m(\mu,\lam)$ the set of arrows pointing to $\mu$ in $\Gamma_m^\lambda$ and by $\ell_m(\mu,\lambda):=|\Arr_m(\mu,\lam)|$ the number of those arrows. 
	
	For $i \in \{1,2,21,12\}$ let $\Arr_m^{i}(\mu,\lam)$ be arrows pointing to $\mu$ in $\Gamma_m^\lambda$ of the form $\mu-k\alpha_i\raw \mu$ for $k\in \bbZ$. Let $\ell_m^i(\mu,\lam)=|\Arr_m^{i}(\mu,\lam)|$.
	
	Let $\Arr_\mu(\mu)$ be the set of arrows pointing to $\mu$ in $\Gamma_X^m$. For $i\in \{1,2,21,12\}$, the set $\Arr_m^i(\mu)$ is defined accordingly.
\end{definition}

Recall from \cite[Lemma 4.6]{Pat} that $|\Arr_m(\mu)|=\ell_m(\mu)$. We have 
\begin{equation}\label{ellsubdivided} \Arr_m(\mu,\lambda)=\bigcup_{i\in \{1,2,12,21\}}\Arr_m^i(\mu,\lambda)\quad\text{and}\quad
\ell_m(\mu,\lambda)=\sum_{i\in \{1,2,12,21\}} \ell_m^i(\mu,\lam)
 \end{equation}
 for any $\mu\leq \lambda$. Notice that, since there are no arrows in $N(y_\infty)$ of the form $M\delta-\alpha_1^\vee$, the set $\Arr_m^1(\mu,\lambda)$ does not depend on $m$, and does not depend on $\lambda$ as long as $\mu\leq \lambda$. If $\mu \leq \lambda$, for all $m$ by \eqref{bruhatorder} we have 
 \begin{align*}
 \Arr_m^1(\mu,\lambda)&=\{\mu - k\alpha_1\ra \mu \mid \mu-k\alpha_1\leq \mu\}\\
 &=\begin{cases}
 \{\mu - k\alpha_1\ra \mu \mid 0<k\leq \mu_1\}&\text{if }\mu_1\geq 0\\
 \{\mu - k\alpha_1\ra \mu \mid 0>k>\ \mu_1\}&\text{if }\mu_1< 0.
 \end{cases}
 \end{align*}
Hence, we have  \begin{equation}\label{ell1}\ell_m^1(\mu,\lam)=\begin{cases} \mu_1 &\text{if }\mu_1\geq 0\\
-\mu_1-1& \text{if }\mu_1<0.\end{cases}\end{equation}

\subsection{Swappable edges}

To pass from $\Gamma^m_\lambda$ to $\Gamma^{m+1}_\lambda$ (and from $\Gamma^m_X$ to $\Gamma^{m+1}_X)$ we need to invert the arrows with label $\alpha^\vee_{t_{m+1}}$, where $t_{m+1}$ is the reflection \begin{equation}\label{tm}t_{m+1}:=y_{m+1}y_m^{-1} =y_ms'_{m+1} y_m^{-1}.
\end{equation}
Here $s'_{m+1}$ denotes the $(m+1)$-th simple reflection in $y_\infty$.
Notice that $\{\alpha^\vee_{t_{m+1}}\}= N(y_{m+1})\setminus N(y_m)$.

If $\mu < t_{m+1}\mu$, then $\Arr_{m+1}(t_{m+1}\mu)\setminus \Arr_m(t_{m+1}\mu)=\{\mu\raw t_{m+1} \mu\}$ and $\Arr_{m}(\mu)$ is in bijection with $\Arr_{m}(t_{m+1}\mu)\setminus \{\mu\raw t_{m+1}\mu\}$ by \cite[Lemma 4.8]{Pat}. In particular, we have 
\begin{equation}\label{lmX}
\ell_m(\mu)=\ell_m(t_{m+1}\mu)-1.
\end{equation}

A remarkable property of the twisted Bruhat graphs in type $A$ (\cite[Prop. 4.14]{Pat}) is that the same is true if we restrict to $\Gamma_\lambda$, i.e. $\ell_m(\mu,\lambda)=\ell_m(t_{m+1}\mu,\lambda)-1$ if $\mu <t_{m+1}\mu \leq \lambda$. This implies that
$\ell_{m+1}(\mu,\lambda)=\ell_m(t_{m+1}\mu,\lambda)$ and $\ell_m(\mu,\lam)=\ell_{m+1}(t_{m+1}\mu,\lam)$. However, as we will see in \Cref{exampleswap}, this property does not hold in type $C_2$. The goal of this section is to classify the set of edges for which it holds.

\begin{definition}
\label{def:swappableedge}
	We say that an edge $\mu \raw t_{m+1}\mu$ in $\Gamma_{\lambda}$ is \emph{swappable} if 
 \begin{equation}\label{eqswapdef}
 \ell_m(\mu,\lambda)= \ell_m(t_{m+1}\mu,\lambda)-1.
 \end{equation}
 We also say that an edge is \emph{NS} if it is not swappable.
	We denote by $E^S(\lambda)$ and $E^N(\lambda)$ the sets of swappable and non-swappable edges in $\Gamma_\lambda$, respectively. 
\end{definition}

As it turns out, to determine if an edge is swappable or not, we have to solve an elementary geometric problem,  as the next example illustrates.


\begin{example}\label{exampleswap}
In the \Cref{figswapp,fignonswapp} the starting points of the arrows in $\ell_m(\mu,\lam)$ are denoted by red circles while the starting points of the arrows in $\ell_m(t_{m+1}\mu,\lam)$ are denoted by blue squares. 

Assume that $\lambda=(2,2)$, $\mu=(2,-1)$ and that $m+1=8$, i.e. that $t:=t_{m+1}$ is the reflection corresponding to the root $2\delta-\alpha_2^\vee$.
In \Cref{figswapp}, the yellow octagon is the convex hull of $W\cdot \lambda$ while the green octagon is (the border of) the convex hull of $y_mWy_m^{-1}\cdot \mu$. As we will observe in \Cref{sec:twistedreflection}, the arrows in $\Arr_m(\mu,\lam)$ and $\Arr_m(t_{m+1}\mu,\lam)$ can be characterized as the weights in the diagonal of the green octagon which lie inside the yellow octagon. In this case we see that there are 
are $7$ red dots and $8$ blue squares, meaning that  the edge $\mu\raw t\mu$ is swappable.

Now assume that $\lam=(2,2)$, $\mu=(4,-2)$ and $m+1=12$, i.e. that $t:=t_{m+1}=s_{3\delta-\alpha_2^\vee}$. As illustrated in \Cref{fignonswapp}, we have $9$ red dots and $9$ blue squares, so in this case the edge $\mu\raw t\mu$ is not swappable.

\begin{figure}
\centering
\begin{minipage}{.5\textwidth}
  \centering
	\begin{tikzpicture}[scale=0.5]
 \node [style=none] (0) at (-2, 4) {};
	\node [style=none] (3) at (2, 4) {$\lambda$};
	\node [style=none] (4) at (-4, 2) {};
	\node [style=none] (5) at (4, 2) {};
	\node [style=none] (6) at (4, -2) {};
	\node [style=none] (7) at (2, -4) {};
	\node [style=none] (8) at (-4, -2) {};
	\node [style=none] (9) at (-2, -4) {};
	\node [style=none] (10) at (-1, 1) {};
	\node [style=none] (11) at (-3, 1) {};
	\node [style=none] (12) at (1, -1) {};
	\node [style=none] (13) at (1, -3) {};
	\node [style=none] (14) at (-1, -5) {};
	\node [style=none] (15) at (-3, -5) {};
	\node [style=none] (17) at (-5, -3) {};
	\node [style=none] (18) at (-5, -1) {};
 	\draw [style=black] (5.center) 	to (3.center) to (0.center) to (4.center) to (8.center) to (9.center) to (7.center) to (6.center) to cycle;
	\draw [style=green] (11.center) to (10.center);
	\draw [style=green] (10.center) to (12.center);
	\draw [style=green] (12.center) to (13.center);
	\draw [style=green] (13.center) to (14.center);
	\draw [style=green] (14.center) to (15.center);
	\draw [style=green] (15.center) to (17.center);
	\draw [style=green] (17.center) to (18.center);
	\draw [style=green] (18.center) to (11.center);
	\node [style=blue] (35) at (-1, 1) {};

	\node [style=none] (19) at (-1, 1) {$\mu$};
	\node [style=red] (20) at (-1, -1) {};
	\node [style=red] (21) at (-1, -3) {};
	\node [style=red] (22) at (-2, 0) {};
	\node [style=red] (23) at (-3, -1) {};
	\node [style=red] (24) at (-4, -2) {};
	\node [style=red] (25) at (0, 0) {};
	\node [style=red] (26) at (1, -1) {};
	\node [style=none] (27) at (-3, 1) {$t\mu$};
	\node [style=red] (36) at (-2, 0) {};
	\node [style=red] (37) at (-1, -1) {};
	\node [style=red] (29) at (-3, -1) {};
	\node [style=blue] (28) at (-3, -1) {};
	\node [style=blue] (30) at (-2, 0) {};
	\node [style=blue] (31) at (-1, -1) {};
	\node [style=blue] (32) at (0, -2) {};
	\node [style=blue] (33) at (1, -3) {};
	\node [style=blue] (34) at (-3, -3) {};
	\node [style=blue] (35) at (-4, 0) {};
    \node at (0,-8) {};
    \draw[green, very thick, ->] (19) to (27);
	\end{tikzpicture}
 \captionof{figure}{A swappable edge}
  \label{figswapp}
\end{minipage}%
\begin{minipage}{.5\textwidth}
  \centering
  \begin{tikzpicture}[scale=0.5]
		\begin{pgfonlayer}{nodelayer}
			\node [style=blue] (40) at (-2,2) {};
			\node [style=none] (0) at (-2, 4) {};
			\node [style=none] (3) at (2, 4) {$\lambda$};
			\node [style=none] (4) at (-4, 2) {};
			\node [style=none] (5) at (4, 2) {};
			\node [style=none] (6) at (4, -2) {};
			\node [style=none] (7) at (2, -4) {};
			\node [style=none] (8) at (-4, -2) {};
			\node [style=none] (9) at (-2, -4) {};
			\node [style=none] (10) at (-2, 2) {};
			\node [style=none] (11) at (-4, 2) {};
			\node [style=none] (12) at (2, -2) {};
			\node [style=none] (13) at (2, -4) {};
			\node [style=none] (14) at (-2, -8) {};
			\node [style=none] (15) at (-4, -8) {};
			\node [style=none] (17) at (-8, -4) {};
			\node [style=none] (18) at (-8, -2) {};
			\node [style=none] (19) at (-2, 2) {$\mu$};
			\node [style=red] (20) at (-2, 0) {};
			\node [style=red] (21) at (-2, -2) {};
			\node [style=red] (22) at (-3, 1) {};
			\node [style=red] (23) at (-4, 0) {};
			\node [style=red] (24) at (-2, -4) {};
			\node [style=red] (25) at (-1, 1) {};
			\node [style=red] (26) at (0, 0) {};
			\node [style=none] (27) at (-4, 2) {$t\mu$};
			\node [style=red] (29) at (-4, 0) {};
			\node [style=blue] (28) at (-4, 0) {};
			\node [style=red] (36) at (-3, 1) {};
			\node [style=red] (37) at (-2, 0) {};
			\node [style=red] (38) at (1, -1) {};
			\node [style=red] (39) at (2, -2) {};
			\node [style=blue] (40) at (1, -3) {};
			\node [style=blue] (41) at (2, -4) {};
			\node [style=blue] (30) at (-3, 1) {};
			\node [style=blue] (31) at (-2, 0) {};
			\node [style=blue] (32) at (-1, -1) {};
			\node [style=blue] (33) at (0, -2) {};
			\node [style=blue] (34) at (-4, -2) {};
            \draw[green, very thick, ->] (19) to (27);

		\end{pgfonlayer}
		\begin{pgfonlayer}{edgelayer}
			\draw [style=black] (5.center)
			to (3.center)
			to [in=0, out=180] (0.center)
			to (4.center)
			to (8.center)
			to (9.center)
			to (7.center)
			to (6.center)
			to (5.center);
			\draw [style=green] (11.center) to (10.center);
			\draw [style=green] (10.center) to (12.center);
			\draw [style=green] (12.center) to (13.center);
			\draw [style=green] (13.center) to (14.center);
			\draw [style=green] (14.center) to (15.center);
			\draw [style=green] (15.center) to (17.center);
			\draw [style=green] (17.center) to (18.center);
			\draw [style=green] (18.center) to (11.center);
		\end{pgfonlayer}
  	\end{tikzpicture}
  \captionof{figure}{A non-swappable edge}
  \label{fignonswapp}
\end{minipage}
\end{figure}

\end{example}

\subsection{Geometry of atoms}

We fix $\lambda\in X_+$. Recall that $\label{leqlambda} \{\leq\lambda\}=(\lambda+\bbZ\Phi) \cap \Conv(W\cdot \lambda).$
In our situation, the convex hull $\Conv(W\cdot \mu)$ is an octagon with vertices as in \Cref{figoctagon}. We can make the actual conditions more explicit.

\begin{figure}[hbt!]
\begin{center}
\begin{tikzpicture}[scale=0.39]
\begin{pgfonlayer}{nodelayer}
\node [style=none] (0) at (3, 5) {};
\node [style=none] (1) at (5, 3) {};
\node [style=none] (2) at (5, -3) {};
\node [style=none] (3) at (-3, 5) {};
\node [style=none] (4) at (-5, 3) {};
\node [style=none] (5) at (-5, -3) {};
\node [style=none] (6) at (-3, -5) {};
\node [style=none] (7) at (3, -5) {};
\node [style=none] (8) at (2, 4.25) {};
\node [style=none] (9) at (2, 4.25) {};
\node [style=none] (10) at (2, 4.25) {};
\node [style=none] (11) at (4, 5.75) {\small{$\lambda=(\lambda_1,\lambda_2)$}};
\node [style=none] (13) at (-10.2, 3.5) {      \small{ $s_2s_1\lambda = (\lambda_1+2\lambda_2,-\lambda_1-\lambda_2)$}};
\node [style=none] (14) at (-9.5, -3.4) {\small{$s_2s_1s_2\lambda=(\lambda_1,-\lambda_1-\lambda_2)$}};
\node [style=none] (15) at (-3, -5.75) {\small{$w_0\lambda=(-\lambda_1,-\lambda_2)$}};
\node [style=none] (16) at (7.5, -5.75) {\small{$s_1s_2s_1\lambda=(-\lambda_1-2\lambda_2,\lambda_2)$}};
\node [style=none] (17) at (10.1, -3.4) {\small{$s_1s_2\lambda=(-\lambda_1-2\lambda_2,\lambda_1+\lambda_2)$}};
\node [style=none] (18) at (9, 3.5) {\small{$s_1\lambda=(-\lambda_1,\lambda_1+\lambda_2)$}};
\node [style=none] (19) at (-3.5, 5.75) {\small{$s_2\lambda=(\lambda_1+2\lambda_2,-\lambda_2)$}};
\end{pgfonlayer}
\begin{pgfonlayer}{edgelayer}
\draw (4.center) to (5.center);
\draw (5.center) to (6.center);
\draw (6.center) to (7.center);
\draw (7.center) to (2.center);
\draw (2.center) to (1.center);
\draw (1.center) to (0.center);
\draw (0.center) to (3.center);
\draw (3.center) to (4.center);
\end{pgfonlayer}
\end{tikzpicture}
\end{center}
\caption{The $W$-orbit and the convex hull of $\lambda$}\label{figoctagon}
\end{figure}
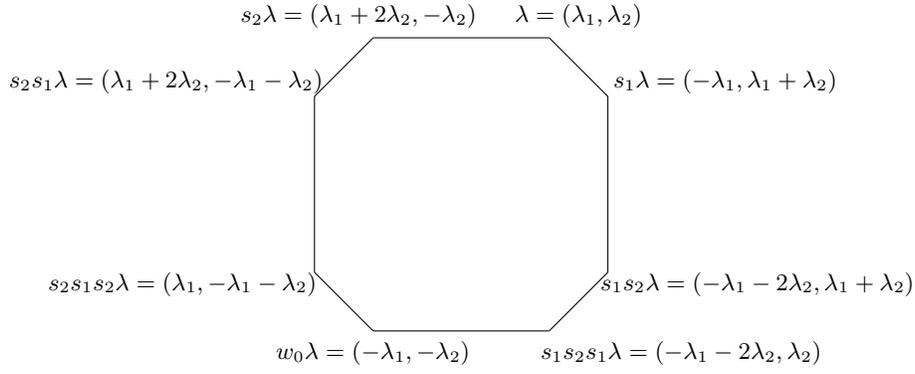

\begin{lemma}
\label{octineq}
We have $\mu\leq \lambda$ if and only if $\mu_1\equiv \lam_1 \pmod{2}$ and the following inequalities hold:
\begin{alignat*}{4}
    &-\lam_1-2\lam_2 	&&\leq \mu_1 &&=\langle \mu,\alpha_1^\vee\rangle &&\leq \lam_1+2\lam_2\\
    &-\lam_1-\lam_2 	&&\leq \mu_1+\mu_2 &&=\langle \mu,\alpha_{12}^\vee\rangle &&  \leq  \lam_1+\lam_2\\
    &-\lam_1-\lam_2	&&\leq\mu_2 &&=\langle \mu,\alpha_2^\vee\rangle&&  \leq \lam_1+\lam_2\\
    &-\lam_1-2\lam_2	&&\leq\mu_1+2\mu_2&&=\langle \mu,\alpha_{21}^\vee\rangle &&   \leq \lam_1+2\lam_2.
\end{alignat*}

\end{lemma}
\begin{proof}
	 It is easy to see that $\mu \equiv \lambda \pmod{\bbZ \Phi}$ if and only if $\mu_1=\lambda_1$. The inequalities can be easily deduced from \Cref{figoctagon}
\end{proof}

We introduce now some helpful quantities which evaluate the distance of a weight $\mu$ from the walls of $\Conv(W\cdot \lambda)$.

\begin{definition}\label{affphidef}
	For $i\in \{1,2,21,12\}$, let $\affphi_i(\mu,\lambda)$ be the maximum integer $k$ such that $\mu-k\alpha_i\leq \lambda$. 
\end{definition}

\begin{lemma}\label{affphicompute}
	Let $\mu\leq \lambda$. We have
	\begin{enumerate}
		\item 
		$
		\affphi_{21}(\mu,\lambda)= \lam_2+\mu_2+\min\left(\lam_1,\frac{\lam_1+\mu_1}{2},\lam_1+\mu_1\right) $   
		\item $		\affphi_{12}(\mu,\lambda)= \frac{\lambda_1+\mu_1}{2} +\min\left(\lambda_2+\mu_2,\lfl\frac{\lambda_2+\mu_2}{2}\rfl,\lambda_2\right)$ 
		
		%
		\item $\affphi_{2}(\mu,\lambda):=\frac{\lambda_1-\mu_1}{2} +\min\left(\lambda_2+\mu_1+\mu_2,\lfl\frac{\lambda_2+\mu_1+\mu_2}{2}\rfl,\lambda_2\right).$	
	\end{enumerate}
	
\end{lemma}
\begin{proof}
	We prove only the first statement, since the other two are analogous. 
	Consider the maximal $x\in \mathbb{R}_{\geq 0}$ such that $\nu:=\mu-x\alpha_{21}\in \Conv(W\cdot \lambda)$. Then $\mu-x\alpha_{21}$ belongs to the boundary of $\Conv(W\cdot \lambda)$ and $\affphi_{21}(\mu,\lambda)=\lfloor x \rfloor$. 
	
	We have $(\nu_1,\nu_2)=(\mu_1,\mu_2-x)$, hence by \Cref{octineq} the following inequalities three inequalities hold
	\begin{align*}
	-\lam_1-\lam_2 \leq &\; \mu_1 +\mu_2-x\\
	-\lam_1-\lam_2\leq &\;\mu_2-x\\
	-\lam_1 -2 \lam_2 \leq&\; \mu_1+2\mu_2-2x	
	\end{align*} 
	and since we are on the boundary at least one of them must be an equality. It follows that
\begin{align*}
x &= \min(\mu_1+\mu_2+\lam_1+\lam_2,\lam_1+\lam_2+\mu_2,\frac{\lam_1+\mu_1}{2} + \lam_2+\mu_2)\\
  &= \lam_2+\mu_2+\min(\lam_1,\frac{\lam_1+\mu_1}{2},\lam_1+\mu_1).\qedhere
\end{align*}
	
%
%
%
%
%
%
\end{proof}

\subsection{Twisted Reflection Groups}
\label{sec:twistedreflection}


For $k\geq 0$ consider the reflection subgroup \[ W^{k}:=y_{k}Wy_{k}^{-1}\subset \affW.\]
Note that for any $k$ we have 
$W^{k+1}=t_{k+1}W^k t_{k+1}$.

\begin{lemma}
\label{twistedweylgroups}
For any $M>0$ we have 
$W^{4M-3}=W^{4M-2}=W^{4M-1}=W^{4M}$. Moreover, the reflections in $W^{4M}$ correspond to the roots
 \[\{\alpha_1^\vee,M\delta-\alpha_2^\vee,M\delta-\alpha_{12}^\vee,2M\delta-\alpha_{21}^\vee\}.\]
\end{lemma}
\begin{proof}
We check this by induction. Recall that for any $M>0$, $t_{4M-3}$, $t_{4M-2}$, $t_{4M-1}$, and $t_{4M}$ are the reflections corresponding to the roots $(2M-1)\delta-\alpha_{21}^\vee$,
$M\delta-\alpha_{12}^\vee$, $2M\delta-\alpha_{21}^\vee$, and $M\delta-\alpha_2^\vee$, respectively. 

Recall that for any $M\in \bbN$ we have
\[W^{4M-3}=t_{4M-3}W^{4M-4}t_{4M-3}.\] 

By induction, the reflections in $W^{4M-4}$ correspond to the roots $\alpha^\vee_{1}, (M-1)\delta -\alpha_{2}^\vee , (M-1)\delta -\alpha_{12}^\vee$, and $2(M-1)\delta - \alpha_{21}^\vee$.

The claim follows since
\begin{align*}
   s_{(2M-1)\delta-\alpha_{21}^\vee}(\alpha_1^\vee)&=\alpha_1^\vee\\
s_{(2M-1)\delta-\alpha_{21}^\vee}((M-1)\delta-\alpha_2^\vee)&=-M+\alpha_{12}^\vee\\ s_{(2M-1)\delta-\alpha_{21}^\vee}((M-1)\delta-\alpha_{12}^\vee)&=-M+\alpha_{2}^\vee\\ s_{(2M-1)\delta-\alpha_{21}^\vee}(2(M-1)\delta-\alpha_{21}^\vee)&=-2M\delta+\alpha_{21}^\vee ,
\end{align*}
therefore $t_{4M-i} \in W^{4M-3}$ for $0 \leq i \leq 3$, which implies that 

\[W^{4M}=W^{4M-1}=W^{4M-2}=W^{4M-3}. \qedhere \]
\end{proof}

The set of reflections in $W^{4M}$ is $\{s_1, v_M, q_M, r_M\}$, where $v_M$, $q_M$ and $r_M$ are the reflection  corresponding to the roots $M\delta-\alpha_2^\vee,M\delta-\alpha_{12}^\vee,2M\delta-\alpha_{21}^\vee$, as depicted in \Cref{Wmu}. More explicitly, we have
 \begin{alignat}{3}\label{vMmu}
 &v_M\mu &&=\mu-(\mu_2+M)\alpha_{2} &&=(\mu_1+2\mu_2+2M,-\mu_2-2M) \\
  &q_M\mu &&= \mu-(\mu_1+\mu_2+M)\alpha_{12} &&=(-\mu_1-2\mu_2-2M,\mu_2)\\
 &	r_M\mu  &&=\mu-(\mu_1+2\mu_2+2M)\alpha_{21} &&= (\mu_1,-\mu_1-\mu_2-2M)
 \end{alignat}
We also have $q_M=s_1v_Ms_1$ and $r_M=v_Ms_1v_M$.

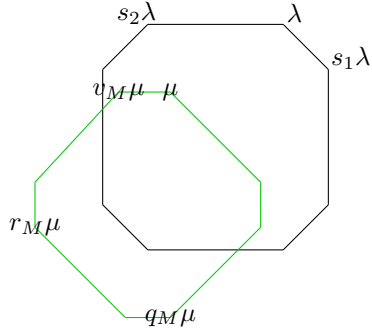
\begin{figure}
	\begin{center}
	\begin{tikzpicture}[scale=0.3]
	\begin{pgfonlayer}{nodelayer}
	\node [style=none] (20) at (-2, 2) {$\mu$};
	\node [style=none] (21) at (-4.3, 2) {$v_M\mu$};
	\node [style=none] (22) at (2, -2) {};
	\node [style=none] (23) at (2, -4) {};
	\node [style=none] (24) at (-2, -8) {$q_M\mu$};
	\node [style=none] (25) at (-4, -8) {};
	\node [style=none] (27) at (-8, -4) {$r_M\mu$};
	\node [style=none] (28) at (-8, -2) {};
	
	\node [style=none] (0) at (3, 5) {};
	\node [style=none] (1) at (5, 3) {};
	\node [style=none] (2) at (5, -3) {};
	\node [style=none] (3) at (-3, 5) {};
	\node [style=none] (4) at (-5, 3) {};
	\node [style=none] (5) at (-5, -3) {};
	\node [style=none] (6) at (-3, -5) {};
	\node [style=none] (7) at (3, -5) {};
	\node [style=none] (8) at (2, 4.25) {};
	\node [style=none] (9) at (2, 4.25) {};
	\node [style=none] (10) at (2, 4.25) {};
	\node [style=none] (11) at (3.5, 5.5) {$\lambda$};
	\node [style=none] (18) at (6, 3.5) {$s_1\lambda$};
	\node [style=none] (19) at (-3.5, 5.5) {$s_2\lambda$};
	\end{pgfonlayer}
	\begin{pgfonlayer}{edgelayer}
	\draw (4.center) to (5.center);
	\draw (5.center) to (6.center);
	\draw (6.center) to (7.center);
	\draw (7.center) to (2.center);
	\draw (2.center) to (1.center);
	\draw (1.center) to (0.center);
	\draw (0.center) to (3.center);
	\draw (3.center) to (4.center);
	\end{pgfonlayer}
	\begin{pgfonlayer}{edgelayer}
	\draw [style=green] (21.center) to (20.center);
	\draw [style=green] (20.center) to (22.center);
	\draw [style=green] (22.center) to (23.center);
	\draw [style=green] (23.center) to (24.center);
	\draw [style=green] (24.center) to (25.center);
	\draw [style=green] (25.center) to (27.center);
	\draw [style=green] (27.center) to (28.center);
	\draw [style=green] (28.center) to (21.center);
	\end{pgfonlayer}
	\end{tikzpicture}
	\end{center}
	\caption{The green octagon is the border of the convex hull of $W^{4M}\cdot \mu$.}
	\label{Wmu}
\end{figure}

We can use the twisted reflection subgroups $W^m$ to describe the set of smaller elements with respect to twistet Bruhat order.

\begin{lemma}\label{leqmu}
Let $\mu\in X$.
\begin{enumerate}
    \item For any $m\geq 0$ we have $\{\leq_{m} \mu\}\subset \Conv(W^m\cdot \mu)$. 
    \item If $\mu_1\geq  0$ and $\mu\leq v_M\mu$, we have
	\[\{\leq_{4M} \mu\} 
	=\Conv(W^{4M} \cdot \mu)\cap (\mu+\bbZ\Phi)=\{\leq_{4M-1} v_M\mu\}\]
\end{enumerate}
\end{lemma}
\begin{proof}
Let $\nu\leq_{m} \mu$. Then $y_{m}^{-1}\nu\leq y_{m}^{-1}\mu$, so $y_{m}^{-1}\nu\in \Conv(W\cdot y_{m}^{-1}\mu)$. This shows the first part.
For the second part, because of \eqref{smallerthanlam},
it is enough to show that $y_{4M}^{-1}\mu=y_{4M-1}^{-1}v_M\mu$ is dominant, since then 
\[\{\leq_{4M} \mu\}=\{\leq_{4M-1} v_M\mu\}=\{\leq y_{4M}^{-1} \mu\} 
	=\Conv(W \cdot y_{4M}^{-1}\mu)\cap (\mu+\bbZ\Phi). \]

\noindent Recall that a weight $\tau \in X$ is dominant if and only if $\tau\geq s_1\tau$ and $\tau\geq s_2\tau$. We have $s_1\mu\leq \mu$, and this is equivalent to $s_1\mu\leq_{4M} \mu$. Moreover, $s_1$ commutes with $y_4$ and therefore also with $y_{4M}$. It follows that $s_1y_{4M}^{-1}\mu=y_{4M}^{-1}s_1\mu\leq  y_{4M}^{-1}\mu$.
	
We have $\mu \leq v\mu$, and this is equivalent to $v\mu\leq_{4M}\mu$, so $y_{4M}^{-1}\mu \geq y_{4M}^{-1}v\mu=y_{4M-1}^{-1}\mu=s_2y_{4M}^{-1}\mu$.
\end{proof}

Recall from \Cref{affphidef} the definition of $\affphi_i(\mu,\lambda)$.

\begin{lemma}\label{l=phi}
Assume that $\mu \leq v_M\mu$. Then we have
\begin{equation}\label{eqtu}v_M\mu \not\leq \lambda \iff M>\affphi_2(\mu,\lam)-\mu_2 \iff  \ell^2_{4M-1}(\mu,\lam)=\affphi_2(\mu,\lam),\end{equation}
\end{lemma}
\begin{proof}
    By \eqref{vMmu} and the definition of $\affphi_2$ we have $v_M\mu\leq \lam$ if and only if $\mu_2+M\leq \affphi_2(\mu,\lam)$.
    It follows from \Cref{leqmu}.2) that $\Arr_{4M}^2(\mu)$ consists precisely of the arrows $(\mu-k \alpha_2\raw \mu)$, with $\mu -k\alpha_2$ lying on the segment between $\mu$ and $v_M\mu$. In other words,  we have
\[ \Arr_{4M}^{2}(\mu)=\{(\mu-k\alpha_2 \raw \mu) \mid 1 \leq k \leq \mu_2 +M\}\]
If $v_M\mu\leq \lam$, then $\Arr_{4M}^2(\mu)=\Arr_{4M}^2(\mu,\lam)$, so \[\ell^2_{4M-1}(\mu,\lam)=\ell^2_{4M}(\mu,\lam)-1=\mu_2+M-1< \affphi_2(\mu,\lam).\] If $v_M\not \leq \lam$ we have \[\Arr_{4M}^2(\mu,\lam)=\{\{(\mu-k\alpha_2) \raw \mu \mid 1 \leq k \leq \affphi_2(\mu,\lam)\}\] and so $\ell^2_{4M-1}(\mu,\lam)=\ell^2_{4M}(\mu,\lam)= \affphi_2(\mu,\lam)$.
\end{proof}

Similarly, we have
\begin{itemize}
    \item $\displaystyle \Arr_{4M-2}^{12}(\mu)=\{(\mu-k\alpha_{12}) \raw \mu \mid 1 \leq k \leq \mu_1+\mu_2 + M\}.$ and if $\mu\leq q_M\mu$ we have 
\begin{equation}\label{eqqu} q_M\mu \not \leq \lambda \iff M> \affphi_{12}(\mu,\lambda)-\mu_1- \mu_2\iff \ell_{4M-3}^{12}(\mu,\lam)=\affphi_{12}(\mu,\lambda)
\end{equation}
\item $\displaystyle\Arr_{4M-1}^{21}(\mu)=\{(\mu-k\alpha_{21}) \raw \mu \mid 1 \leq k \leq \mu_1 + 2\mu_2 +2M\}$ and if $\mu\leq r_M\mu$ we have
\begin{equation}\label{eqru}
    r_M\mu \not \leq \lambda \iff 2M > \affphi_{21}(\mu,\lambda)-\mu_1 -2 \mu_2\iff \ell_{4M-2}^{21}(\mu,\lam)=\affphi_{21}(\mu,\lambda).
\end{equation} 
\end{itemize}


In the following Lemma we describe the Bruhat order on a $W^{4M}$-orbit.

\begin{lemma}\label{bruhatorderWk}
Let $\mu \in X$ and $v_M,r_M,q_M$ as before.
If $\mu < v_M\mu$ and $\mu_1\geq 0$ or if $\mu<q_M\mu$ and $\mu_1\leq 0$, then $v_M\mu\leq r_Mv_M\mu<r_M\mu$ and $q_M\mu < r_M\mu$.
\end{lemma}
\begin{proof}
Assume first $\mu_1\geq 0$ and $\mu< v_M\mu$, so $\mu_2> -M$.
We have $\langle v_M\mu,\alpha_{21}^\vee\rangle=(v_M\mu)_1+2(v_M\mu_2)=\mu_1-2M\geq -2M$, so $r_Mv_M\mu \geq v_M\mu$ by \eqref{bruhatorder}. We have $q_Mr_M =r_Mv_M$ and
$\langle r_M\mu,\alpha_{12}^\vee\rangle=-\mu_1-\mu_2-2M<-M$ so $r_Mv_M\mu<r_M\mu$. Similarly, we have $\mu < q_M\mu \leq v_Mq_M\mu \leq s_1v_Mq_M\mu=r_M\mu$. 
The case $\mu_1\leq 0$ and $\mu<q_M\mu$ is similar.
\end{proof}


\begin{lemma}\label{smallerissmaller}
Let $m>0$ and assume $(t_m\mu)_1 \geq 0$ and $\mu\leq t_m\mu\leq \lam$. Then $t_kt_m\mu \leq t_m \mu$ for all $k\leq m$ corresponding to roots of the form $K\delta-\alpha_2^\vee$. 

Assume instead $\mu_1 \geq 0$ and $\mu\leq t_m\mu\leq \lam$. Then $t_k\mu \leq \lam$ for all $k\leq m$ corresponding to roots of the form $K\delta-\alpha_2^\vee$.
\end{lemma}
\begin{proof}
First we prove the first part of the lemma. By assumption we have $k = 4K$, since $t_{k}$ corresponds to a root of the form $K\delta - \alpha^{\vee}_{2}$. First assume that $m = 4M$, so $t_m=s_{M\delta-\alpha_2^\vee}$. Since $k \leq m$ we have $K \leq M$. By (\ref{bruhatorder}) we have that for $k = 4K \leq m = 4M$,
\begin{align*}
 t_{k} t_{m} \mu &\leq t_{m} \mu \iff
 \langle t_{m} \mu, -\alpha_{2}^\vee \rangle = \mu_{2} +2M > K.
 \end{align*}
 We conclude the proof in this case since by assumption $\mu = t_{m}t_{m}\mu \leq t_{m} \mu $ and $K\leq M$. 
 
 Now we assume $m = 4M-2$, so that $t_m=s_{M\delta-\alpha_{12}^\vee}$. In this case $4K \leq 4M-2$, so in particular $K <M$. We have 
\begin{align*}
 t_{k} t_{m} \mu \leq t_{m} \mu \iff
 \langle t_{m} \mu, -\alpha_{2}^\vee \rangle = -\mu_{2} > K.
 \end{align*}
 Our assumption $\mu\leq t_m\mu$ implies that
 $\langle t_m\mu,-\alpha_{12}^\vee\rangle =\mu_1+\mu_2+2M>M$ and $(t_m\mu)_1\geq 0$ implies $\mu_1+2\mu_2+2M>0$.  Putting them together we obtain:
\begin{align*}
K < M \leq 
\mu_{1}+\mu_{2}+2M  \leq 
\mu_1 +2\mu_2 +2M -\mu_2 \leq -\mu_2.
\end{align*}
which finishes the proof in this case. 

Now we assume $m = 4M-1$, so that $t_m=s_{2M\delta-\alpha_{21}^\vee}$. In this case, we have
\begin{align*}
 t_{k} t_{m} \mu \leq t_{m} \mu \iff
 \langle t_{m} \mu, -\alpha_{2}^\vee \rangle = \mu_{1}+\mu_{2} +2M > K.
 \end{align*}
Our assumption $\mu\leq t_m\mu$ implies that
 $\langle t_m\mu,-\alpha_{21}^\vee\rangle =\mu_1+2\mu_2+4M\geq 2M$ and $\mu_1=(t_m\mu)_1\geq 0$.
Putting them together we obtain:
\[2K < 2M \leq 
\mu_{1}+2\mu_{2}+4M  \leq 
2\mu_1 +2\mu_2 +4M.\]
Finally, assume that $m=4M-3$, so that $t_m=s_{(2M-1)\delta-\alpha_{21}^\vee}$. This case follows by the same argument of the case $m=4M-1$ since we have $2K<2M-1$.

Now we proceed to prove the second part of the lemma, namely that, assuming $\mu_1 \geq 0$ and $\mu \leq t_m \mu \leq\lambda$, then $t_k \mu \leq \lambda$ for  $k=4K \leq m$. We can assume $\mu <t_k\mu$, otherwise the statement is obvious.

The case $m=4M$ is clear since $t_k\mu$ lies on the segment between $t_m\mu$ and $\mu$.
Assume now $m=4M-1$ or $m=4M-3$. In both cases, we have $r_K\mu=t_{k-1}\mu\leq \lam$ since it lies on the segment between $\mu$ and $t_m\mu$. We conclude by \Cref{bruhatorderWk}, since we get $v_K\mu\leq r_Kv_K\mu\leq r_K\mu$.
The last case to consider is $m=4M-2$. Similarly, we have $q_K\mu\leq \lam$ and also $s_1q_K\mu=r_Kv_K\mu\leq \lam$. We conclude again by \Cref{bruhatorderWk} since $v_K\mu\leq r_Kv_K\mu$.

\end{proof}

\subsection{Analysis of \texorpdfstring{$\alpha_{2}$-}{horizontal }edges}

In this section we fix $m+1=4M$ so that $v:=v_M=t_{m+1}$ is the reflection corresponding to the affine root $M\delta -\alpha_2^\vee$, i.e. the reflection over the vertical axis $\{x\mid \langle x,\alpha_2^\vee\rangle = -M\}$. Let $r:=r_M$ and $q:=q_M$.



	


\subsubsection{Sufficient conditions for swappableness}

We assume in this section that $\mu<v\mu \leq \lam$
The goal of this section is to provide a first important constraint on an $\alpha_2$-edge to be swappable (see \Cref{figqmu<lam})

%

\begin{proposition}\label{qmuswappable}
	
	Assume that $q\mu \leq \lambda$. Then $\mu\raw v\mu$ is swappable.
\end{proposition}

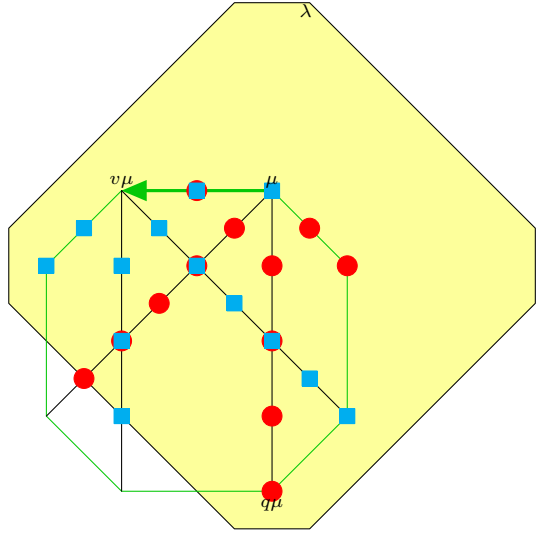
\begin{figure}
\begin{center}
\begin{tikzpicture}[line cap=round,line join=round,>=triangle 45,x=.5cm,y=.5cm]
\draw[style=black] (1,7) -- (-1,7) -- (-7,1) -- (-7,-1) -- (-1,-7) -- (1,-7) -- (7,-1) -- (7,1) -- cycle;
\draw[style=green] (-6,-4) -- (-6,0) -- (-4,2) -- (0,2) -- (2,0) -- (2,-4) -- (0,-6) -- (-4,-6) -- cycle;
\draw (-4,2)-- (0,2);
\draw (0,2)-- (-6,-4);
\draw (0,2)-- (0,-6);
\draw  (-4,2)-- (-4,-6);
\draw  (-4,2)-- (2,-4);
\begin{scriptsize}
\draw[green, very thick, ->] (0,2) to (-4,2);
\draw (.9,6.8) node {$\lam$};
\node [style=blue] at (0,2) {};
\draw (0,2.23) node {$\mu$};
\draw (-4,2.23) node {$v\mu$};
\node [style=red] at (2,0) {};
\node [style=blue] at (2,-4) {};
\node [style=blue] at (-6,0) {};
\node [style=red] at (0,-6) {};
\draw (0,-6.35) node {$q\mu$};
\node [style=red] at (-5,-3) {};
\node [style=red] at (-4,-2) {};
\node [style=blue] at (-4,-2) {};
\node [style=blue] at (-4,-4) {};
\node [style=red] at (0,0) {};
\node [style=red] at (0,-2) {};
\node [style=red] at (0,-4) {};
\node [style=blue] at (0,-2) {};
\node [style=blue] at (-4,0) {};
\node [style=blue] at (-5,1) {};
\node [style=red] at (1,1) {};
\node [style=red] at (-2,2) {};
\node [style=blue] at (-2,2) {};
\node [style=red] at (-1,1) {};
\node [style=red] at (-2,0) {};
\node [style=blue] at (-2,0) {};
\node [style=red] at (-3,-1) {};
\node [style=blue] at (-3,1) {};
\node [style=blue] at (-1,-1) {};
\node [style=blue] at (1,-3) {};
\end{scriptsize}
\end{tikzpicture}
\caption{In this example $q\mu\leq \lam$ and the edge $\mu\raw v\mu$ is indeed swappable.}\label{figqmu<lam}
\end{center}
\end{figure}

We begin with a preliminary computation.
\begin{lemma}\label{qmuineq}
    If $q\mu\leq \lambda$ and $r\mu\not \leq \lambda$, then $-\lambda_1\leq \mu_1 \leq \lambda_1$ and $(v\mu)_2=-\mu_2-2M\leq - \lam_2$.
\end{lemma}
\begin{proof}
Observe that, since $\lam\in X_+$, for any $\nu \in X$ we have $\nu\leq \lam$ if and only if $s_1\nu\leq \lam$. So we also have $s_1r\mu=vq\mu\not\leq \lam$, $s_1q\mu =qr\mu\leq \lam$ and $s_1v \mu = vr\mu\leq \lam$. 	

If $\mu_1> \lam_1$ the line $\{ \mu - x \alpha_{21}\}_{x\in \bbR_{>0}}$ intersects the boundary of $\Conv(W\cdot \lambda)$ in the segment \[[s_2s_1\lam,s_2s_1s_2\lam]\subset H:=\{\nu \in X_\bbR \mid \langle\nu,\alpha_2^\vee\rangle= \langle s_2s_1\lam, \alpha_2^\vee\rangle=-\lam_1-\lam_2\},\] and since $r \mu \notin \operatorname{Conv}(W \cdot \lam)$ we have
	\[\langle r\mu,\alpha_2^\vee\rangle
 <-\lam_1 -\lam_2,\]
However, $qr\mu = s_1q \mu$ lies on the same side of $H$  as $r\mu$, since $\langle s_1q\mu,\alpha_2^\vee\rangle=\langle r\mu,\alpha_2^\vee\rangle$. Therefore, $s_1q\mu \not \in \Conv(W\cdot \lam)$, contradicting our assumption. Similarly, if $\mu_1<-\lam_1$, then we must have 
	\[\langle r\mu,\alpha_{12}^\vee\rangle <-\lam_1-\lam_2,\]
which implies $vr\mu=s_1v\mu \not \in\Conv(W\cdot \lam)$. We conclude theat $-\lam_1\leq \mu_1\leq \lam_1$. 

For the second part, assume that $(v\mu)_2>-\lam_2$, then 
the line $\{v\mu-x\alpha_{12}\}_{x\in \bbR_{>0}}$ intersects the segment \[[s_2s_1s_2\lam,w_0\lam]\subset H':=\{\nu \in X_\bbR \mid \langle \nu, \alpha_{21}^\vee\rangle=\langle w_0\lambda,\alpha_{21}^\vee\rangle=-\lambda_1-\lambda_2\}\] forcing
	$\langle qv\mu,\alpha_{21}^\vee\rangle<-\lam_1-\lam_2$.
	But since $\langle q\mu,\alpha_{21}^\vee\rangle=\langle qv\mu,\alpha_{21}^\vee\rangle$, this would contradict $q\mu \leq \lam$.
\end{proof}

\begin{proof}[Proof of \Cref{qmuswappable}.]

Recall that $\ell_m(\mu)=\ell_m(v\mu)-1$ by \eqref{lmX}. To conclude it is enough to show that \begin{equation}\label{elldiff}
	    \ell_m(\mu)-\ell_m(\mu,\lam) = \ell_m(v\mu)-\ell_m(v\mu,\lam)
	\end{equation} 

The proof is divided in two cases.
 Assume first that $r\mu \leq \lambda$. In this case, since additionally $q\mu \leq \lambda$, the convex hull $\Conv(W^{m+1}\cdot \mu)$ is contained in $\Conv(W\cdot \lam)$ entirely, so 
 $\ell_m(\mu,\lambda)=\ell_m(\mu)$ and $\ell_m(v\mu)=\ell_m(v\mu,\lambda)$.

We can assume now that $r\mu \not \leq \lambda$. 
By \Cref{qmuineq}, we have $-\lam_1\leq \mu_1\leq \lam_1$, which implies that $\lambda_1 + \mu_1 \geq 0$ and $\operatorname{min} (\lambda_1, \frac{\lambda_1 +\mu_1}{2}, \lambda_1 +\mu_1) = \frac{\lambda_1 +\mu_1}{2}$.
It follows from \Cref{affphicompute} that
\[ \affphi_{21}(\mu,\lam)=\mu_2+\lam_2+\frac{\mu_1+\lam_1}{2}.\]
		Since $q\mu \leq \lambda$, we have $\Arr_m^{12}(\mu)=\Arr_m^{12}(\mu,\lam)$ and \[\Arr_m^{21}(\mu)\setminus \Arr_m^{21}(\mu,\lam)=\{(\mu -k\alpha_{21})\raw \mu \mid \affphi_{21}(\mu,\lam)<k\leq \mu_1+2\mu_2+2M\},\] 
 so we get
	\begin{align}\label{Arrmu}
	\ell_m(\mu)-\ell_m(\mu,\lam) =	\ell_m^{21}(\mu)-\ell_m^{21}(\mu,\lam) =& 2M+\mu_1+2\mu_2-\affphi_{21}(\mu,\lambda)\nonumber\\
	=&2M +\mu_2 -\lam_2 +\frac{\mu_1-\lam_1}{2}.
	\end{align}
	 By \Cref{affphicompute}, since $(v\mu)_2\leq -\lam_2$ we have
	\[\affphi_{12}(v\mu,\lam)=\frac{\mu_1+\lam_1}{2}+\lam_2-M.\]
Similarly, since $rv\mu=s_1q\mu\leq \lam$ we have 
	$\Arr_m^{21}(v\mu)=\Arr_m^{21}(v\mu,\lam)$ and
	\[\Arr_m^{12}(v\mu)\setminus \Arr_m^{12}(v\mu,\lam)=\{v\mu-k\alpha_{12}\raw v\mu \mid \affphi_{12}(v\mu,\lam)<k\leq (v\mu)_1+(v\mu)_2+M\}\]
	We get
	\begin{align}\nonumber\label{arrtu}\ell_m(v\mu)-\ell_m(v\mu,\lam)&=\ell_m^{12}(v\mu)-\ell_m^{12}(v\mu,\lam)\\\nonumber&=(v\mu)_1+(v\mu)_2+M-\affphi_{12}(v\mu,\lam)\\
	&=\mu_1+\mu_2+M-\lam_2+M-\frac{\mu_1+\lam_1}{2}.\end{align}
The claimed identity \eqref{elldiff} now follows by comparing \eqref{Arrmu} and \eqref{arrtu}.
\end{proof}

As a consequence, an edge $\mu\raw v\mu$ can only be not swappable if $q\mu\not \leq \lam$. This gives some constraint on the possible location of such weights $\mu$ (see \Cref{startingpoints}).
\begin{lemma}
\label{qmuneqlambda}
	If $q\mu\not\leq \lambda$, then $\mu_1> 0$, $\mu_2<\lambda_2$ and $r\mu \not \leq \lambda$.
\end{lemma}

\begin{figure}
\begin{center}
\begin{tikzpicture}[line cap=round,line join=round,>=triangle 45,x=.15cm,y=.15cm]
\draw[style=black] (4,10) -- (-4,10) -- (-10,4) -- (-10,-4) -- (-4,-10) -- (4,-10) -- (10,-4) -- (10,4) -- cycle;
\fill [red] (4,10)-- (4,4)-- (-7,-7) -- (-10,-4) -- (-10,4) -- (-4,10) -- (4,10) -- cycle;
\fill [orange]  (-4,-4)-- (-7,-7) -- (-10,-4) -- cycle;
\begin{scriptsize}
\draw (-10,-10) -- (10,10) node[above] {$\lam_1=0$}; 
\draw (4,-13) -- (4,13) node[above] {$\mu_2=\lam_2$};
\draw (-13, -4) -- (13,-4) node[right] {$\mu_1+\mu_2=-\lam_2$};    
\end{scriptsize}
\end{tikzpicture}
\end{center}
\caption{By \Cref{qmuneqlambda} the starting point of a non-swappable edge in the $\alpha_2$-direction must lie in the red or in the orange region. We further show in \Cref{mu1<lam1} that actually a starting point of a non-swappable edge can only be in the red region.}\label{startingpoints}
\end{figure}
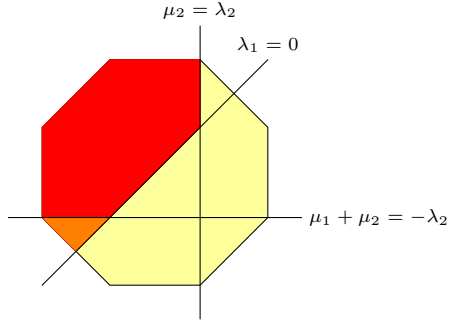

\begin{proof}
	Assume that $q\mu \not \leq \lambda$. Then by \eqref{eqtu} and \eqref{eqqu} we have 
	\[ \affphi_2(\mu,\lambda)\geq \mu_2+M > \affphi_{12}(\mu,\lambda) -\mu_1.\]
	This is equivalent to
	\begin{equation}\label{minconfront}
		\min(\lambda_2+\mu_1+\mu_2,\lfl\frac{\lambda_2+\mu_1+\mu_2}{2}\rfl,\lambda_2)>\min(\lambda_2+\mu_2,\lfl\frac{\lambda_2+\mu_2}{2}\rfl,\lambda_2).\end{equation}
	This forces $\mu_1>0$.  Moreover, we have $\mu_2< \lambda_2$ otherwise both sides of \eqref{minconfront} would be  equal to $\lambda_2$.
	
	Notice that if $q\mu \not \leq \lam$, also $s_1q\mu\not \leq \lam$.
 Moreover,  $r=qs_1q$ and $\langle r\mu , \alpha_{12}^\vee\rangle = -\mu_2 -2M< -M$. By \eqref{bruhatorder}, we conclude that $q\mu <r\mu \not \leq \lambda$.
\end{proof}

\subsubsection{Classification of swappable edges}


\begin{lemma}
Assume $\mu_1\geq 0$. An edge $\mu\raw v\mu$ is swappable if and only if 
\begin{equation}\label{count1}2(\mu_2+M) +\ell_{m}^{12}(v\mu,\lambda)+\ell_{m}^{21}(v\mu,\lambda)=\ell_{m}^{12}(\mu,\lambda)+\ell_{m}^{12}(\mu,\lambda).\end{equation}
\end{lemma}
\begin{proof}
By \Cref{leqmu}, an arrow $(\mu-k\alpha_2\raw v\mu)$ is in $\Arr_m^2(\mu,\lam)$ if and only if $0\leq k< \mu_2 +M$. It follows that $\ell_m^2(\mu)=\ell_m^2(v\mu)-1$.
Moreover, by \eqref{ell1}, we have \[ \ell_{m}^{1}(\mu,\lam) = \ell_{m}^1(v\mu,\lam) - 2(\mu_2+M).\]
The claim now follows directly from \eqref{ellsubdivided}. 
\end{proof}
We now need to estimate carefully $\ell_m^{12}(\mu,\lambda)$ and $\ell_m^{21}(\mu,\lam)$, i.e., we need to characterize the arrows in $\Arr_m^{12}(\mu)$ and $\Arr_m^{21}(\mu)$ whose starting point is contained in $\Conv(W\cdot \lam)$.


We are now ready to classify all swappable $\alpha_2$-edges. We have already seen that it is always swappable if $\mu_1\leq 0$. Now we divide the rest into two cases: $\mu_1\geq \lam_1$ and $0<\mu_1<\lam_1$.
As illustrated in \Cref{figmu1>l1}, in the case $\mu_1\geq \lam_1$ it is sufficient to compare the number of weights below $\mu$ and $v\mu$ in the convex hull of $W\cdot \lam$. We prove now this analytically.

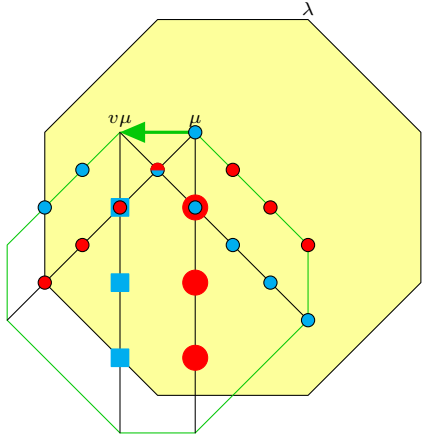
\begin{figure}
\begin{center}
\begin{tikzpicture}[line cap=round,line join=round,>=triangle 45,x=.5cm,y=.5cm]
\draw[style =black] (2,5) -- (-2,5) -- (-5,2) -- (-5,-2) -- (-2,-5) -- (2,-5) -- (5,-2) -- (5,2) -- cycle;
\draw[style=green] (-6,-3) -- (-6,-1) -- (-3,2) -- (-1,2) -- (2,-1) -- (2,-3) -- (-1,-6) -- (-3,-6) -- cycle;
 \draw  (-1,2)-- (-6,-3);
 \draw  (-1,2)-- (-1,-6);
 \draw  (-3,2)-- (-3,-6);
 \draw  (-3,2)-- (2,-3);
 \begin{scriptsize}
\node at (2,5.3) {$\lam$};
 \node at (-1,2.3) {$\mu$};
 \node at (-3,2.3) {$v\mu$};  
 \end{scriptsize}
\draw[green, very thick, ->] (-1,2) to (-3,2);

\draw [fill=cyan] (-1,2) circle (2.5pt);
\draw [fill=red] (2,-1) circle (2.5pt);
\draw [fill=cyan] (2,-3) circle (2.5pt);
\draw [fill=cyan] (-5,0) circle (2.5pt);
\draw [fill=red] (-5,-2) circle (2.5pt);
\halfcirc{-2}{1};
\draw [fill=cyan] (-4,1) circle (2.5pt);
\draw [fill=red] (0,1) circle (2.5pt);
\draw [fill=red] (-4,-1) circle (2.5pt);
\draw [fill=cyan] (0,-1) circle (2.5pt);
\draw [fill=cyan] (1,-2) circle (2.5pt);
\draw [fill=red] (1,0) circle (2.5pt);

\node[style=red] at (-1,-2) {};
\node[style=red] at (-1,-4) {};
\node[style=red] at (-1,0) {};
\node[style=blue] at (-3,-2) {};
\node[style=blue] at (-3,-4) {};
\node[style=blue] at (-3,0) {};
\draw [fill=red] (-3,0) circle (2.5pt);
\draw [fill=cyan] (-1,0) circle (2.5pt);
\end{tikzpicture}
\end{center}
\caption{We have $\mu_1\geq \lam_1$, so to check whether $\mu\raw v\mu$ is swappable we just need to count the weights below $\mu$ and $v\mu$. In this example they are both $3$, hence the edge is swappable.}\label{figmu1>l1}
\end{figure}

\begin{proposition}\label{mu1geqlam1}
Let $\mu$ be such that $\mu_1\geq \lam_1$. Then $\mu\raw v\mu$ is swappable if and only if 
\[ \mu_2\geq -\lam_2+1\text{ and }M \leq \lce \frac{\lam_2-\mu_2 }{2}\rce.\]
\end{proposition}
\begin{proof}

Since $\mu_1\geq \lam_1$ and $\mu\leq \lam$, we have $\mu_2\leq \lam_2$.
	Since $\mu <v\mu$ we have $M+\mu_2 >0$.
We know that if $q\mu \leq \lam$ then $\mu \raw v\mu $ is swappable. In the other direction, if $\mu_2\leq -\lam_2$ or $\mu_2>-\lam_2$ and $M> \lce \frac{\lam_2-\mu_2 }{2}\rce$ then it follows from \eqref{eqqu} that $q\mu\not \leq \lam$. So it is enough to consider the case $q\mu \not \leq \lam$.

We have $\lam_1\leq \mu_1< (v\mu)_1$ and $(v\mu)_2\leq \lam_2$. In this case, we have 
\begin{equation}\label{phi12vu} \affphi_{21}(v\mu,\lambda)=\lam_1 + \lam_2 -\mu_2 -2M = \affphi_{21}(\mu,\lambda)-2(\mu_2+M).
\end{equation}

Combining this with \eqref{count1} and \eqref{eqqu} we get that $\mu\raw v\mu$ is swappable if and only if 
$\affphi_{12}(\mu,\lambda)=\affphi_{12}(v\mu,\lambda)$, which is equivalent to
\[ \min \left(\lam_2+\mu_2,\lfl \frac{ \lam_2+ \mu_2}{2}\rfl\right)=\min\left(\lam_2-M ,\lfl \frac{ \lam_2+ \mu_2}{2}\rfl\right).\]
This equality holds if and only if both minima are achieved at  $\lfl \frac{ \lam_2+ \mu_2}{2}\rfl$, i.e. if
\[ \lam_2 + \mu_2 \geq \lfl \frac{ \lam_2+ \mu_2}{2}\rfl\leq \lam_2 -M.\]
So we have $-\mu_2< M\leq \lce \frac{\lam_2-\mu_2 }{2}\rce$, which is equivalent to $\mu_2\geq -\lam_2+1$ and the claim follows.
\end{proof}

\begin{proposition}
\label{mu1<lam1}
	Let $\mu\leq \lam$ be such that $0<\mu_1< \lam_1$. Then $\mu\raw v\mu\in E^N(\lam)$ if and only if
\begin{equation}\label{mu1>lam1ineq}
M>\frac{\lam_1-\mu_1}{2}+\max\left(-\mu_2,\lce\frac{\lam_2-\mu_2}{2} \rce\right).\end{equation}
\end{proposition}
\begin{proof}
Notice that if the inequality \eqref{mu1>lam1ineq} holds, then $q\mu\not \leq \lam$ by \eqref{eqqu}. Since $\mu\raw v\mu$ is swappable if $q\mu\leq \lam$, we can just assume that $q\mu\not \leq \lam$.

We begin by proving the following inequality.
\begin{claim}\label{claimtmu}
 We have $(v\mu)_2< -\lam_2$.
\end{claim}
\begin{proof}[Proof of the claim.]
We have $(v\mu)_2=-\mu_2-2M$.
If $\mu_2\leq -\lam_2$, then $-\mu_2-2M<-M<\mu_2\leq -\lam_2$. If $\mu_2\geq \lam_2$, we have $-\mu_2-2M<-\mu_2\leq -\lam_2$.

If $-\lam_2<\mu_2<\lam_2$, then we have by \eqref{eqqu} that
\begin{align*} -\mu_2-2M&<\mu_2+2\mu_1-2\affphi_{12}(\mu,\lam)\\
&=-\lam_1+\mu_1+\mu_2-2\lfl\frac{\lambda_2+\mu_2}{2}\rfl\leq -\lam_2\qedhere
\end{align*}
\end{proof}

Assume first $\mu_1<(v\mu)_1\leq \lam_1$, or equivalently that $M \leq \frac{\lam_1-\mu_1}{2}-\mu_2$. Since $q\mu \not \leq \lambda$, we have by \eqref{eqqu} that
\begin{equation}\label{notsmaller}M>\frac{\lam_1-\mu_1}{2}-\min(\lam_2,\lfl \frac{\lam_2-\mu_2}{2}\rfl),\end{equation}
forcing $\mu_2<M-\frac{\lam_1-\mu_1}{2}<-\lam_2$. Now we can easily compute both sides of \eqref{count1} 
and check that they are both equal to $\lam_1+\mu_1+2(\lam_2+\mu_2)$. So $\mu\raw v\mu$ is always swappable.

Assume now $\mu_1<\lam_1<(v\mu)_1$, that is, that $M>\frac{\lam_1-\mu_1}{2}-\mu_2$. Since we assumed that $q\mu \not\leq \lam$, by \eqref{eqqu}, we also have  that 
\begin{equation}\label{notsmaller2}M>\frac{\lam_1-\mu_1}{2}+\lfl\frac{\lam_2-\mu_2}{2}\rfl.\end{equation}

In this case \eqref{count1} is equivalent to 
\begin{equation}\label{mu1<lam1eq} \frac{3\lam_1+\mu_1}{2}+2\lam_2+\mu_2-M= \lam_1+\mu_1+\lam_2+\mu_2+\min\left(\lam_2+\mu_2,\lfl \frac{\lam_2+ \mu_2}{2}\rfl\right)\end{equation}
and we get an equality if and only if 
\begin{equation}\label{Mlimit}
    M = \frac{\lam_1-\mu_1}{2}+\max\left(-\mu_2,\lce\frac{\lam_2-\mu_2}{2}\rce\right)
\end{equation}
Notice that by \eqref{notsmaller2} we cannot have $M< \frac{\lam_1-\mu_1}{2}+\lce\frac{\lam_2-\mu_2}{2}\rce$. The claim now follows.
\end{proof}

We can restate \Cref{mu1<lam1} in more convenient terms.

\begin{corollary}
    Assume $0<\mu_1<\lam_1$. Then $(\mu \raw v \mu)\in E^N(\lam)$ if and only if $\lam_1<(v\mu)_1$ and $q\mu\not \leq \lam$, except when $\lam_2 \not \equiv \mu_2 \pmod{2}$, $\lam_2+\mu_2>0$ and
    \[ M = \frac{\lam_1-\mu_1}{2}+\frc{\lam_2-\mu_2}\]
\end{corollary}
\begin{proof}
As in the proof of \Cref{mu1<lam1}, we know that the only case in which $\lam_1<(v\mu)_1$, $q\mu \not \leq \lam$ and $(\mu \raw v\mu)\in E^S(\lam)$ is for $M$ as in \eqref{Mlimit}.
    Since $\lam_1<(v\mu)_1$ then $M>\frac{\lam_1-\mu_1}{2}-\mu_2$. Since $q\mu\not \leq \lam$ then $M>\frac{\lam_1-\mu_1}{2}+\frc{\lam_2-\mu_2}$.
So the equality in \eqref{Mlimit} is possible if and only if 
$\frc{\lam_2-\mu_2}>\frf{\lam_2-\mu_2}$ and $\frc{\lam_2-\mu_2}>-\mu_2$, i.e. if $\lam_2 \not \equiv \mu_2\pmod{2}$ and $\lam_2+\mu_2\geq1$.   
\end{proof}

\begin{example}
Let $\lam=(3,2)$, $\mu=(1,-1)$ and $M=\frac{\lam_1-\mu_1}{2}-\frc{\lam_2-\mu_2}=3$. As illustrated in \Cref{limitcase}, in this case the edge $\mu\raw v\mu$ is swappable even if $(v\mu)_1>\lam_1$.
\end{example}

\begin{figure}
\begin{center}
\begin{tikzpicture}[line cap=round,line join=round,>=triangle 45,x=.4cm,y=.4cm]
\draw[style=black] (2,5) -- (-2,5) -- (-5,2) -- (-5,-2) -- (-2,-5) -- (2,-5) -- (5,-2) -- (5,2) -- cycle;
\draw[style=green] (-6,-5) -- (-6,-1) -- (-5,0) -- (-1,0) -- (0,-1) -- (0,-5) -- (-1,-6) -- (-5,-6) -- cycle;
\draw  (-1,0)-- (-6,-5);
\draw  (-1,0)-- (-1,-6);
\draw  (-5,0)-- (-5,-6);
\draw  (-5,0)-- (0,-5);
\begin{scriptsize}
\draw [fill=cyan] (-1,0) circle (2.5pt);
\draw [fill=red] (0,-1) circle (2.5pt);
\draw [fill=red] (-4,-3) circle (2.5pt);
\draw [fill=red] (-1,-2) circle (2.5pt);
\draw [fill=cyan] (-5,-2) circle (2.5pt);
\draw[green, very thick, ->] (-1,0) to (-5,0);

\halfcirc{-3}{0}
\halfcirc{-3}{-2}
\halfcirc{-1}{-4}
\node at (2,5.4) {$\lam$};
\node at (-1,0.4) {$\mu$};
\node at( -5.4,0.4) {$v\mu$};

\draw [fill=red] (-2,-1) circle (2.5pt);
\draw [fill=cyan] (-4,-1) circle (2.5pt);
\draw [fill=cyan] (-2,-3) circle (2.5pt);
\draw [fill=cyan] (-1,0) circle (2.5pt);
\draw [fill=cyan] (0,-5) circle (2.5pt);
\end{scriptsize}
\end{tikzpicture}
\end{center}
\caption{The exceptional case in which the edge $\mu\raw v\mu$ is swappable even if $(v\mu)_1>\lam_1$.}
\label{limitcase}
\end{figure}
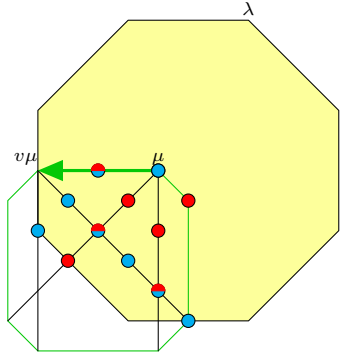

\subsection{Analysis of \texorpdfstring{$\alpha_{12}$-}{vertical }edges}

Assume now that $m+1=4M-2$, so that $q_M:=t_{m+1}$ is the reflection corresponding to the root $M\delta-\alpha_{12}^\vee$, i.e. a reflection over the hyperplane $\{\nu \mid \langle \nu,\alpha_{12}^\vee\rangle =-M\}$. Let $r:=r_M$, $q:=q_M$ and $v:=v_M$ so that the reflections in the reflection subgroup $W^{m+1}$ are $\{s_1,r,q,v\}$. The classificiation of NS-edges in the $\alpha_{12}$-direction can be reduced to the known case of $\alpha_2$-edges, as the \Cref{12notswap} shows.

\begin{proposition}\label{12notswap}
An edge of the form $\mu \ra q\mu$ is swappable if and only if $s_1\mu\ra v s_1\mu$ is swappable.
\end{proposition}
\begin{proof}
It is enough to show that
\begin{equation}\label{ellm+2} \ell_m(\mu,\lam)-\ell_m(q\mu,\lam)=\ell_{m+2}(s_1\mu,\lam)-\ell_{m+2}(vs_1\mu,\lam).
\end{equation}

\begin{claim}
Conjugation by $s_1$ induces a  bijection \[\Arr_m(\mu,\lambda)\setminus \{s_1\mu\raw \mu\}\cong \Arr_m(s_1\mu,\lam)\setminus \{\mu \raw s_1\mu\}\] which sends $(u\mu \raw \mu)$ to $(s_1 u \mu \raw s_1 \mu)$. In particular, we have  $\ell_m(\mu,\lam)=\ell_m(s_1\mu,\lam)+1$ if $\mu_1>0$.
\end{claim}
\begin{proof}[Proof of the claim.]
Notice that, since $m=4M-3$, we have  $\beta^\vee\in N(y_m)$ if and only if $s_1(\beta^\vee)\in N(y_m)$.
If $t\mu<_m\mu$ for $t=s_{N\delta-\alpha^\vee}$ with $\alpha^\vee\in \{\alpha_2^\vee,\alpha_{12}^\vee,\alpha_{21}^\vee\}$, then also $s_1(\alpha^\vee)\in \{\alpha_2^\vee,\alpha_{12}^\vee,\alpha_{21}^\vee\}$
and, since $\langle \mu,\alpha^\vee\rangle = \langle s_1(\mu),s_1(\alpha^\vee)\rangle$ it follows by \eqref{bruhatorder} that also $s_1 t s_1s_1(\mu) <_m s_1(\mu)$.
If $t=s_{N\delta+\alpha_1^\vee}$ with $N\neq 0$, then $s_1ts_1=s_{N\delta-\alpha_1^\vee}$
and we have
\begin{align*}
    t\mu <_m \mu &\iff \sgn(N) \langle \mu, \alpha_1^\vee\rangle >|N|\\
    & \iff \sgn(-N) \langle s_1(\mu), \alpha_1^\vee\rangle>|-N|\\
    & \iff s_1t(\mu)<_m s_1(\mu).\qedhere
\end{align*} 
\end{proof}

We assume now $\mu_1< 0$ and let $\nu=s_1(\mu)$. Notice that $(q\mu)_1<0$.
Since $vs_1=s_1q$, using the claim, \eqref{ellm+2} is equivalent to
\[\ell_m(\nu,\lam)-\ell_m(v\nu,\lam)=\ell_{m+2}(\nu,\lam)-\ell_{m+2}(v\nu,\lam).\]

For any weight $\mu'$, 
the symmetric difference $\Arr_m(\mu',\lam)\vartriangle \Arr_{m+2}(\mu',\lam)$ is contained in $\{q\mu'\raw \mu', r\mu'\raw \mu'\}$ since these are the two only edges which we are possibly reversing. Then the claim follows since we have 
\[ \ell_{m+2}(\nu,\lam)-\ell_m(\nu,\lam) = |\{q\nu,r\nu\} \cap \{\leq \lam\}|= \ell_{m+2}(v\nu,\lam)-\ell_m(v\nu,\lam).\]
In fact, by \Cref{bruhatorderWk} we have $\nu<q\nu$, and $v\nu<rv\nu$ and $qv\leq \lam \iff s_1q\nu =rv\nu<\lam$. Similarly, we have $\nu<r\nu$ and $v\nu <qv\nu$ and $r\nu \leq \lam \iff s_1r\nu =qr\nu < \lam$.

The case $\mu_1\geq 0$ is similar.
\end{proof}

\subsection{Analysis of \texorpdfstring{$\alpha_{21}$-}{diagonal }edges}

We conclude the classification of swappable edges by looking at edges in the $\alpha_{21}$-direction. In this case, the classification is trivial since, as it turns out, all the $\alpha_{21}$-edges are swappable. 
\begin{proposition}
\label{21swappable}
Any edge of the form $\mu \ra \mu-k\alpha_{21}$ is swappable.
\end{proposition}
\begin{proof}
We can assume that $\mu-k\alpha_{21}=s_{(2M-j)\delta -\alpha_{21}^\vee}(\mu)$ with $j=0$ or $j=1$. The root $(2M-j)\delta -\alpha_{21}^\vee$ is the $(4M-1-2j)$-th root occurring in \eqref{reflectionorder}.  Let $m+1=4M-1-2j$ so that $r:=t_{m+1}=s_{(2M-j)\delta -\alpha_{21}^\vee}$.

We have $\ell^1(\mu)=\ell^1(r\mu)$ and $\ell^{21}(\mu,\lam)=\ell^{21}(r\mu,\lam)-1$, so to show that $\mu\raw r\mu$ is swappable it is enough to check that
\begin{equation}\label{ell212}
  \ell^2(\mu,\lam)+\ell^{12}(\mu,\lam)=\ell^2(r\mu,\lam)+\ell^{12}(r\mu,\lam).  
\end{equation}

We consider first the case $j=0$, so $W^m$ is the reflection subgroup with reflections $s_1,q_M,r_M,v_M$. Notice that $r=r_M$.
If $\mu_1\geq 0$ and $v_M\mu \geq\mu$ then
 $\Conv(W^m\cdot \mu)\subset \Conv(W\cdot \lam)$ and the edge $\mu \ra r\mu$ is swappable by the same argument as in the proof of \Cref{qmuswappable}.
 
 Assume now $\mu_1\geq 0$ and $v_M\mu < \mu$.
 Notice that we also have $v_Mr\mu < r\mu$.
 If $q\mu \leq \lam$, then again $\Conv(W^m\cdot \mu)\subset \Conv(W\cdot \lam)$. If $q\mu \not \leq \lam$, then also $s_1q\mu =qr\mu\not \leq \lam$ and we can rewrite \eqref{ell212} as
 \begin{equation}\label{diagonalcount} \ell^2_m(\mu)+\affphi_{12}(\mu,\lam)=\ell^2_m(r\mu)+\affphi_{12}(r\mu,\lam).\end{equation}

\begin{claim}
    We have $\mu_2<-\lam_2$.
\end{claim}
\begin{proof}[Proof of the claim.]
we have $\mu>v_M\mu$ so $\mu_2+M<0$. If $\mu_2\geq -\lam_2$ then $\affphi_{12}(\mu,\lam)=\lam_1+\mu_1+\frf{\lam_2+\mu_2}$ and $q\mu\not\leq \lam$ implies by \eqref{eqqu} that $\mu_1+\mu_2+2M>\lam_1+\lam_2$.  In particular, $\mu_1> \lam_1+\lam_2-\mu_2-2M\geq \lam_1$, so $\affphi_{21}(\mu,\lam)=\lam_2+\mu_2+\lam_1$ but this leads to a contradiction since $r\mu\leq \lam$ and by \eqref{eqru} we get $\mu_1+\mu_2+2M\leq\lam_1+\lam_2$.
\end{proof}

We now go back to the proof of \eqref{diagonalcount}.
 We have $\ell^2_m(\mu)=-\mu_2-M-1$ and $\ell^2_m(r\mu)=\mu_1+\mu_2+M-1$. Since $\mu_2<-\lam_2$ we have by \Cref{affphicompute} that $\affphi_{12}(\mu,\lam)=\frac{\mu_1+\lam_1}{2}+\lam_2+\mu_2$ and $\affphi_{12}(r\mu,\lam)=\frac{\mu_1+\lam_1}{2}+\lam_2-\mu_1-\mu_2-2M$ and the claim easily follows. The case $\mu_1<0$ is analogous.

 Consider now the case $j=1$. The proof here is similar, with the main difference being that the reflections in  $W^m$ are $s_1,q_M,r_M,v_M$ but $r\neq r_M$. In fact, we  have  $r=s_{(2M-1)\delta-\alpha_{21}^\vee}$ and $r_M=s_{2M\delta-\alpha_{21}^\vee}$. In the case $\mu_1\geq 0$ and $v_M\mu\geq \mu$, or $q_M\mu\leq \lam$ then,  similarly to the previous case, we have
 \[ \{\leq_m \mu\} \subset \Conv(W^m\cdot \mu)\setminus \{r_M\mu\} \subset \Conv(W\cdot \lambda).\]
 (In other words, the convex hull of $W^m\cdot \mu$ must lie inside $\Conv(W\cdot \lam)$, except possibly for $r_M\mu$, but this does not matter since $\mu <_m r_M\mu$.)
 It follows that $\mu \raw r\mu$ is swappable. If $\mu>v_M\mu$ and $q_M\mu \not \leq \lambda$ then we conclude by checking the identity \eqref{diagonalcount} as before. The case $\mu_1<0$ is symmetric.
\end{proof}

\subsection{Consequences of the classification}

We can summarize the results from the previous in three sections in the following proposition.

\begin{proposition}
    \label{mu1>0nonswap2}
Assume $\mu_{1} >0$ and let $t$ be a reflection. If the edge $\mu \ra t\mu$ is not swappable, then $t$ corresponds to a root of the form $M\delta -\alpha^{\vee}_{2}$. 
\end{proposition}

\begin{proof}
By \Cref{21swappable}, we know that $\mu \ra t_{m+1}\mu$ cannot be in the $\alpha_{21}$-direction. Since $\mu_1\geq 0$, by \Cref{12notswap} we also know that it cannot be in the $\alpha_{12}$-direction. Hence, the only possibility is that it is an edge in the $\alpha_2$ direction.
\end{proof}

The classification of swappable edges also allows us to easily compare swappable edges for different atoms.

\begin{proposition}\label{corcon}
Let $\mu \leq \lam$ with $\mu_1\geq 0$. Let $m=4M$ so that $t_m=s_{M\delta-\alpha_2^\vee}$ and assume $\mu <t_m\mu \leq \lambda$. Consider the arrow $(\mu\raw t_m\mu)\in E(\lam)$.
\begin{enumerate}
    \item \label{smallerbad}
If $(\mu\raw t_m\mu)\in E^S(\lambda)$, then $(\mu\raw t_m\mu)\in E^S(\lambda+k\varpi_2)$, for any $k\geq 0$.
 \item If $(\mu\raw t_m\mu)\in E^S(\lambda)$, then  for any $k<m$ such that $\mu<t_k\mu\leq \lam$, we also have $(\mu\raw t_k\mu)\in E^S(\lam)$. 
 \item If $(\mu\raw t_m\mu)\in E^N(\lambda)$, then $\lambda - \varpi_2$ is dominant and $\mu\leq \lambda-\varpi_2$.
\end{enumerate}
\end{proposition}
\begin{proof}
The first two statements are clear from the explicit description given in \Cref{mu1geqlam1} and \Cref{mu1<lam1}.

To prove $3.$ first notice that if $\lambda_2 = 0$, by \Cref{qmuneqlambda} and \Cref{mu1<lam1} there can be no non-swappable edges.

\
\
\

We now need to show the inequalities from Lemma \ref{octineq} for $\mu$ and for $\lam' = \lam - \varpi_{2}$. In fact, by \Cref{qmuneqlambda} and \Cref{mu1<lam1}, we only need to establish the following inequalities, since they describe the hyperplanes delimiting the red region in \Cref{startingpoints}:

\begin{enumerate}
\item \label{3.1} $\mu_1 + \mu_2 \leq \lambda_1 + \lambda_2 -1$
\item \label{3.2} $\mu_1 \leq \lambda_1 + 2\lambda_2 - 2$
\item \label{3.3} $\mu_2 \geq -\lambda_1 - \lambda_2 +1$
\end{enumerate}

However note that if $\mu$ lies on either one of the hyperplanes defined by $\mu_1 = \lambda_1 + 2\lambda_2$ or $\mu_2 = -\lambda_1 -\lambda_2 $ or , then $t_m \mu \nleq \lambda$ since $t_m\mu$ is ``on the left'' of $\mu$. Therefore the only inequality we really need to prove is \ref{3.1}.

We assume that the inequality is not true, that is, $\mu$ lies in the hyperplane defined by $\lambda_1 +\lambda_2 = \mu_1 +\mu_2$. In particular, since $\mu \leq \lambda$, such $\mu$ must belong to the ``top side'' of the octagon $\Conv(W\cdot \lam)$ and it must satisfy $\lambda_1 \leq \mu_1 \leq \lambda_1 + 2\lambda_2$ and $-\lam_2\leq \mu_2\leq 
\lam_2$. Also $(t_m\mu)_2$ must lie on the same side of the octagon, so  necessarily then $(t_m \mu)_2 = -\mu_2 -2M \geq -\lambda_2$, which holds if and only if 
\[ M \leq \frac{\lambda_2 -\mu_2}{2}.\]
We get a contradiction, since by  \Cref{mu1geqlam1} the edge $\mu\raw t_m\mu$ is swappable.

\end{proof}

We are now ready to count the number of non-swappable edges.

\begin{definition}
\label{nonswappablenumber}
	For $\mu\leq \lambda$ and $m\in \bbN$, we denote by
	\[\calN_m(\mu,\lambda):=|\{k \leq m \mid  \mu<t_k\mu \leq \lambda \text{ and }\mu\raw t_k\mu\text{ is not swappable}\}|\]
	the number of  non-swappable edges in $E^N(\lambda)$ corresponding to a reflection $t_k$, for $k\leq m$, with starting point $\mu$. Let
\
\[\calN_\infty(\mu,\lam):=\left\lvert\{k \in \bbN \mid  \mu<t_k\mu \leq \lambda \text{ and }\mu\raw t_k\mu\text{ is not swappable}\}\right\rvert.\]
\end{definition}	

\begin{lemma}\label{Nintmuis0}
If $\mu\leq t_m\mu\leq \lam$, then $\calN_m(t_m\mu,\lam)=0$.
\end{lemma}
\begin{proof}
If $(t_m\mu)_1\geq 0$, this follows directly from \Cref{smallerissmaller}. The case $(t_m\mu)<0$ is symmetric.
\end{proof}

Note that since $\Gamma_\lambda$ is a finite graph, we have $\calN_\infty(\mu,\lam)=\calN_m(\mu,\lam)$ for $m$ large enough. If $\mu_1\geq 0$, the only non-swappable edges are in the $\alpha_2$-direction, so in this case we have
		\[\calN_{m}(\mu,\lambda)=\left\lvert\left\{ 1 \leq K\leq \lfl \frac{m}{4}\rfl \mid  
		(\mu\raw s_{K\delta-\alpha_2^\vee}\mu)\in E^N(\lam)\right\}\right\rvert.\]

\begin{proposition}
 
Let $\tilde{M}=\min(\lfl \frac{m}{4}\rfl,-\mu_2+\affphi_2(\mu,\lam))$ and assume $\mu_1\geq 0$. Then we have
\[ \calN_m(\mu,\lambda)=\begin{cases}\tilde{M}+\min(\mu_2,\lfl\frac{\mu_2-\lam_2}{2}\rfl) & \begin{array}{l}\text{if }\mu_1\geq \lam_1\end{array}\\
\tilde{M}+\frac{\mu_1-\lam_1}{2}+
 \min(\mu_2,\lfl\frac{\mu_2-\lam_2}{2}\rfl) 
&		
\begin{array}{l}\text{if }0<\mu_1<\lam_1,\; \mu_2\leq \lam_2\\ \text{ and }\mu_1+\mu_2\geq -\lam_2\end{array}\\

0 & \begin{array}{l}\text{if }\mu_1=0,\; \mu_2\geq \lam_2 \\ \text { or }\mu_1+\mu_2\leq -\lam_2.\end{array}\end{cases}
	\]

\end{proposition}
\begin{proof}
 This follows directly from \Cref{mu1geqlam1,mu1<lam1}.
\end{proof}

If $\mu_1\geq 0$, taking the limit $m\raw \infty$ we get
\begin{equation}\label{Ninf}
    \calN_\infty(\mu,\lambda)=\begin{cases}\affphi_2(\mu,\lam)-\max(0,\frc{\mu_2+\lam_2}) & \begin{array}{l}\text{if }\mu_1\geq \lam_1;\end{array}\\
\affphi_2(\mu,\lam)+\frac{\mu_1-\lam_1}{2}
 -\max(0,\frc{\mu_2+\lam_2}) 
&	\begin{array}{l}	
\text{if }0<\mu_1<\lam_1,\; \mu_2\leq \lam_2 \\\text{ and }\mu_1+\mu_2\geq -\lam_2;\end{array}\\
0 &  \begin{array}{l}\text{if }\mu_1=0,\; \mu_2\geq \lam_2 \\\text { or }\mu_1+\mu_2\leq -\lam_2.\end{array}\end{cases}
\end{equation} 

If $\mu_1<0$ we have $\calN_\infty(\mu,\lam)=\calN_\infty(s_1(\mu),\lam)$.	In particular, we have
\begin{equation}\label{Ninf<0}
    \calN_\infty(\mu,\lambda)=\begin{cases}\begin{aligned}\affphi_{12}(\mu,\lam)+\min\left(0,\frac{-\mu_1-\lam_1}{2}\right)\\-\max\left(0,\frc{\mu_1+\mu_2+\lam_2}\right)\end{aligned} & \begin{array}{l}\text{if }
\mu_1+ \mu_2\leq \lam_2 \\\text{ and }\mu_2\geq -\lam_2; \end{array}\\
0 &  \begin{array}{l}\text{if } \mu_1+\mu_2\geq \lam_2\\\text { or }\mu_2\leq -\lam_2.\end{array}\end{cases}
\end{equation}

A remarkable property is that the number of NS edges gives exactly the correction term in \eqref{eqswapdef} for non-swappable edges.

\begin{proposition}\label{Ncount}For any $\mu\leq \lam$ with $t_{m+1}\mu\leq \lam$, we have 
\[\ell_{m+1}(\mu,\lam)-\ell_{m+1}(t_{m+1}\mu,\lam)-1=\ell_m(\mu,\lam)-\ell_m(t_{m+1}\mu,\lam)+1=\calN_{m+1}(\mu,\lam).\]
\end{proposition}
\begin{proof}
The first equality is clear because $\mu <_m t_{m+1}\mu <_{m+1} \mu$, so we just need to show the second one.

If $\mu \raw t_{m+1}\mu$ is swappable the claim is clear since $\calN_{m+1}(\mu,\lam)=0$ by \Cref{corcon}. We can assume $\mu_1>0$ and $v:=t_{m+1}=s_{M\delta-\alpha_2^\vee}$, since the case $\mu_1<0$ and $t_{m+1}=s_{M\delta-\alpha_{12}^\vee}$ is analogous. 
In this case we have $q_M\mu\not \leq \lam$ and $q_Mv\mu \not \leq \lam$.

Assume first $\mu_1\geq \lam_1$.  Then
\begin{align*}
   \ell_m(\mu,\lam)-\ell_m(v\mu,\lam)+1 &=
      \ell_m^{12}(\mu,\lam)-\ell_m^{12}(v\mu,\lam)=\\
   & =
   \min (\lam_2+\mu_2,\lfl \frac{ \lam_2+ \mu_2}{2}\rfl)-\min(\lam_2-M ,\lfl \frac{ \lam_2+ \mu_2}{2}\rfl)\\
   &= M+\min(\mu_2,\lfl\frac{\mu_2-\lam_2}{2}\rfl)
\end{align*}
In fact, since $\mu\raw v\mu$ is not swappable, and $\lam_2+\mu_2>\lam_2-M$, we have $\lam_2-M>\frf{\lam_2+\mu_2}$.
The same computation also shows that the minimal $K$ such that $\mu\raw s_{K\delta-\alpha_2^\vee}\mu\in E^N(\lam)$ is 
\[K=\max(-\mu_2,\lce\frac{\lam_2-\mu_2}{2}\rce) +1\]
so also 
$\calN_{m+1}(\mu,\lam)=M-K+1=M+
\min(\mu_2,\lfl\frac{\mu_2-\lam_2}{2}\rfl)$.

Assume now $\mu_1<\lam_1$. Recall that in this case we have $\mu_2\geq -\lam_2+1$. As in \eqref{mu1<lam1eq}, we have
\begin{align*}
   \ell_m(\mu,\lam)-\ell_m(t_{m+1}\mu,\lam)+1=  M+\lfl\frac{\mu_2-\lam_2}{2}\rfl+\frac{\mu_1-\lam_1}{2}
\end{align*}
In this case, the minimal $K$ such that $\mu\raw s_{K\delta-\alpha_2^\vee}\mu\in E^N(\lam)$ is 
\[K=\frac{\lam_1-\mu_1}{2}+\frc{\lam_2-\mu_2}+1.\]
and again $\calN_{m+1}(\mu,\lam)=M-K+1$.
\end{proof}

We can also generalize \Cref{Ncount} to the case when $t_{m+1}\mu \not \leq \lambda$. In this case $\ell_m(t_{m+1}\mu,\lambda)$ is not properly defined, so we first need to generalize its definition.

\begin{definition}
Let $\mu \in X$ and assume $\mu_1\geq 0$. 
For $m\in \bbN$ and $i \in \{2,12,21\}$ we define 
\[\affell_m^{i}(\mu,\lam) :=\begin{cases}
    \ell_m^i(\mu,\lam) & \text{if }\mu \leq \lam\\
    \affphi_i(\mu,\lam) & \text{if }\mu\not \leq \lam,
\end{cases}\]
where here $\affphi_i(\mu,\lam)$ is to be interpreted as the function given in \Cref{affphicompute} (notice that $\affphi_{i}$ is not properly defined if $\mu\not \leq \lam$).
Then we define
\[\affell_m(\mu,\lambda):=\ell^1(\mu)+\ell_m^2(\mu)+\affell^{21}_m(\mu,\lam)+\affell^{12}_m(\mu,\lam).\]
\end{definition}
Notice that $\affell_m(t_{m}\mu,\lam)=\ell_m(t_{m}\mu,\lam)$ if $m=4M$ and  $\mu\leq t_{m}\mu\leq \lam$.

\begin{corollary}\label{Ncountgen}
Let $\mu \leq \lam$ and $m=4M$. Then we have 
\[\ell_{m}(\mu,\lam)-\affell_{m}(t_{m}\mu,\lam)-1=\calN_{m}(\mu,\lam)\]
\end{corollary}
\begin{proof}
Let $v:=t_{m}$.
We can assume $v\mu \not \leq \lam$, otherwise the claim follows by \Cref{Ncount}. Notice that this forces $q_M\mu \not \leq \lam$ and $r_M\mu \not \leq \lam$. Notice also that $\ell^2_{m}(v\mu)=-\mu_2-M-1$ and that $\calN_{m}(\mu,\lam)=\calN_{\infty}(\mu,\lam)$.

Assume first $\mu_1\geq \lam_1$. We have
\begin{align*}\ell_{m}(\mu,\lam)-\affell_{m}(v\mu,\lam)-1 &=
\affphi_{12}(\mu,\lam)-\affphi_{12}(v\mu,\lam)+\affphi_2(\mu,\lam)-\ell^2_{m}(v\mu)-1\\
&=M+\min(\mu_2,\frf{\mu_2-\lam_2})+\affphi_2(\mu,\lam)-\mu_2-M\\
&=\affphi_2(\mu,\lam)+\min(0,-\lce\frac{\lam_2+\mu_2}{2}\rce).\end{align*}

Assume now $\mu_1<\lam_1$. In this case, we have
\begin{align*}\ell_{m}(\mu,\lam)-\affell_{m}(v\mu,\lam)-1 &=
M+\lfl\frac{\mu_2-\lam_2}{2}\rfl+\frac{\mu_1-\lam_1}{2}+\affphi_2(\mu,\lam)-\ell^2_{m}(v\mu)-1\\
&=M+\min(\mu_2,\frf{\mu_2-\lam_2})+\affphi_2(\mu,\lam)-\mu_2-M\\
&=\affphi_2(\mu,\lam)+\min(0,-\lce\frac{\lam_2+\mu_2}{2}\rce)+ \frac{\mu_1-\lam_1}{2}.\end{align*}

The claim follows by comparing these formulas with \eqref{Ninf}.
\end{proof}

\subsection{Non-swappable staircases}

In type $A_n$, swapping functions can be defined within a single atom.
Unfortunately, the existence of non-swappable edges in type $C_2$ means that we cannot do the same, causing a relevant increase in complexity. Instead, for every non-swappable edge, the swapping functions we are going to construct in \Cref{sec:swapping} will involve two elements from two different atoms within the same preatom. To determine which are the two atoms involved we need to introduce a new quantity, which we call the elevation of an edge and that measures the height of the maximal staircases of non-swappable edges lying underneath it.


\begin{definition}
\label{def:elevation}
Let $e=(\mu\raw \mu-k\alpha)\in E(\lambda)$ be an edge. We call the \emph{elevation} of $e$, and denote it by $\Omega(e)$, the minimum integer $j\geq 0$ such that $(\mu\raw \mu-(k-j)\alpha)\in E^S(\lambda-j\varpi_2)$. 
\end{definition}

Notice that $\Omega(e)=0$ if and only if $e$ is swappable. 
The elevation of a non-swappable edge is well defined by \Cref{mu1<lam1}.

In the other directions, we need a way to control how many times an element gets swapped with elements from higher atoms.

\begin{definition}
\label{def:nonswappable}
Let $k\geq 0$ and let $\mu \leq \lambda$.  A \emph{staircase of non-swappable edges over $(\mu,\lam)$} (or \emph{NS-staircase}, for short) is a sequence of edges $(e_i)_{1\leq i\leq a}$ such that 
\begin{itemize}
\item $e_i:=(\mu \raw \mu - (n+i)\alpha)\in E^N(\lambda+i\varpi_2)$ for any  $i=1,\ldots, a$.
\item $n = 0$ or $e_0:=(\mu \raw \mu -n\alpha)\in E^S(\lam)$.
\end{itemize}

We define  $\affD_\infty(\mu,\lam)$ to be the length of the longest NS-staircase over $(\mu,\lam)$.
We define  $\affD_m(\mu,\lam)$ to be the length of the longest NS-staircase over $(\mu,\lam)$ where the label of every edge in $e_i$ is a root in $N(y_m)$.
\end{definition}

\begin{example}
Let $\lam=(3,1)$ and $\mu=(3,0)$. Then $e_0:=\mu\raw \mu-\alpha_2=v_{1}\mu$ is a swappable edge, while $e_1:=(\mu\raw \mu-2\alpha_2)\in E^N(\lam+\varpi_2)$ and $e_2:=(\mu\raw \mu-3\alpha_2)\in E^N(\lam+2\varpi_2)$, as illustrated in \Cref{fig:staircase}. So $(e_1,e_2)$ is a NS staircase of $(\mu,\lam)$ and we have $\Omega(e_2)=2$, $\Omega(e_1)=1$ and $\Omega(e_0)=0$.

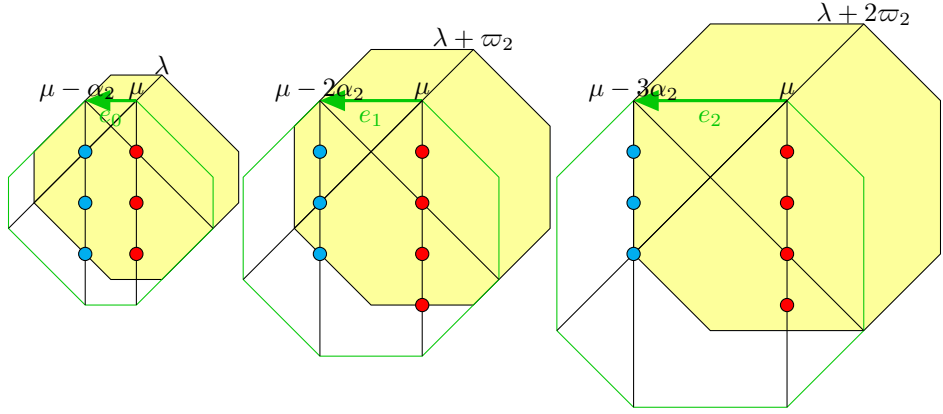
\begin{figure}
\begin{center}
\begin{tikzpicture}[x=0.17cm,y=0.17cm]
\draw[black] (2,8) -- (-2,8) -- (-8,2) -- (-8,-2) -- (-2,-8) -- (2,-8) -- (8,-2) -- (8,2) -- cycle;
\draw[green] (-10,-4) -- (-10,0) -- (-4,6) -- (0,6) -- (6,0) -- (6,-4) -- (0,-10) -- (-4,-10) -- cycle;
\draw  (0,6)-- (-10,-4);
\draw  (0,6)-- (0,-10);
\draw  (-4,6)-- (-4,-10);
\draw  (-4,6)-- (6,-4);
\draw  (-8,-2)-- (2,8);
\draw[green, very thick, -triangle 45] (0,6)  to node[below,color=mygreen] {$e_0$} (-4,6);
\node at (2,8.8) {$\lam$};
\node at (0,6.8) {$\mu$};
\node at (-4.6,6.8) {$\mu-\alpha_2$};

\draw [fill=cyan] (-4,-6) circle (2.5pt);
\draw [fill=cyan] (-4,2) circle (2.5pt);
\draw [fill=red] (0,2) circle (2.5pt);
\draw [fill=red] (0,-2) circle (2.5pt);
\draw [fill=cyan] (-4,-2) circle (2.5pt);
\draw [fill=red] (0,-6) circle (2.5pt);
\begin{scope}[xshift=3.8cm]
\draw[black] (4,10) -- (-4,10) -- (-10,4) -- (-10,-4) -- (-4,-10) -- (4,-10) -- (10,-4) -- (10,4) -- cycle;
\draw[green] (-14,-8) -- (-14,0) -- (-8,6) -- (0,6) -- (6,0) -- (6,-8) -- (0,-14) -- (-8,-14) -- cycle;
\draw[green, very thick, -triangle 45] (0,6) to node[below,color=mygreen] {$e_1$} (-8,6);
\draw  (0,6)-- (-14,-8);
\draw  (0,6)-- (0,-14);
\draw  (-8,6)-- (-8,-14);
\draw  (-8,6)-- (6,-8);
\draw  (-10,-4)-- (4,10);
\node at (4,10.8) {$\lam+\varpi_2$};
\node at (0,6.8) {$\mu$};
\node at (-8,6.8) {$\mu-2\alpha_2$};
\draw [fill=cyan] (-8,-6) circle (2.5pt);
\draw [fill=cyan] (-8,-2) circle (2.5pt);
\draw [fill=red] (0,-2) circle (2.5pt);
\draw [fill=red] (0,2) circle (2.5pt);
\draw [fill=cyan] (-8,2) circle (2.5pt);
\draw [fill=red] (0,-6) circle (2.5pt);
\draw [fill=red] (0,-10) circle (2.5pt);
\end{scope}
\begin{scope}[xshift=8.65cm]
\draw[black] (6,12) -- (-6,12) -- (-12,6) -- (-12,-6) -- (-6,-12) -- (6,-12) -- (12,-6) -- (12,6) -- cycle;
\draw[green] (-18,-12) -- (-18,0) -- (-12,6) -- (0,6) -- (6,0) -- (6,-12) -- (0,-18) -- (-12,-18) -- cycle;
\draw[green, very thick, -triangle 45] (0,6) to node[below,color=mygreen] {$e_2$} (-12,6);
\draw  (0,6)-- (-18,-12);
\draw  (0,6)-- (0,-18);
\draw  (-12,6)-- (-12,-18);
\draw  (-12,6)-- (6,-12);
\draw  (-12,-6)-- (6,12);
\node at (6,12.8) {$\lam+2\varpi_2$};
\node at (0,6.8) {$\mu$};
\node at (-12,6.8) {$\mu-3\alpha_2$};
\draw [fill=cyan] (-12,-6) circle (2.5pt);
\draw [fill=red] (0,-6) circle (2.5pt);
\draw [fill=red] (0,2) circle (2.5pt);
\draw [fill=cyan] (-12,2) circle (2.5pt);
\draw [fill=red] (0,-2) circle (2.5pt);
\draw [fill=red] (0,-10) circle (2.5pt);
\draw [fill=cyan] (-12,-2) circle (2.5pt);
\end{scope}
\end{tikzpicture}
\end{center}
\caption{The edge $e_0$ is swappable while $e_1$ and $e_2$ are not. To check this, since $\mu_1\geq \lam_1$, as explained in \Cref{mu1geqlam1}, it is enough to compare the number of weight in the convex hull lying below $\mu$ and $v\mu$.} 
\label{fig:staircase}
\end{figure}

Moreover, as illustrated in \Cref{staircasemax}, the staircase $(e_1,e_2)$ cannot be extended, since $\mu-4\alpha_2\not \leq \lam+3\varpi_2$. Hence, we have $\affD_\infty(\mu,\lam)=2$.

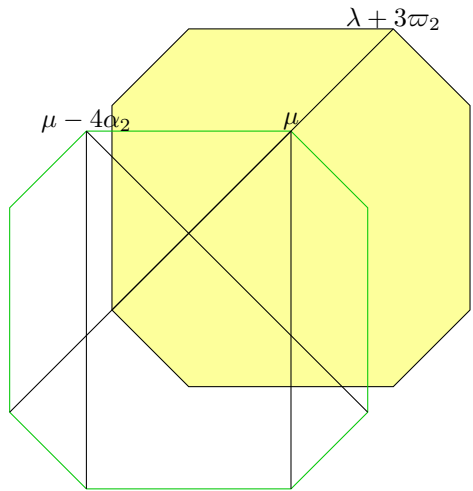
\begin{figure}
\begin{center}
\begin{tikzpicture}[x=0.17cm,y=0.17cm]
\draw[black] (8,14) -- (-8,14) -- (-14,8) -- (-14,-8) -- (-8,-14) -- (8,-14) -- (14,-8) -- (14,8) -- cycle;
\draw[green] (-22,-16) -- (-22,0) -- (-16,6) -- (0,6) -- (6,0) -- (6,-16) -- (0,-22) -- (-16,-22) -- cycle;
\draw  (0,6)-- (-22,-16);
\draw  (0,6)-- (0,-22);
\draw  (-16,6)-- (-16,-22);
\draw  (-16,6)-- (6,-16);
\draw  (-14,-8)-- (8,14);
\node at (8,14.8) {$\lam+3\varpi_2$};
\node at (0,6.8) {$\mu$};
\node at (-16,6.8) {$\mu-4\alpha_2$};
\end{tikzpicture}
\end{center}
\caption{We have $\mu-4\alpha_2\not \leq \lam+3\varpi_2$ and the NS staircase $(e_1,e_2)$ from \Cref{fig:staircase} cannot be extended.}
\label{staircasemax}
\end{figure}

\end{example}

\begin{lemma}\label{staircaseunique}
There exists at most one non-empty NS-staircase over $(\mu,\lam)$.
\end{lemma}
\begin{proof}
Assume that the there are two non-empty NS-staircases of the form $\mu\raw \mu-(n+i)\alpha\in E^N(\lam+i\varpi_2)$ and $\mu\raw \mu-(n'+i)\beta\in E^N(\lam+i\varpi_2)$. Now, if $\mu_1> 0$, by \Cref{mu1>0nonswap2} we have $\alpha=\beta=\alpha_2$ and if $\mu_1<0$ we have $\alpha=\beta=\alpha_{12}$, so in particular $\alpha=\beta$.

We can assume $n'<n$. Since $\mu \raw \mu - n\alpha\in E^S(\lam)$, by \Cref{corcon}.1), we have that $\mu\raw \mu -n\alpha\in E^S(\lam+\varpi_2)$. With this and  \Cref{corcon}.2), we get $(\mu\raw \mu -(n'+1)\alpha)\in E^S(\lam+\varpi_2)$. Our second NS-staircase must therefore be empty. 
\end{proof}

\begin{lemma}\label{ifN=0thenD=0}
If $\calN_m(\mu,\lam)=0$ and $\mu <t_{m}\mu\leq \lam$, then also $\affD_m(\mu,\lam)=0$.
\end{lemma}
\begin{proof}
Assume that $\mu_1>0$.
If $\mu\raw t_k\mu\in E^S(\lam)$ for $k\leq m$, then also $\mu \raw t_k\mu \in E^S(\lam+\varpi_2)$ by \Cref{corcon}. If $t_k\mu\not \leq \lam$ and $(\mu\raw t_k\mu)\in E^N(\lam+\varpi_2)$, then $t_k=s_{K\delta-\alpha_2^\vee}$. But this cannot happen by \Cref{smallerissmaller}.

The case $\mu_1<0$ is symmetric.
\end{proof}

\begin{proposition}\label{Dcount}
 If $\mu_1>0$ we have
	\[ \affD_\infty(\mu,\lambda)=
	\begin{cases}
	\max(0,\min(\lam_1,\mu_1)-1) &\begin{array}{l}\text{if }-\lam_2\leq \mu_2\leq \lam_2 \\\text{ and }\mu_2\not\equiv\lam_2 \pmod{2};\end{array}\\
\max(0,\min(\mu_1,\lam_1)+\lam_2+\mu_2) &\begin{array}{l}\text{if } \mu_2<-\lam_2;\end{array}\\
0 &\begin{array}{l}\text{otherwise}.\end{array}
\end{cases}
\]
\end{proposition}
\begin{proof}
 Let $(e_i)_{1\leq i \leq a}=(\mu\raw v_{M+i}\mu)_{1 \leq i \leq a}$ be a non-empty maximal $NS$-staircase over $(\mu,\lam)$ with $M\geq -\mu_2$.

Assume first $\mu_1\geq \lam_1$. We have $e_1=(\mu\raw v_{M+1}\mu)\in E^N(\lam+\varpi_2)$, so by \Cref{mu1geqlam1} we get $\mu_2<-\lam_2$ or $M+1> \frc{\lam_2+1-\mu_2}$. We have either $v_M\mu=\mu$  or $e_0=(\mu\raw v_M\mu)\in E^S(\lam)$.
In the first case we get $-M=\mu_2$. In the second case we have $\mu_2\geq -\lam_2+1$, $M\leq \frc{\lam_2-\mu_2}$ and $M+1> \frc{\lam_2+1-\mu_2}$, so the only possibility is \[M=\frc{\lam_2-\mu_2}=\frc{\lam_2-\mu_2+1},\]
which also implies $\lam_2\not\equiv \mu_2 \pmod{2}$.

Assume further that $\mu_2<-\lam_2$.
From the discussion above we must have $M=-\mu_2$. It is easy to check that for any $k\geq 1$ we have $e_k\in E^N(\lam+k\varpi_2)$ if  $v_{M+k}\mu \leq \lam+k\varpi_2$ and that \[v_{M+k}\mu \leq \lam+k\varpi_2\iff \frf{\lam_1+\lam_2+\mu_2-k} \geq 0\]
so we get $\affD_{\infty}(\mu,\lam)=\max(0,\lam_1+\lam_2+\mu_2)$.

Assume now $\mu_2\geq -\lam_2$. If $\lam_2\not\equiv \mu_2 \pmod{2}$ then $\affD_\infty(\mu,\lam)=0$. If $\lam_2\not\equiv \mu_2 \pmod{2}$ we must have $M=\frc{\lam_2-\mu_2}$.
Since $v_M\mu\leq \lam$, from \eqref{eqtu} we get
\[ \frc{\lam_2-\mu_2}\leq \frf{\lam_1+\lam_2-\mu_2},\]
so this is possible only if $\lam_1 > 0$.
It easy to check that for any $k\geq 1$ if $v_{M+k}\mu \leq \lam + k\varpi_2$, then also $(\mu\raw v_{M+k}(\mu))\in E^N(\lam+k\varpi_2)$. Moreover, from \eqref{eqtu} $v_{M+k}\leq \lam+k\varpi_2$ we see that is equivalent to
\[\frc{\lam_2-\mu_2}\leq \frf{\lam_1+\lam_2-\mu_2-k}\]
which is true if and only if $k\leq \lam_1-1$. Hence $\affD_\infty(\mu,\lam)=\max(0,\lam_1-1)$.

The proof in the case  $0<\mu_1 < \lam_1$ is similar.
Since $e_1\in E^N(\lam+\varpi_2)$ we have $M+1>\frac{\lam_1-\mu_1}{2}+\max(-\mu_2,\frc{\lam_2-\mu_2+1})$. We have either $v_M\mu=\mu$ or $(\mu\raw v_M\mu)\in E^S(\lam)$. However, the first case is not possible because 
\[M+1=-\mu_2+1\leq \frac{\lam_1-\mu_1}{2}+\max(-\mu_2,\frc{\lam_2-\mu_2+1}).\]
In the second case, we have $M\leq \frac{\lam_1-\mu_1}{2}+\max(-\mu_2,\frc{\lam_2-\mu_2})$, which forces 
\begin{equation}\label{maxequality}
\max(-\mu_2,\frc{\lam_2-\mu_2})=(-\mu_2,\frc{\lam_2-\mu_2+1})
\end{equation}
 and $M=\frac{\lam_1-\mu_1}{2}+\max(-\mu_2,\frc{\lam_2-\mu_2}).$
The equality in \eqref{maxequality} can only occur if $\mu_2< -\lam_2$ or if $\mu_2\geq -\lam_2$ and $\lam_2\not \equiv \mu_2 \pmod{2}$.

Assume now $\mu_2<-\lam_2$. Then $M=\frac{\lam_1-\mu_1}{2}-\mu_2$. It is easy to check that for any $k\geq 1$ we have $e_k\in E^N(\lam+k\varpi_2)$ if  $v_{M+k}\mu \leq \lam+k\varpi_2$ and that \[v_{M+k}\mu \leq \lam+k\varpi_2\iff \frf{\mu_1+\lam_2+\mu_2-k} \geq 0\]
so we get $\affD_{\infty}(\mu,\lam)=\max(0,\mu_1+\lam_2+\mu_2)$.

Finally assume $\mu_2\geq -\lam_2$. If $\lam_2\not\equiv \mu_2 \pmod{2}$ then $\affD_\infty(\mu,\lam)=0$. If $\lam_2\not\equiv \mu_2 \pmod{2}$ we must have $M=\frac{\lam_1-\mu_1}{2}+\frc{\lam_2-\mu_2}$.
It easy to check that for any $k\geq 1$ if $v_{M+k}\mu \leq \lam + k\varpi_2$, then also $(\mu\raw v_{M+k}(\mu))\in E^N(\lam+k\varpi_2)$. Moreover, from \eqref{eqtu} $v_{M+k}\leq \lam+k\varpi_2$ we see that is equivalent to $k+1\leq \mu_1$.
 Hence $\affD_\infty(\mu,\lam)=\max(0,\mu_1-1)$.
\end{proof}

\begin{corollary}\label{Dinf<0}
If $\mu_1<0$ we have 
\begin{align*}\affD_\infty(\mu,\lambda) =\begin{cases}
	\max(0,\min(\lam_1,-\mu_1)-1) &\begin{array}{l}
 \text{if }-\lam_2\leq \mu_1+\mu_2\leq \lam_2 \\\text{and } \mu_1+\mu_2\not\equiv\lam_2 \;(\mathrm{mod }\;2);\end{array}\\
\max(0,\min(-\mu_1,\lam_1)+\lam_2+\mu_1+\mu_2) \hspace{-10pt}&\begin{array}{l}\text{if } \mu_1+\mu_2<-\lam_2;\end{array}\\
0 &\begin{array}{l}\text{otherwise}.\end{array}\end{cases}\end{align*}
\end{corollary}
\begin{proof}
    This immediately follows from \Cref{Dcount}, since by symmetry (cf. \Cref{12notswap}) we have $\affD_{\infty}(\mu,\lam)=\affD_{\infty}(s_1(\mu),\lam)$.
\end{proof}

Suppose than $T\in \calA(\lamk)\subset \calP(\lam)$.
In our applications in \Cref{sec:swapping}, we are only interested in NS staircase over $(\wt(T),\lamk)$ that live inside the preatom $\calP(\lam)$. In other words, we truncate our NS staircases $(e_i)_{1\leq i\leq a}$ so that $a\leq k$. 

The following quantity measures the longest possible truncated NS staircase over $(\mu,\lamk)$ in a preatom of highest weight $\lambda$.
\begin{definition}
\label{truncatedns}
Assume that $k\geq 0$ and $\mu \leq \lam-k\varpi_2$. Then, for any $m\in \bbN\cup \{\infty\}$ we define
\[ \calD_m(\mu,\lam,k):=\min(k,\affD_m(\mu,\lam-k\varpi_2)).\]
\end{definition}

\section{The charge and recharge statistics}\label{sec:swapping}

\subsection{A family of cocharacters}\label{sec:family}

We recall some definitions from \cite{Pat}.
Let $\affX=X\oplus \bbZ d$ be the cocharacter lattice of $T^\vee \times \bbC^*$, where $T^\vee$ is the maximal torus of $G^\vee$.
Let $\affX_\bbQ  := \affX \otimes_\bbZ \bbQ$ and $\affX_\bbR := \affX \otimes_\bbZ \bbR$.

 The \emph{KL region} is the subset of $\affX_\bbQ$ of the elements $\eta$ such that $\langle \alpha^\vee,\eta\rangle >0$ for all $\alpha^\vee \in \affPhi_+^\vee$. Concretely, an element in the KL region can be written as $\eta=\lambda+Cd$ where $\lambda\in X_{++}$ and $C>\langle \lambda,\beta^\vee\rangle$ for all $\beta^\vee \in \Phi^\vee_+$.
 The \emph{MV region} is the subset of $\affX_\bbQ$ consisting of elements of the form $\eta=\lambda+Cd$, with $\lambda\in X_{++}$ and $C=0$.

We call \emph{wall} a hyperplane in $\affX_\bbR$ of the form 
\[H_{\alpha^\vee}:=\{ \eta \in \affX_\bbR \mid \langle \eta,\alpha^\vee \rangle =0\}\subset \affX_\bbR\]
for $\alpha^\vee\in \affPhi^\vee$.
	For $\lambda\in X_+$, we denote by $\affPhi^\vee(\lambda)$ the set of all the labels present in the graph $\Gamma_\lambda$. We say that a wall $H_{\alpha^\vee}$ is a $\lambda$-wall if $\alpha^\vee\in \affPhi^\vee(\lambda)$.
We call \emph{$\lambda$-chamber} (or simply chamber, if $\lambda$ is clear from the context)  the intersection of $\affX_\bbQ$ with a connected component of 
\[\affX_\bbR \setminus \bigcup_{\alpha^\vee\in \affPhi^\vee(\lambda)}H_{\alpha^\vee}.\]
Two chambers are \emph{adjacent} if they are separated by a single $\lambda$-wall. 
The \emph{KL chamber} is the unique chamber containing the KL region and the  \emph{MV chamber} is the unique chamber containing the MV region.
We say that $\lambda\in X_\bbQ$ is \emph{regular} if it does not lie on any wall, and it is \emph{singular} otherwise.

For $\lam\in X_+$ let $\bar{\Gr^\lam}$ denote the corresponding Schubert variety in the affine Grassmannian of $G^\vee$ (cf. \cite[\S 2.1.2.]{Pat}). For any regular $\eta \in \affX$ and any $\mu \leq \lam$ the hyperbolic localization induces a functor 
\[ \HL^\eta_\mu: \calD^b_{T^\vee \times \bbC^*}(\bar{\Gr^\lam})\raw \calD^b(pt)\cong \mathrm{Vect}^\bbZ,\]
where $\calD^b_{T^\vee \times \bbC^*}(\bar{\mathcal{G}r}^\lam)$ is the derived category of $T^\vee \times \bbC^*$-equivariant constructible sheaves on the Schubert variety $\bar{\Gr^\lam}$ with $\bbQ$-coefficients, and $\calD^b(pt)$ is the derived category of sheaves on a point, which is equivalent to the category of graded $\bbQ$-vector spaces (see \cite[\S 2.4]{Pat}).
In general, for any regular $\eta\in \affX_\bbQ$ we can define $\HL^\eta_\mu$ as $\HL^{N\eta}_\mu$, where $N$ is any positive integer such that $N\eta\in \affX$. By abuse of terminology, we are then allowed to refer to all the elements in $\affX_\bbQ$ as cocharacters.

Let $\htil^\eta_{\mu,\lam}(v) := \grdim(\HL^\eta_\mu(IC_\lam))$, where $IC_\lam$ denotes the intersection cohomology sheaf on $\bar{\Gr^\lam}$. The polynomials $\htil^\eta_{\mu,\lam}(v)$ are called \emph{renormalized $\eta$-Kazhdan--Lusztig polynomials}. We say that a function $r:\calB(\lam)\raw \bbZ$ is a $\eta$-\emph{recharge} for $\eta$ if we have
\[\htil^\eta_{\mu,\lam}(q^{\frac12})=\sum_{T\in \calB(\lam)_\mu} q^{r(T)}\in \bbZ[q^{\frac12},q^{-\frac12}].\]
If $\eta_{KL}$ is in the KL chamber and $\mu \in X_+$, then
\[K_{\mu,\lam}(q)=\htil^{\eta_{KL}}_{\mu,\lam}(q^{\frac12})q^{\frac12 \ell(\mu)}\] is a Koskta--Foulkes polynomial by \cite[Proposition 2.14]{Pat}. So if $r_{KL}$ is a recharge for $\eta_{KL}$  in the KL region, we obtain a charge statistic $c:\calB(\lam)\raw \bbZ$ by setting $c(T):=r_{KL}(T)+\frac 12 \ell(\wt(T))$. 
Notice that if $\wt(T)\in X_+$ this is equal to $c(T)=r_{KL}(T)+\langle \wt(T),\rho^\vee\rangle$.

We specialize \cite[Definition 3.29]{Pat} to our setting. 

\begin{definition}
Let $\lambda\in X_+$. We call \emph{$\lambda$-parabolic region} the subset of $\affX_\bbQ$ consisting of regular cocharacters $\eta$ such that	\begin{itemize}
		\item $\langle \eta,\beta^\vee\rangle>0$ for every $\beta^\vee$ of the form 
		$M\delta -\alpha_1^\vee$ with $M>0$, or of the form $M\delta+\alpha^\vee$, with $\alpha\in \Phi_+$ and $M\geq 0$.
		\item $\langle \eta,\beta^\vee\rangle<0$ for every $\beta^\vee\in \affPhi^\vee_+(\lambda)$ of the form $M\delta-\alpha_i^\vee$ such that $M> 0$ and $i\in \{2,12,21\}$.
	\end{itemize} 
\end{definition}

The walls that separate the  parabolic region from the KL region are precisely
\[ H_{M\delta-\alpha_i^\vee} \qquad \text{with }M>0\text{ and }i\in \{2,12,21\}.\]
Every cocharacter $\eta_P$ of the form
\[ \eta_P= A_1\varpi_1 + A_2\varpi_2 + Cd\]
with $0\ll A_1\ll C\ll A_2$ lies in the parabolic region.\footnote{More precisely, sufficient conditions are $0<A_1<C<\frac{A_2}{\gamma}$ where $\gamma=\max\{ M \mid M\delta-\beta^\vee \in \Phi^\vee(\lam)\}$.}

We consider the following family of cocharacters:
\begin{equation}\label{family}
	\eta:\bbQ_{\geq 0}\ra \affX_\bbQ,\qquad \eta(t)=\eta_P+t d.
\end{equation} 
Observe that $\eta(t)$ is in the KL chamber for $t\gg 0$. We can choose $t_0$ such that $\eta(t_0)$ is in the KL chamber and for any $i$ we choose $t_{i+1}<t_{i}$ so that $\eta(t_i)$ and $\eta(t_{i+1})$ lie in adjacent $\lambda$-chambers until we arrive at $t_M$ in the parabolic region. We can furthermore choose  $t_M=0$ and set $t_{M+1}=\ldots =t_{\infty}=0$ and $\eta_i:=\eta(t_i)$ for any $i\in \bbN \cup \{\infty\}$.

\subsection{Recharge statistics from the parabolic to the KL region}

Our goal is to attach a recharge statistic to each of the cocharacters $\eta_i$. 

Let $T\in \calB(\lambda)$. Recall that by \Cref{preatomicnumber,defatomicnumber} we have 
\[T\in \calA(\lam-\at(T)\varpi_2-2\pat(T) \varpi_1)\subset \calP(\lam-2\varpi_1(T))\subset \calB(\lam).\]

\begin{definition}\label{N=0}
Assume that $T\in  \calP(\lam)\subset \calB(\lam')$ with $\mu:=\wt(T)$. Let $a:=\at(T)$ and $p:=\pat(T)$ so that $\lam'=\lam+2p\varpi_1$. We define
\[ \sigma_m(T):= \ell_m(\mu,\lama)-\calN_m(\mu,\lama)+\calD_m(\mu,\lambda,a)+a+2p.\]
Let $r_m(T):=-\sigma_m(T)+\langle \lam',\rho^\vee\rangle=-\sigma_m(T)+\langle \lam,\rho^\vee\rangle+3p$.

\end{definition}

Our main result is the following.

\begin{theorem}
\label{maintheorem}
The function $r_m:\calB(\lambda)\raw \bbZ$ is a recharge statistic for $\eta_m$ for any $m\in \bbN\cup \{\infty\}$.
\end{theorem}

The proof that $r_i$ is a recharge for $\eta_i$ is divided in two parts. We first show directly in \Cref{sec:parabolic} that $r_\infty$ is a recharge statistic for $\eta_\infty=\eta(0)$, i.e. a recharge in the parabolic region, and then we construct for any $i$ swapping functions between $\eta_i$ and $\eta_{i+1}$. 
After putting everything together, this proves that $r_{KL}:=r_0$ is a recharge in the KL region, and we can easily obtain from that the following formula for a charge statistic in type $C_2$.

\begin{corollary}
\label{maincharge}
The function
$c:\calB(\lam)_+\raw \bbZ$ defined as
\[c(T)=\langle \lambda -\wt(T),\rho^\vee\rangle -\at(T)-\pat(T)\]
is a charge statistic.
\end{corollary}
\begin{proof}
By definition, we have $\calN_0=\calD_0=0$ and $\ell_0=\ell$. Hence
\[ c(T)=r_0(T)+\frac12\ell(\wt(T))=\langle \lambda,\rho^\vee\rangle -\frac12\ell(\wt(T))-\at(T)-2\pat(T)\]
is a charge statistic. We conclude since, for $T=\calB(\lam)_+$, we have $\ell(\wt(T))=2\langle \wt(T),\rho^\vee\rangle$.
\end{proof}

\subsection{Recharge in the parabolic region}\label{sec:parabolic}

Let $\eta_{MV}$ be a cocharacter in the MV region and $\eta_P$ be in the parabolic region. The only walls separating $\eta_{MV}$ from $\eta_P$ are of the form $H_{M\delta-\alpha_1^\vee}$, with $M>0$. We now from  \cite[Eq. (17)]{Pat} that
\[r_{MV}(T) = -\langle \rho^\vee, \wt(T)\rangle.\] 
is a recharge in the MV region. To construct a recharge in the parabolic region, after Levi branching, we can assume we are in rank $1$ and thus compute the recharge as illustrated in \cite[\S 3.4]{Pat}. In particular, it follows from \cite[Lemma 3.26]{Pat} that 
\[ r_P(T)=-\langle \rho^\vee,\wt(T)\rangle+\phi_1(T)-\ell^1(\wt(T))\]
is a recharge in the parabolic region. It remains to show the equality between $r_P$ and $r_{\infty}$.


Let $T\in \calP(\lam)\subset \calB(\lam+2p\varpi)$ with $p=\pat(T)$ and let $a=\at(T)$ and $\mu=\wt(T)$.
At $m=\infty$, we have 
\begin{align}\label{sigmainf}
 \sigma_\infty(T) = \ell^1(\mu)+  \sum_{i\in \{2,12,21\}}\affphi_i(\mu,\lama) \nonumber \\
 -\calN_\infty(\mu,\lama)+\calD_\infty(\mu,\lam,a)+a+2p.
\end{align}

Our next goal is to simplify the expression \eqref{sigmainf}.

\begin{lemma}\label{phi21k}
We have $\affphi_{21}(\mu,\lama)+a=\affphi_{21}(\mu,\lambda)$.
\end{lemma}
\begin{proof}
    This follows directly from \Cref{affphicompute}.
\end{proof}

\begin{proposition}\label{phi12hard}
Let $\mu = \wt(T)$ and assume that $\mu_1\leq  0$. We have $\phi_2(T)=\affphi_2(\mu,\lama)$ and
\[\phi_{12}(T)=\affphi_{12}(\mu,\lama)-\calN_\infty(\mu,\lama)+\calD_\infty(\mu,\lambda,a).\]
\end{proposition}
The proof of \Cref{phi12hard} is rather long and technical and we postpone it to \Cref{sec:phi2}.

\begin{lemma}\label{sigmainfnicelemma}
Let $\mu=\wt(T)$. We have 
\begin{equation}\label{sigmainfnice}\sigma_\infty(T)=\ell^1(\mu)+\phi_2(T)+\phi_{12}(T)+\affphi_{21}(\mu,\lambda)+2p.
\end{equation}
\end{lemma}
\begin{proof}
If $\mu_1\leq 0$, this follows immediately from \Cref{phi21k} and \Cref{phi12hard}.

If $\mu_1>0$, then let $T'=s_1(T)$. Recall that atoms are stable under $s_1$ by \Cref{lemmaonPsi}. So the element $T'$ can also be characterized as the element in the same atom of $T$ with weight $s_1(\mu)$. 
Notice that $\affphi_{21}$, $\calN_\infty$ and $\calD_\infty$ are preserved by $s_1$, while  $\affphi_{2}(\mu,\lama)=\affphi_{12}(s_1(\mu),\lama)$ and $\ell^1(\mu)=\ell^1(s_1(\mu))- 1$. It follows that $\sigma_\infty(T)=\sigma_\infty(T')-1$. On the other hand, we also have $\phi_2(T)=\phi_{12}(T')$ and $\phi_{12}(T')=\phi_2(T)$, so we obtain the desired identity \eqref{sigmainfnice} for $T$ as well.
\end{proof}

\begin{proposition}\label{rP}
We have $r_P(T)=r_\infty(T)$ for any $T \in \calB(\lambda')$.
\end{proposition}
\begin{proof}
Let $\mu=\wt(T)$ and assume $T\in \calP(\lam)\subset \calB(\lam+2p\varpi_1)$. By \Cref{sigmainfnicelemma} we have
\[r_\infty(T)=-\ell^1(\mu)-\phi_2(T)-\phi_{12}(T)-\affphi_{21}(\mu,\lam)+\langle \lam,\rho^\vee\rangle+p.\]
So our claim is equivalent to
\[ \langle \lambda+\mu,\rho^\vee \rangle -\affphi_{21}(\mu,\lam)=\phi_1(T)+\phi_2(T)+\phi_{12}(T)-p=Z(T)-p.\]
By  \Cref{atomicnumber} and \Cref{affphicompute} we have
\begin{align*}\langle \mu+\lam,\rho\rangle -\affphi_{21}(\mu,\lam)
		&= \lam_2+\mu_2+\frac32\lam_1+\frac32\mu_1-\min(\lam_1,\frac{\lam_1+\mu_1}{2},\lam_1+\mu_1)\\
		&= \lam_2+\mu_2+\lam_1+\mu_1-\min(\frac{\lam_1-\mu_1}{2},0,\frac{\lam_1+\mu_1}{2})\\
  & =Z(T) - p.\qedhere
\end{align*}
\end{proof}

\subsection{Computing \texorpdfstring{$\phi_2$}{psi}}
\label{sec:phi2}
It remains to prove the identity \Cref{phi12hard}.

We begin by considering the case $\at(T)=0$. The general case will follow by induction on the atomic number.

\begin{proposition}\label{atomicphi2}
    For any $T \in \calP(\lam)$ with $\wt(T)_1\leq 0$ we have $\phi_2(T)=\affphi_2(\wt(T),\lam-\at(T)\varpi_2)$.
\end{proposition}

\begin{proof}
Let $\mu \leq \lam$. Consider the multiset
\[ M_2(\mu,\lam):=\{ \phi_2(X) \mid X \in \calP(\lambda)\text{ with }\wt(X)=\mu\}\]
Since $\calP(\lambda)$ is a union of $f_2$-strings, we have an equality of multisets
\begin{equation}\label{M2alt}M_2(\mu,\lam)=\{\affphi(\mu,\lamk) \mid 0\leq k\leq \lam_2 \text{ with }\mu \leq \lamk\}.
\end{equation}
In fact, the $f_2$-strings contained in $\calP(\lambda)$ which pass through an element of weight $\mu$ are in bijection with the atoms in $\calP(\lambda)$ containing an element of weight $\mu$.

The claim now follows by induction on $\lam_2$. If $\lam_2=0$, then $\calP(\lambda)=\calA(\lambda)$ and $M_2(\mu,\lambda)=\{\phi_2(T)\} =\{\affphi_{2}(\mu,\lam)\}$.

If $\lam_2>0$, consider the embedding $\Psi:\calP(\lam-\varpi_2)\hookrightarrow \calP(\lam)$ from \Cref{defPsi}. The map $\Psi$ is weight preserving and we have  $\phi_2(\Psi(X))=\phi_2(X)$ and $\at(\Psi(X))=\at(X)+1$ for any $X\in \calP(\lam-\varpi_2)$ with $\wt(X)_1\leq 0$. 
If $T= \psi(X)$ for some $X\in \calP(\lam-\varpi_2)$, then $\phi_2(T)=\phi_2(X)=\affphi_2(\mu,\lam-\varpi_2-\at(X)\varpi_2)$ and the claim follows. Otherwise, we have $T\in \calA(\lam)=\calP(\lam)\setminus \Psi(\calP(\lam-\varpi_2))$
and by \eqref{M2alt} we see that
\[ \{\phi_2(T)\}=M_2(\mu,\lam)\setminus M_2(\mu,\lam-\varpi_2)=\{ \affphi_2(\mu,\lam)\}.\qedhere\]
\end{proof}

\begin{lemma}\label{phi1at0}
Let $T\in \calP(\lam)\subset \calB(\lam)$ and let $\mu=\wt(T)$. Assume $\mu_1< 0$ and $\at(T)=0$. Then we have $
    \phi_1(T) = \max(0,\mu_1+\mu_2-\lam_2,-\mu_2-\lam_2).$
\end{lemma}

\begin{proof}


Let $\stn(T)=(a,b,c,d)$.
By \Cref{atom0} we have 
\[\at(T)=0\iff (c=d=0) \text{ or }(b=\lam_1+2c-2d\text{ and }d\leq 1\text{ or }c= \lam_2+d)\]

As computed in \eqref{phi1stn}, we have
\[\phi_1(T)=\lam_1+2a-2b+2c-2d+\max(d,2c-b,b-2a).\]

We now divide into several cases. 
Assume first $c=d=0$. Then the statement is equivalent
\begin{equation}\label{eqcd0}\lam_1+2a-b-\min(2a,b) = \max(0,\lam_1-b,-2\lam_2-b+2a).\end{equation}
Since $\mu_1=\lam_1-2b+2a <0$ and $b\leq \lam_1$, we have $2a\leq b$, so the LHS in \eqref{eqcd0} is $\lam_1-b$.
Moreover, $\lam_1-b\geq 0$ and $\lam_1-b\geq 2a-b-2\lam_2$ otherwise we get.
$\lam_1+2\lam_2< 2a<b.$ So the RHS in \eqref{eqcd0} is also equal to $\lam_1-b$.

We can now assume $b=\lam_1+2c-2d$, so we have
$\phi_1(T)=\max(-\lam_1+2a-2c+3d,0)$,
while the RHS  can be rewritten as 
$\max(0,d-2c,-2\lam_2+d+2a-\lam_1)$.
Moreover, we have $d-2c\leq 0$ and $\mu_1 = -\lam_1+2a-2c+2d\leq 0$.

So it is enough to show that 
\begin{equation}\label{eq2}\max(0,\mu_1+d)=\max(0,\mu_1-2(c-\lam_2-d)+d)\end{equation}
The equality is clear if $c=\lam_2+d$ and it also follows if $d\leq 1$ since that both term vanish for $\mu_1<0$.
\end{proof}
%

\begin{proposition}
\label{phi12at0}
Let $T\in \calP(\lam)$ and let $\mu=\wt(T)$. Assume $\mu_1\leq  0$ and $\at(T)=0$. Then we have $\phi_{12}(T)=\affphi_{12}(\mu,\lam)-\calN_{\infty}(\mu,\lam).$
\end{proposition}
\begin{proof}
If $\mu_1=0$, then $\phi_{12}(T)=\phi_2(T)$, $\affphi_{12}(T)=\affphi_{2}(T)$ and $\calN_\infty(\mu,\lam)=0$, so the claim follows from \Cref{atomicphi2}. We assume in the rest of the proof $\mu_1<0$. We can also assume that $T$ lies in the biggest preatom, i.e. that $\calP(\lam)\subset \calB(\lam)$. In fact, since $\Phi$ commutes with $s_1$ and $f_2$, the claim for the other preatoms easily follows by induction.

Recall now by \Cref{atomicphi2} that $\phi_2(T)=\affphi_2(\mu,\lam)$. We divide into three cases.

We assume first $\mu_1+\mu_2\leq \lam_2$ and $\mu_2\geq -\lam_2$. Notice that this precisely means that $\phi_1(T)=0$. By \Cref{Ninf}, we have in this case
\[ \affphi_{12}(T)-\calN_\infty(\mu,\lam)=\max(0,\frf{\lam_2+\mu_1+\mu_2})-\min(0,\frac{-\lam_1-\mu_1}{2}).\]
Let $\chi:=\lam_2+\mu_1+\mu_2$. Then by \Cref{atomicnumber,atomicphi2,phi1at0} we have
\begin{align*} \phi_{12}(T) &=Z(T)-\phi_1(T)-\phi_2(T)\\
&=\frac{\lam_1+\mu_1}{2}+\chi +\max(0,\frac{-\mu_1-\lam_1}{2})-\min(\chi,\frf{\chi},\lam_2).
\end{align*}

Notice that $\min(0,\frac{-\lam_1-\mu_1}{2})+\max(0,\frac{-\lam_1-\mu_1}{2})=\frac{-\lam_1-\mu_1}{2}$ and that $\lam_2\geq \frf{\chi}$. So our claim results equivalent to the easy-to-check identity
\[\chi -\min(\chi,\frf{\chi})=\max(0,\frc{\chi}).\]

We now assume that $\mu_1+\mu_2> \lam_2$ or that $\mu_2<-\lam_2$. In both cases we have from \Cref{octineq} that $\mu_1>-\lam_1$, so $Z(T)=\lam_1+\lam_2+\mu_1+\mu_2$. Moreover, we have from \Cref{Ncount} that $\calN_{\infty}(\mu,\lam)=0$, so the claim is equivalent to $Z(T)-\phi_1(T)-\phi_2(T)=\affphi_{12}(T)$.

If $\mu_1+\mu_2>\lam_2$, then 
$\affphi_2(T)=\frac{\lam_1-\mu_1}{2}+\lam_2$ and $\affphi_{12}(T)=\frac{\lam_1+\mu_1}{2}+\lam_2$, so the desired equality reduces to the identity
\[ \lam_1+\lam_2+\mu_1+\mu_2-\mu_1-\mu_2+\lam_2-\frac{\lam_1}{2}+\frac{\mu_1}{2}-\lam_2=\frac{\lam_1}{2}+\frac{\mu_1}{2}+\lam_2.\]
Finally, if $\mu_2<-\lam_2$, the desired equality reduces to the identity
 \[ \lam_1+\lam_2+\mu_1+\mu_2+\mu_2+\lam_2-\frac{\lam_1}{2}+\frac{\mu_1}{2}-\mu_1-\mu_2-\lam_2=\frac{\lam_1}{2}+\frac{\mu_1}{2}+\lam_2+\mu_2.\qedhere\]
\end{proof}

\begin{proposition}\label{phi2claim}
Let $T\in \calP(\lam)$ and let $\mu=\wt(T)$. Let $A:=\at(T)$.
	If $\mu_1\leq 0$, we have
\[\phi_{12}(T)=\affphi_{12}(\mu,\lam- A\varpi_2)-\calN_\infty(\mu,\lambda-A\varpi_2)+\calD_\infty(\mu,\lambda,A).\]
\end{proposition}

\begin{proof}
As in \Cref{phi12at0} we can assume that $\calP(\lam)\subset \calB(\lam)$.
We show the claim it by induction on $A$.
	If $A=0$, the claim immediately follows from \Cref{phi1at0} since $\calD_\infty(\mu,\lam,0)=0$. 

 If $A>0$, then $T=\Psi(U)$ for some $U\in \calP(\lambda-\omega_2)\subset \calB(\lam-\omega_2)$ with $\at(U)=A-1$.
By induction, we have 
\[ \phi_{12}(U)=\affphi_{12}(\mu,\lamA)-\calN_\infty(\mu,\lamA)+\calD_{\infty}(\mu,\lam-\varpi_2,A-1).\]
 So it suffices to show that, for any $U$ in $\calP(\lam-\varpi_2)\subset \calB(\lam-\varpi_2)$ with $\wt(U)=\mu$, we have
\begin{align}\label{Psiphi12} \phi_{12}(\Psi(U))-\phi_{12}(U) &=\calD_\infty(\mu,\lambda,A)-\calD_{\infty}(\mu,\lam-\varpi_2,A-1)\\
& =\min(A,\affD_\infty(\mu,\lamA))-\min(A-1,\affD_\infty(\mu,\lamA)).\nonumber\end{align}

Let $\stn(U)=(a,b,c,d)$. We know from \Cref{cor:phi12psi} that 
\[ \phi_{12}(\Psi(U))-\phi_{12}(U)=\begin{cases}
1& \text{if }d=0\text{ and }2a>b>2c\text{ or }d\neq 0,\lam_1\text{ and }b\geq 2a+d\\
0&\text{otherwise}.
\end{cases}\]
However, notice that we cannot have $d=0$ and $2a>b>2c$ since otherwise
$\mu_1=\lam_1+2a-2b+2c> \lam_1+2c-b\geq 0$.
It follows that \eqref{Psiphi12} is equivalent to showing that
\[ \affD_\infty(\mu,\lamA)\geq A \iff d\neq 0,\lam_1 \text{ and }b\geq 2a + d.\]
We show this in the following lemma.
\end{proof}

\begin{lemma}
Let $X\in \calP(\lam)\subset \calB(\lam)$ with $\mu=\wt(X)$ such that $\mu_1< 0$. Let $A:=\at(X)$ and $\stn(X)=(a,b,c,d)$. We have 
\[\affD_\infty(\mu,\lamA)> A \iff d\neq 0,\lam_1 \text{ and }b\geq 2a + d.
\]
\end{lemma}
\begin{proof}

Recall from \Cref{Dinf<0} that we have 
\[\affD_\infty(\mu,
\lamA)=\begin{cases}
	\max(0,\min(\lam_1,-\mu_1)-1) &\begin{array}{l}\text{if }\mu_1+\mu_2+\lam_2\geq A,\\
 \mu_1+\mu_2+A\leq \lam_2 \text{ and }\\
 \mu_1+\mu_2\not\equiv\lam_2-A \pmod{2};\end{array}\\
\begin{aligned}\max(0,\min(-\mu_1,\lam_1)+\lam_2\\-A+\mu_1+\mu_2) \end{aligned}&\begin{array}{c}\text{if } \mu_1+\mu_2+\lam_2<A;\end{array}\\
0 &\begin{array}{c}\text{otherwise.}\end{array}\end{cases}\]
Moreover,  from \Cref{atomformula} we have
\[A= \at(X)=\begin{cases} \min(c,\lam_1+2c-b)& \text{if }d=0\\
\lam_1+2c-2d-b+\min(\lam_2+d-c,d-1) &\text{if }d>0.\end{cases}\]

We divide the proof into three cases.

\noindent \textbf{First case: $d=0$.\;}
 We claim that in this case we  actually have $\affD_\infty(\mu,\lamA)=0$.
Notice that $\mu_1+\mu_2+\lam_2=\lam_1+2\lam_2-b\geq \lam_1+2c-b\geq A$. So we can also assume that $A\leq \lam_2-\mu_1-\mu_2=b-\lam_1$. Notice that this is equivalent to $c+\lam_1\geq b$ and $A=\lam_1+2c-b$. However, if $A=\lam_1+2c-b$ then $\mu_1+\mu_2+\lam_2+A\equiv 0 \pmod{2}$, and therefore $\affD_\infty(\mu,\lam-A\varpi_2)=0$.

\noindent \textbf{Second case: $d=\lam_1$.\;} In this case we have $A=\min(\lam_2+c-b,2c-b-1)$ and $\mu_1+\mu_2+\lam_2=2\lam_2-b$. It follows that $\mu_1+\mu_2+\lam_2\geq A$ if and only if $\lam_2\geq c$. Recall also that $b\leq 2c-\lam_1$.

Assume first $\lam_2\geq c$, so that $A=2c-b-1$ and $\mu_1+\mu_2+\lam_2\geq A$. The claim now follows since $\affD_\infty(\mu,\lam-A\varpi_2)\leq \lam_1-1\leq 2c-b-1$.

Assume now $\lam_2<c$ so that $A=\lam_2+c-b$ and $\mu_1+\mu_2+\lam_2\leq A$. The claim follows because, if $\affD_\infty(\mu,
\lamA)\geq 0$, then 
$\affD_\infty(\mu,
\lamA)\leq \lam_1+\lam_2+\mu_1+\mu_2-A=\lam_1+\lam_2-c\leq \lam_2+c-b=A$.

\noindent \textbf{Third case: $d\neq 0,\lam_1$.\;}
In this case we have $b=\lam_1-2d+2c$. Notice that $b\geq 2a+d$ is equivalent to $\lam_1-2a+2c> 3d$. We also have $A=\min(\lam_2+d-c,d-1)$ and $\mu_1+\mu_2+\lam_2=2\lam_2-2c+d$, so $\mu_1+\mu_2+\lam_2\geq A$ if and only if $\lam_2\geq c$.

Assume first $\lam_2\geq c$, so that $A=d-1$ and $\mu_1+\mu_2+\lam_2\geq A$. Notice that $\lam_1-1>A$ and also \[-\mu_1-1>A\iff \lam_1+2c-2a-2d-1>d-1\iff \lam_1+2c-2a>3d\]
Hence, $\affD_\infty(\mu,\lamA)>A$ if and only if $\lam_1+2c-2a>3d$.

Finally assume $\lam_2<c$ so that $A=\lam_2+d-c$ and $\mu_1+\mu_2+\lam_2<A$. In this case we have $\affD_\infty(\mu,\lamA)=\max(0,\min(0,\mu_1+\lam_1)+\lam_2+\mu_2-A)$.
We have $\mu_1+\lam_1+\lam_2+\mu_2-A=\lam_1+\lam_2-c>\lam_2+d-c=A$ and
\[\lam_2+\mu_2-A>A\iff \lam_1+2c-2a> 3d.\]
It follows that $\affD_\infty(\mu,\lamA)>A$ if and only if $\lam_1+2c-2a>3d$.
\end{proof}

\section{Swapping functions}

Recall the family of cocharacters $\{\eta_m\}_{m\in \bbN}$ introduced in \Cref{sec:family}.
The unique wall separating $\eta_m$ and $\eta_{m+1}$ is 
$H_{\alpha_{m+1}^\vee}$, where $\alpha_{m+1}^\vee \in \affPhi^\vee_+$ is the $(m+1)$-th root occurring in the sequence \eqref{reflectionorder}.  
As in \eqref{tm}, let $t:=t_{m+1}$  denote the corresponding reflection.
 For any $\mu\in X$ such that $\mu<t\mu\leq \lambda$ we define
\[ \psi_{t\mu}:\calB(\lambda)_{t\mu}\ra 
\calB(\lambda)_{\mu}\]
as follows. 
Let $T\in \calB(\lam)_{t\mu}$ and assume that $T\in \calA(\lam-a\varpi_2)\subset \calP(\lam)$ and let $e:=(\mu \raw t\mu)\in E(\lambda-a\varpi_2)$. Then $\psi_{t\mu}(T)=T'$, where $T'$ is the only element of weight $\mu$ in $\calA(\lam-(a+\Omega(e))\varpi_2)\subset \calP(\lam)$.

\begin{proposition}\label{swappingfunctionprop}
The collection of maps $\psi=\{\psi_\nu\}$ is a swapping function between $\eta_{m+1}$ and $\eta_m$. In particular, if $r_{m+1}$ is a recharge for $\eta_{m+1}$ then $r_m$ is a recharge for $\eta_m$. 
\end{proposition}

To prove \Cref{swappingfunctionprop} we need to check that for any $m$ and $T$ we have $r_{m+1}(T)=r_{m+1}(\psi_{t\mu}(T))+1$, or equivalently that $\sigma_{m+1}(T)=\sigma_{m+1}(\psi_{t\mu}(T))-1$. 

\begin{proposition}\label{swapcheck}
Assume $T\in \calA(\lam-k\varpi_2)\subset \calP(\lam)$ with $t\mu =\wt(T)$ and that $e:=(\mu\raw t\mu)\in E(\lama)$. 
Then, we have 
\begin{equation}\label{sigmans}\sigma_{m+1}(T)=\sigma_{m+1}(\psi_{t\mu}(T))-1.\end{equation}
\end{proposition}
\begin{proof} 
By \Cref{Nintmuis0} we have $\calN_{m+1}(t\mu,\lama)=0$ and by \Cref{ifN=0thenD=0} we also get $\calD_{m+1}(t\mu,\lam,a)=0$. Let $\Omega:=\Omega(e)$ and recall that $\psi_{t\mu}(T)$ is the element of weight $\mu$ in the atom $\calA(\lam-(a+\Omega)\varpi_2)\subset \calP(\lam)$.

First assume $\Omega=0$, or equivalently that $e$ is swappable. Since $\mu\raw t\mu$ is swappable, by definition we have $\ell_{m+1}(\mu,\lama)=\ell_{m+1}(t\mu,\lama)+1$. By \Cref{corcon} we have that  $\calN_{m+1}(\mu,\lama)=0$ and by  \Cref{ifN=0thenD=0}, we also get $\calD_{m+1}(\mu,\lam,a)$. The claim now easily follows.

Assume now $\Omega>0$, so $e$ is not swappable. Notice that $\calD_{m+1}(\mu,\lam,a+\Omega)=\Omega$.
Combining with \Cref{Nintmuis0}, our claim \eqref{sigmans} becomes equivalent to
\[\ell_{m+1}(t\mu,\lama)=\ell_{m+1}(\mu,\lam-(a+\Omega)\varpi_2)-\calN_{m+1}(\mu,\lam-(a+\Omega)\varpi_2)+2\Omega-1.\]

We can assume $\mu_1\geq 0$ as the case $\mu_1<0$ is symmetric. Because $e$ is not swappable, we have $m+1=4M$, $q_M\mu\not\leq \lam$ and $r_M\mu \not \leq \lam$. In particular, by \eqref{eqqu} and \eqref{eqru} we have
 $\ell^{12}_{m+1}=\affphi_{12}$ and $\ell^{21}_{m+1}=\affphi_{21}$. 
Using \Cref{Ncountgen}, our claim is then equivalent to
\begin{equation}\label{2Omegaeq}\ell_{m+1}(t\mu,\lama)-\affell_{m+1}(t\mu,\lam-(a+\Omega)\varpi_2)=2\Omega.
\end{equation}
By \Cref{affphicompute} we have  
\[\affphi_{21}(t\mu,\lama)-\affphi_{21}(t\mu,\lam-(a+\Omega)\varpi_2)=\Omega\]
\[\affphi_{12}(t\mu,\lama)-\affphi_{12}(t\mu,\lam-(a+\Omega)\varpi_2)=\Omega\]
because $(t_{m+1}\mu)_2\leq -\lam_2$ (as proven in \Cref{claimtmu}) and the identity \eqref{2Omegaeq} now follows directly from the definition of $\affell_{m+1}$. 
\end{proof}


\begin{proof}[Proof of \Cref{swappingfunctionprop}.]
\Cref{swapcheck} precisely shows that $\psi$ is a swapping function for $r_{m+1}$. This means that we can obtain a new recharge $r'_m$ for $\eta_m$ by swapping the values of $r_{m+1}$ as indicated by $\psi$. It remains to show that $r_m=r_m'$. In other words, for $t=t_{m+1}$ and for any $\mu\leq t\mu$  we need to show that 
\begin{enumerate}
    \item if $\wt(T)=t\mu$ then $r_m(T)=r_{m+1}(\psi(T))=r_{m+1}(T)-1$;
    \item if $\wt(U)=\mu$ and $U\in \Ima(\psi_{t\mu})$ then $r_m(U)=r_{m+1}(\psi_{t\mu}^{-1}(U))=r_{m+1}(U)+1$;
    \item   if $\wt(U)=\mu$ and $U\not\in \Ima(\psi_{t\mu})$ then $r_m(U)=r_{m+1}(U)$.
\end{enumerate}

The first statement is clear since $r_m(T)-r_{m+1}(T)=\ell_{m+1}(t\mu,\tau)-\ell_{m}(t\mu,\tau)=-1$ by \Cref{Nintmuis0,ifN=0thenD=0}. Let now $U\in \calA(\zeta)$ with $\wt(U)=\mu$ and let $a:=\at(U)$. We need to compute
\begin{align}
     r_m(U)-r_{m+1}(U)=&\ell_{m+1}(\mu,\zeta)-\ell_m(\mu,\zeta)-\calN_{m+1}(\mu,\zeta)+\calN_m(\mu,\zeta)+ \nonumber\\&+\calD_{m+1}(\mu,\zeta+a\varpi_2,a)-\calD_m(\mu,\zeta+a\varpi_2,a)\label{rmT'}.
\end{align}
If $U\in \Ima(\psi_{t\mu})$, there exists $\Omega\in \bbN$ such that $e:=(\mu\raw t\mu)\in E(\zeta+\Omega\varpi_2)$ and $\Omega=\Omega(e)$. 
If $\Omega=0$, then $e$ is swappable and $t\mu\leq \zeta$. So \eqref{rmT'} simplifies to $r_m(U)-r_{m+1}(U)=\ell_{m+1}(\mu,\zeta)-\ell_m(\mu,\zeta)=1$. If $\Omega>0$, then we have by \Cref{corcon}.1 that $(\mu\raw t\mu)\in E^N(\zeta)$ or $t\mu\not \leq \zeta$. It follows that
\begin{equation}\label{Ndiff}
    \ell_{m+1}(\mu,\zeta)-\ell_m(\mu,\zeta)=\calN_{m+1}(\mu,\zeta)-\calN_m(\mu,\zeta)=\begin{cases} 1&\text{if }t\mu \leq \zeta\\
0&\text{if }t\mu \not \leq \zeta,\end{cases}
\end{equation}
so the first line in the RHS of \eqref{rmT'}  vanishes.
Since $\Omega\leq a$, the edge $e$ belongs to a truncated NS staircase over $(\mu,\zeta+a\varpi_2)$, hence $\calD_{m+1}(\mu,\zeta+a\varpi_2,a)-\calD_m(\mu,\zeta+a\varpi_2,a)=1$.

Finally, assume that $U\not\in \Ima(\psi_{t\mu})$. This means that $(\mu\raw t\mu)\not\in E^S(\zeta)$, so \eqref{Ndiff} holds again in this case. Moreover, there does not exists $k\leq a$ such that $f:=(\mu\raw t\mu)\in E^N(\zeta+k\varpi_2)$ with $\Omega(f)=k$, from which it follows that $\calD_{m+1}(\mu,\zeta+a\varpi_2,a)=\calD_m(\mu,\zeta+a\varpi_2,a)$ and \eqref{rmT'} can be simplified to $r_m(U)-r_{m+1}(U)=0$.
\end{proof}

\subsection{Alternative formula}
We can obtain an alternative formula for the charge statistic by focusing on a single element and counting how many times its recharge gets changed by a swapping function. In type $A$, this is discussed in \cite[Remark 5.6]{Pat}.

\begin{definition}
	We define $\Delta^\alpha:
	\calB(\lambda)\raw \bbZ$, for $\alpha\in \Phi_+$ as the \emph{total contribution} of the swapping functions along the direction $\alpha$. It is defined as 
	\[ \Delta^{\alpha}=\sum r_m(T)-r_{m-1}(T)\]
	where the sum runs over all $m$ such that the (unique) wall between the $\lambda$-chambers of $\eta_m$ and $\eta_{m-1}$ is of the form $H_{M\delta-\alpha^\vee}$.
	
	We write $\Delta^i:=\Delta^{\alpha_i}$  for $i\in \{1,2,21,12\}$.
\end{definition}

We have $r_{KL}-r_{MV}=\sum_{\alpha\in \Phi_+} \Delta^\alpha$.
Recall in type $A$ for any $\alpha\in \Phi_+$ we have $\Delta^{\alpha}(T)=\phi_\alpha(T)-\ell^\alpha(\wt(T))$. When we apply the swapping functions along the $\alpha_1$-direction, to go from the MV region to the parabolic region, we  regard $\calB(\lambda)$  as a crystal of type $A_1$.  It follows that $\Delta^1(T)=\phi_1(T)-\ell^1(\wt(T))$ as in \cite[Lemma 3.26]{Pat}. Moreover, if $T\in \calB_+(\lam)$, we have
\begin{equation}\label{Deltaepsilon}
    \Delta^1(T)=\phi_1(T)-\langle \wt(T),\alpha_1^\vee\rangle=\eps_1(T).
\end{equation}


\begin{proposition}
	For $T\in \calA(\zeta)\subset \calB(\lambda)$ with $\wt(T)=\mu$ and $\at(T)=a$ we have 
	\begin{enumerate}
		\item $\Delta^{21}(T)=\affphi_{21}(\mu,\zeta)-\ell^{21}(\mu)(T)$.
		\item $\Delta^2(T)=\phi_2(T)-\ell^{2}(\wt(T))$
		\item $\Delta^{12}(T)=\phi_{12}(T)-\ell^{12}(\wt(T))$
	\end{enumerate}
\end{proposition}
\begin{proof}
By \Cref{21swappable}, the swaps in the $\alpha_{21}$-direction always occur  within the atom of $T$, so to compute $\Delta^{21}(T)$ we just need to consider the string of elements in the atom of $T$ of weights $\mu+k\alpha_{21}$. This means that we can compute $\Delta^{21}$ as in the rank one case and have $\Delta^{21}(T)=\affphi_{21}(T)-\ell^{21}(\mu)$.

	Assume first $\mu_1\leq 0$. Then the swapping occurring on $T$ in the $\alpha_2$ direction only occur within the atom of $T$, so as for $\Delta^{21}$, we have 
 \[\Delta^2(T)=\affphi_{2}(\mu,\zeta)-\ell^2(\mu)=\phi_2(T)-\ell^2(\mu),\]
where the second equality comes from \Cref{atomicphi2}.
 
	Assume now $\mu_1\geq 0$. Then by construction the number of swappable edges containing $\mu$ in the atom of $T$ is $\affphi_2(\mu,\zeta)-\calN_\infty(\mu,\zeta)$. Of these, there are $\ell^2(\mu)$ attached to roots $M\delta+\alpha_2^\vee$, which do not correspond to any crossed wall. Moreover, $T$ is also in the image of $\calD_\infty(\mu,\zeta+a\varpi_2,a)$ swapping functions, corresponding to non-swappable edges in atoms bigger than $\calA(\zeta)$. It follows that \[\Delta^2(T)=\affphi_2(\mu,\zeta)-\ell^2(\mu)-\calN_\infty(\mu,\zeta)+\calD_{\infty}(\mu,\zeta+a\varpi_2,a)=\phi_2(T)-\ell^2(\mu)\] by \Cref{phi2claim}.
	
	The proof of the formula for $\Delta^{12}$ is symmetric.
\end{proof}

Assume now that $T\in \calB_+(\lam)$. Then, as in \eqref{Deltaepsilon}, we have $\Delta^2(T)=\eps_2(T)$ and $\Delta^{12}(T)=\eps_{12}(T)$. 

\begin{definition}
Let $T\in \calA(\zeta)$ be such that  $\wt(T)=\mu$. We define $\affeps_{21}(T):=\affphi_{21}(\mu,\zeta)-\ell^{21}(\mu)$.
\end{definition}

Notice that $\affeps_{21}(T)$ can equivalently be defined as the largest integer $k$ such that $\wt(T)+k\leq \zeta$, for $T\in \calA(\zeta)$.

For $T\in \calB_+(\lam)$, we have $r_{KL}(T)-r_{MV}(T)=\eps_1(T)+\eps_2(T)+\eps_{12}(T)+\affeps_{21}(T)$. Since $r_{MV}(T)+\frac 12 \ell(\wt(T))=0$ for $T\in \calB_+(\lam)$, it follows that
\[ c(T)=\eps_1(T)+\eps_2(T)+\eps_{12}(T)+\affeps_{21}(T)\]
is a charge statistic on $\calB_+(\lam)$.

We conclude by giving a more explicit way to compute $\affeps_{21}(T)$.

\begin{definition}\label{defe21}
    Let $T$ be in the biggest atom, that is we assume $T\in \calA(\lam)\subset \calP(\lam)\subset \calB(\lam)$ and let $\str_2(T)=(a,b,c,d)$.    We define
    \[e_{21}^{\str}(T):=\begin{cases}
        (a-1,b-1,c,d) & \text{if }c=d=0\\
        (a-1,b-2,c,d+1) & \text{if }c>0\text{ and }d=0\\
        (a,b,c-1,d-1) & \text{if }d>0
    \end{cases}\]
    and $\bar{e}_{21}(T)$ as the element in $\calB(\lam)$ such that $\str_2(\bar{e}_{21}(T))=e_{21}^{\str}(T))$ if it exists, and $0$ otherwise.
    Finally, define $\affe_{12}(T)$ as $\bar{e}_{12}(T)$ if $\wt(T)_1\leq 0$ and $s_1(\bar{e}_{12}(s_1(T)))$ if $\wt(T)_1\geq 0$.
\end{definition}

\begin{proposition}
    Let $T\in \calA(\lam)\subset \calB(\lam)$
    \begin{itemize}
        \item If $\affe_{21}(T)\neq 0$, then $\affe_{21}(T)\in \calA(\lam)$ and $\affeps_{21}(T)>0$.
        \item If $\affe_{21}(T)=0$ and $\langle\wt(T),\alpha_{21}^\vee\rangle \geq 0$, then $\affeps_{21}(T)=0$.
    \end{itemize}  
\end{proposition}
\begin{proof}
    It can be easily verified by \Cref{atom0} that if $T\in \calA(\lam)$ and $\affe_{21}(T)\neq 0$, then also $\affe_{21}(T)\in \calA(\lam)$.

To prove the second statement, we introduce operators $f_{21}^{\str},\bar{f}_{21},\afff_{21}$, similarly to \Cref{defe21}, where $f_{21}^{\str}(T)$ is defined, for $T\in \calA(\lam)$ with $\str_2(T)=(a,b,c,d)$ as 
    \[f_{21}^{\str}(T)=\begin{cases}(a+1,b+1,c,d)& \text{if }b<\lam_1-2d+2c\\
    (a,b,c+1,d+1)& \text{if }d=0\text{ or }c=\lam_2+d\\
    (a-1,b+2,c,d-1)&\text{if }d=1\text{ and }c<\lam_2+d\end{cases}\]
    Again, one can verify via \Cref{atom0}, that if $T\in \calA(\lam)$ also $\afff_{21}(T)\in \calA(\lam)$. If $\affeps_{21}(T)\neq 0$, there exists $U\in \calA(\lam)$ with $\wt(U)=\wt(T)+\alpha_{21}$. Then, we have  $f_{21}(U)=T$, from which it follows that $e_{21}(T)=U\neq 0$, or $\afff_{21}(U)=0$. But we cannot have $\afff_{21}(U)=0$ and $\langle\wt(U),\alpha_{21}^\vee\rangle \geq 2$. For example, if $c=\lam_2+d$ or $d=0$, then $\bar{f}_{21}(T)=0$ only if $a=\lam_2+b-2c+2d$, which implies $\langle\wt(U),\alpha_{21}^\vee\rangle =\wt(U)_1+2\wt(U)_2=-b\leq 0$.
\end{proof}
The proposition implies that $\affeps_{21}$ is associated to the operator $\affe_{21}$. That is, we have $\affeps_{21}(T)=\max\{ k \mid \affe_{21}^k(T)\neq 0\}$.
Similar expressions for $\affe_{21}$ on the other atoms can be obtained recursively using the embeddings $\Phi$ and $\bPsi$.

We believe one can construct similar charge statistics in higher ranks.
\begin{conjecture}
Assume $\calB$ is a crystal of type $C_3$. Then there exists a function $\affeps_{32}:\calB\raw \bbZ_{\geq 0}$ such that
\[
\begin{aligned} c(T)=\eps_1(T)+\eps_2(T)+\eps_2(s_1(T))+\eps_3(T)+\eps_3(s_2(T))+\eps_3(s_1s_2(T))\\+\affeps_{32}(T)+\affeps_{32}(s_1(T))+\affeps_{32}(s_2s_1(T))\end{aligned}\]
    is a charge statistic on $\calB_+(\lam)$.
\end{conjecture}
Notice that if $\wt(T)=0$ the conjecture predicts that $c(T)=\eps_1(T)+2\eps_2(T)+3\eps_3(T)+3\affeps_{321}(T)$. We have checked in many examples that such a function exists on elements of weight $0$.

\appendix
\section{Proof of \texorpdfstring{\Cref{atomicnumber}}{the preatomic formula} with SageMath}
\label{appendix}
\begin{python}
R.<a,b,c,d,L1,L2>=PolynomialRing(QQ) 
#L1 and L2 represent lam_1 and lam_2
K=R.fraction_field()

def theta12(): 
    X = [a,b,c,d]
    X[0] = 1/K(1/d + b/c + a/b)
    X[1] = 1/K(1/c + b^2/(a*c^2) + 1/(a*d^2))
    X[2] = K(b+b^2*d/c+a*d)
    X[3] = K(a+b^2/c+c/d^2)
    F(a,b,c,d) = tuple(X)
    return F

def theta21():
    X = [a,b,c,d]
    X[0] = 1/K(1/d+b/c^2+a^2/b)
    X[1] = 1/K(1/c+1/(a*d)+b/(a*c^2))
    X[2] = K(b+b^2*d/c^2+a^2*d)
    X[3] = K(a+c/d+b/c)
    F(a,b,c,d) = tuple(X)
    return F

def RRTAux(P):  
#From a tropical polynomials we can remove coefficients bigger than 1.
#Moreover, we are only interested in the function on positive values of a,b,c,d,L1 and L2
#we can remove monomials which are divisible by another monomial, as the minimum is never 
#expressed exclusively by them.
    M = P.monomials()
    R = []
    for i in range(len(M)):
        for j in range(len(M)):
            if M[i].divides(M[j]) and i != j: 
                R.append(j)
    return sum([M[j] for j in range(len(M)) if not j in R])

def RemoveRedundantTerms(X):
    return RRTAux(X.numerator())/RRTAux(X.denominator())

t12 = theta12()
t21 = theta21()
s1(a,b,c,d) = (L1*b^2*d^2/(a*c^2),b,c,d)
phi2(a,b,c,d) = L2*b*d/(a*c^2)
Phi2= K(L2*b*d/(a*c^2))
phi1aux(a,b,c,d) = L1*b^2*d^2/(a*c^2)
Phi1 = K(phi1aux(*t21))
phi12aux1 = s1(*t21)
phi12aux2 = t12(*phi12aux1)
Phi12 = K(phi2(*phi12aux2))
Z = Phi2*Phi1*Phi12
RHS = K(L1^2*L2^2/(b*d*(1+L1*a*c/(b*d)+b*d/(a*c))))
Q = Z/RHS
\end{python}
We first compute the quotient $Q$ on the subset of elements in $\calP(\lam)$ such that $d=0$. 
\begin{python}
f1(a,b,c,d,L1,L2) = (a,b,c,1,L1,L2)
Q1 = RemoveRedundantTerms(K(Q(*f1)(a,b,c,d,L1,L2))) 
# Q(*f1) denotes composition of functions in Sage
print(Q1)
print(Q1.numerator()-Q1.denominator())

\end{python}
\begin{verbatim}
(a^3*c^3*L1 + a^2*c^4*L1 + a^2*b*c^2 + a*b*c^3 + a*b^2*c +
b^2*c^2 + b^3)/(a^3*c^3*L1 + a^2*c^4*L1 + a*c^5*L1 + 
a^2*b*c^2 + a*b*c^3 + a*b^2*c*L1 + b^2*c^2 + b^3)
-a*c^5*L1 - a*b^2*c*L1 + a*b^2*c
\end{verbatim}
There is one extra monomial ($ab^2c$) in the numerator  which does not occur  in the denominator
and two extra monomials ($ac^5\lam_1$ and $ab^2c\lam_1)$ in the denominator.
However, we have 
\begin{itemize}
\item $a + 2b + c + \lam_1 \geq  a + 2b + c \geq  \min(2a+b+2c,3b)$ 
\item
$a + 5c + \lam_1 \geq  a + b + 3c$ (because $b \leq \lam_1 + 2d - 2c)$).
\end{itemize}
Hence the minimum is never expressed by these monomials. So the quotient function $Q$ is constantly zero on the elements of the preatom with $d=0$.

Now we compute the quotient $Q$ on the subset of elements in $\calP(\lam)$ such that $d=\lam_1$.
\begin{python}
f2(a,b,c,d,L1,L2) = (a,b,c,L1,L1,L2)
Q2 = RemoveRedundantTerms(K(Q(*f2)(a,b,c,d,L1,L2)))
print(Q2)
print(Q2.numerator()-Q2.denominator())
\end{python}
\begin{verbatim}
(a^3*c^3*L1^2 + a^2*c^4*L1 + a^2*b*c^2*L1^2 + a*b*c^3*L1 + 
a*b^2*c*L1^2 + b^2*c^2*L1^2 + b^3*L1^3)/(a^3*c^3*L1^2 +
a^2*c^4*L1 + a^2*b*c^2*L1^2+ a*c^5 + a*b*c^3*L1 +
a*b^2*c*L1^2 +b^2*c^2*L1^2 + b^3*L1^3)
-a*c^5
\end{verbatim}
There is one extra monomial in the denominator: $ac^5$.
However, we have 
$a + 5c  \geq a + 2b+ c + 2\lam_1$  (because $b \leq 2c-\lam_1)$).
Hence the minimum is never expressed by this monomial and the tropical function $Q$ is constantly zero when $d=
lam_1$. Finally, we compute $Q$ for $b=\lam_1+2c-2d$.
\begin{python}
f3(a,b,c,d,L1,L2) = (a,L1*c^2/d^2,c,d,L1,L2)
Q3 =RemoveRedundantTerms(K(Q(*f3)(a,b,c,d,L1,L2)))
print(Q3)
print(Q3.numerator()-X.denominator())
\end{python}
\begin{verbatim}
(a^3*d^4 + a^2*c*d^3 + a^2*c*d^2*L1 + a*c^2*d^2+ a*c^2*d*L1 + 
c^3*d*L1 + c^3*L1^2)/(a^3*d^4 +a^2*c*d^3 + a^2*c*d^2*L1 + 
a*c^2*d^2 + a*c^2*d*L1+ c^3*d*L1 + a*c^2*L1^2 + c^3*L1^2)
-a*c^2*L1^2
\end{verbatim}
There is one extra monomial in the denominator: $ac^2\lam_1^2$.
However, we have 
$a + 2c +2\lam_1\geq  a + 2c + d+ \lam_1$  (because $d \leq  \lam_1$). This shows again that $Q$ is zero when $b=\lam_1+2d-2c$. Hence it is always zero on the preatom, concluding the proof of \Cref{atomicnumber} in the case $\pat(T)=0$.

\section{Combinatorial interpretation of the map \texorpdfstring{ \\ $\Psi: \mathcal{P}(\lambda) \rightarrow \mathcal{P}(\lambda + \varpi_{2})$}{the atomic embedding map}}\label{sec:psitableaux}

In this subsection we give a combinatorial description of the map $\Psi$ on Kashiwara--Nakashima tableaux. The reader interested in the thought process behind the present work might want to know that this was the first map $\Psi$ we obtained. Although it is not used in the rest of the paper, we include it for completeness. It may also be used to compute examples by hand directly. 

\begin{proposition}
Let $T \in \mathcal{P}(\lambda) \subset \calB(\lam)$. Then the Kashiwara--Nakashima tableau of $\Psi(T)$ can be obtained from the KN tableau corresponding to $T$ by the following algorithm.
Let $d'$ be the number of $\bar{1}$'s in the first row of $T$.

\begin{itemize}
\item Add the column $\Skew(0:\hbox{\tiny{$1$}} |0: \hbox{\tiny{$2$}} )$ on the left tableau.
\item If $d' = 0$ or $d'=\lam_1$ then
\begin{itemize}
    \item If $T$ has a column of the form $\Skew(0:\hbox{\tiny{$2$}} |0: \hbox{\tiny{$\bar 2$}} )$ (note that $T$ can have at most one such column), then replace this column by $\Skew(0:\hbox{\tiny{$ 2$}} |0: \hbox{\tiny{$\bar 1$}} )$
    \item If $T$ does not have column of the form $\Skew(0:\hbox{\tiny{$2$}} |0: \hbox{\tiny{$\bar 2$}} )$,  then replace the rightmost $1$ in the first row by a $2$.
    \item Replace the rightmost $2$ in the second row by $\bar 2$.
\end{itemize}
\item If $0<d'<\lam_1$ then replace the rightmost $2$ in the first row by a $\bar{1}$. If there is no $2$ in the first row, then replace the rightmost $\bar{2}$ in the first row by $\bar 1$ and the rightmost $2$ in the second row by $\bar 2$.
\end{itemize}

\end{proposition}

\begin{proof}
Let $v_{\lambda}$ and $v_{\lambda + \varpi_{2}}$ be the highest weight tableaux of shapes $\lambda$ and $\lambda + \varpi_{2}$, respectively. 
Let $\stn(T)=(a,b,c,d)$.
Note that, since $b\geq c \geq d$, it follows from the definition of the crystal operators that the letter $\bar 1$ appears in the first row of $T$ precisely $d$ times, so $d' = d$. 

Assume first that $d \in \left\{ 0,\lambda_{1}\right\}$ and recall that $\stn(\Psi(T))=(a,b+1,c+1,d)$. One obtains the above described combinatorial interpretation by making the following observations.
 
 \begin{itemize}
     \item The element $v_{\lambda + \varpi_{2}}$ is obtained from $v_{\lambda}$ by adding to its left the column  $\Skew(0:\hbox{\tiny{$1$}} |0: \hbox{\tiny{$2$}} )$.
     \item When the operator $f_{1}$ is first applied $d$ times to $v_{\lambda}$, it transforms the $d$ leftmost $1$'s into $2's$. Similarly, $f_1^d(v_{\lam+\varpi_2})$ is obtained by replacing the $d$ leftmost $1$ in $v_{\lam+\varpi_2}$ by $2$.
     \item To obtain $f_2^cf_1^d(v_\lam)$, the $2$ which have just been transformed, are replaced again by $\bar 2$. Since $c\geq d$, we must also transform $c-d$ $2$ in the second row into $\bar 2$. Similarly,
     $f_2^{c+1}f_1^d(v_{\lam+\varpi_2})$ is obtained from $f_2^cf_1^d(v_\lam)$ by adding $\Skew(0:\hbox{\tiny{$1$}} |0: \hbox{\tiny{$2$}} )$ on the left and replacing one extra $2$ in the second row by $\bar{2}$.

     \item To obtain $f_1^bf_2^cf_1^d(v_\lam)$ from $f_2^cf_1^d(v_\lam)$ we  transform all the $\bar 2$'s in the first row (if any) into $\bar 1$'s. Then, the rightmost $b-d$ boxes in the columns of size two are transformed according to the crystal rule: $1 \mapsto 2 $ as well as $\bar 2 \mapsto \bar 1$. Therefore if in $T$ there is a column of the form $\Skew(0:\hbox{\tiny{$2$}} |0: \hbox{\tiny{$\bar 2$}} )$, it means that $f_{1}$ modified the top box of this column last, because $f_{2}$ would never modify this column further as its $2$-signature is $+-$. Therefore, if we would apply it one more time, it would modify also the bottom box into $\bar 1$. If there is no such column in $T$, it must therefore mean that the operator $f^{b}_{1}$ either finished at the bottom of one of these columns or was not applied at all to the columns of size two. Applying this operator once more to $f^{b}_{1}f^{c+1}_{2}f^{d}_{1}v_{\lambda + \varpi_{2}}$ would then transform the rightmost $1$ in the first column  into $\bar 2$.
     \item Finally, note that applying the operator $f_{2}^a$ to   $f^{b+1}_{1}f^{c+1}_{2}f^{d}_{1}v_{\lambda + \varpi_{2}}$ or to $f^{b}_{1}f^{c}_{2}f^{d}_{1}v_{\lambda + \varpi_{2}}$ only transforms the letters $2$ into a $\bar 2$ (with the one exception of $2, \bar 2$ being in the same column already described above), hence the changes made by $f^{b}_{1}, f^{b+1}_{1}$ are not modified by $f_{2}^{a}$. 
 \end{itemize}
 
 This finishes the proof in case $d \in \left\{ 0,\lambda_{1} \right\}$. Now assume that $0< d < \lambda_{1}$. From the description above we conclude that $f^{2}_{a}f^{b}_{1}f^{c+1}_{2}f^{d+1}_{1}v_{\lambda + \varpi_{2}}$ is obtained from $f^{c}_{2}f^{d}_{1}v_{\lambda }$ by adding a highest weight column of shape $\varpi_{2}$ at the beginning of $f^{d}v_{\lambda }$  and by replacing the rightmost letter $2$ in the first row, if it exists (note that this is the case if and only if $a<b$), by $\bar 1$, and re-ordering the letters of that row.  \\
 
 If there is no $2$ in the first row, then we claim that the first row must contain at least one letter $\bar 2$. This follows from \Cref{preatominequalities} since in this case we have $b = \lambda_{1}-2d+2c \geq \lambda_{1} > d$. Then from the description above we deduce that $f^{2}_{2}f^{b}_{1}f^{c+1}_{2}f^{d+1}_{1}v_{\lambda + \varpi_{2}}$ is obtained from $f^{c}_{2}f^{d}_{1}v_{\lambda }$ by adding a highest weight column of shape $\varpi_{2}$ at the beginning of $f^{d}v_{\lambda }$  and by replacing the rightmost letter $\bar 2$ in the first row, by $\bar 1$, and, since $a \geq b$, by replacing the rightmost $2$ in tho second row by $\bar 2$. This last action comes from the fact that since a new $\bar 1$ was created in the first column, $f_{1}^{b}$ created one $1 \mapsto 2$ move less (which would then become a $2 \mapsto \bar 2$ by $f_{2}^{a}$). Hence necessarily $f^{a}_{2}$ will create one more move of the form $2 \mapsto \bar 2$ in the second row. 
\end{proof}

\end{document}